\documentclass[12pt]{amsart}
\usepackage[osf,sc]{mathpazo}
\usepackage{amssymb}

\usepackage{geometry}
\usepackage{bm}
\usepackage{tensor}
\usepackage{graphicx}
\usepackage{multirow}
\usepackage{tikz-cd}
\usepackage{varwidth}
\usepackage{hyperref}
\hypersetup{
	colorlinks=true, 
	linktoc=all,     
	linkcolor=blue,
	citecolor=red,
	filecolor=black,
	urlcolor=blue	
}
\usepackage{enumerate}
\usepackage[inline]{enumitem}
\makeatletter
\newcommand{\inlineitem}[1][]{%
	\ifnum\enit@type=\tw@
	{\descriptionlabel{#1}}
	\hspace{\labelsep}%
	\else
	\ifnum\enit@type=\z@
	\refstepcounter{\@listctr}\fi
	\quad\@itemlabel\hspace{\labelsep}%
	\fi} \makeatother
\newcommand{\ga}{\alpha}
\newcommand{\gb}{\beta}
\newcommand{\gga}{\gamma}
\newcommand{\gd}{\delta}
\newcommand{\gep}{\epsilon}

\newcommand{\gth}{\theta}

\newcommand{\gk}{\kappa}
\newcommand{\gl}{\lambda}
\newcommand{\gm}{\mu}
\newcommand{\gn}{\nu}

\newcommand{\gp}{\pi}

\newcommand{\gr}{\rho}
\newcommand{\gs}{\sigma}

\newcommand{\gt}{\tau}

\newcommand{\gf}{\phi}

\newcommand{\gc}{\psi}
\newcommand{\gch}{\chi}
\newcommand{\gom}{\omega}

\newcommand{\Gd}{\Delta}

\newcommand{\Gl}{\Lambda}

\newcommand{\Gs}{\Sigma}

\newcommand{\Gf}{\Phi}

\newcommand{\Gom}{\Omega}

\newcommand{\ugl}{\ul{\gl} = (\gl_{1}^{\gr _1}>\gl_{2}^{\gr _2}>\gl_{3}^{\gr _3}>\ldots>\gl_{k}^{\gr _k})}

\newcommand{\grpp}{\mcl{A}_{\ul{\gl}}}

\newcommand{\autgp}{\mcl{G}_{\ul{\gl}}}

\newcommand{\autmgp}{\mcl{G}_{\ul{\gm}}}

\newcommand{\subs}{\subset}

\newcommand{\sbnq}{\subsetneq}

\newcommand{\bs}{\backslash}

\newcommand{\nin}{\notin}

\newcommand{\ti}{\tilde}

\newcommand{\mbb}{\mathbb}

\newcommand{\mcl}{\mathcal}

\newcommand{\ul}{\underline}
\newcommand{\ol}{\overline}

\newcommand{\us}{\underset}
\newcommand{\os}{\overset}

\newcommand{\lra}{\longrightarrow}

\newcommand{\Z}{\mbb Z}

\newcommand{\ZZ}[1]{\Z/p^{#1}\Z}

\newcommand{\Ra}{\Rightarrow}
\newcommand{\Llra}{\Longleftrightarrow}

\newcommand{\es}{\emptyset}

\newcommand{\fo}[3]{
	\begingroup
	{\fontsize{#1}{#2}\selectfont {#3}}
	\endgroup
}
\newcommand{\equ}[1]{%
	\begin{equation*}
		#1
	\end{equation*}
}
\newcommand{\equa}[1]{%
	\begin{equation*}
		\begin{aligned}
			#1
		\end{aligned}
	\end{equation*}
}
\newcommand{\equan}[2]{%
	\begin{equation}
		\label{Eq:#1}
		\begin{aligned}
			#2
		\end{aligned}
	\end{equation}
}
\DeclareMathOperator{\Det}{Det}
\DeclareMathOperator{\Hom}{Hom}

\DeclareMathOperator{\Aut}{Aut}
\DeclareMathOperator{\Ker}{Ker}

\DeclareMathOperator{\Diag}{Diag}
\DeclareMathOperator{\Tra}{Trace}

\newcommand{\mattwo}[4]{%
	\begin{pmatrix}
		#1 & #2\\ #3 & #4
	\end{pmatrix}
}

\newcommand{\matthree}[9]{%
	\begin{pmatrix}
		#1 & #2 & #3\\ #4 & #5 & #6\\ #7 & #8 & #9
	\end{pmatrix}
}

\theoremstyle{plain}
\newtheorem{theorem}{Theorem}[section]

\newtheorem{prop}[theorem]{Proposition}

\newtheorem{cor}[theorem]{Corollary}

\newtheorem{claim}[theorem]{Claim}

\makeatletter
\def\namedlabel#1#2{\begingroup
	\def\@currentlabel{#2}%
	\label{#1}\endgroup
}
\makeatother

\theoremstyle{definition}
\newtheorem{defn}[theorem]{Definition}

\theoremstyle{remark}
\newtheorem{remark}[theorem]{Remark}
\newtheorem{example}[theorem]{Example}
\numberwithin{equation}{section}

\begin{document}
	
	\title[A Vanishing Criterion]{On the Vanishing Criterion for the Cohomology Groups of the Automorphism Group of a finite Abelian $p$-Group}
	\author{Chudamani Pranesachar Anil Kumar}
	\address{Department of Mathematics, KREA University Campus, 5655, Central Expressway, Sri City, Andhra Pradesh, 517646, INDIA  \,\, email: {\tt akcp1728@gmail.com}}
	\author{Soham Swadhin Pradhan}
	\address{School of Mathematics, Harish Chandra Research Institute, Chhatnag Road, Jhunsi, Prayagraj, 211019, INDIA  \,\, email: {\tt soham.spradhan@gmail.com}}
	
	\thanks{This work is done while the first author is a Postdoctoral Fellow at Harish-Chandra Research Institute, INDIA and later an Assistant Professor of Mathematics at KREA University, Sri City, Andhra Pradesh, INDIA and the second author is a Post Doctoral Fellow at Harish-Chandra Research Institute, INDIA}
	\date{\sc \today}
	\vspace*{-3cm}
	\begin{abstract}
		For a partition $\ugl$ and its associated finite abelian $p$-group $\grpp=\us{i=1}{\os{k}{\oplus}} (\ZZ {\gl_i})^{\gr_i}$, where $p$ is a prime, we consider two actions of its automorphism group $\autgp=\Aut(\grpp)$ on $\grpp$. The first action is the natural action $g\bullet a=\ ^ga$ for all $g\in\autgp$ and $a\in\grpp$ where the action map is denoted by $\Gl_1=Id_{\autgp}:\autgp\lra \autgp$ and the second action is the trivial action $g\bullet a=a$ for all $g\in\autgp$ and $a\in\grpp$ where the action map is denoted by $\Gl_2:\autgp\lra \{e\}\subs\autgp$ the trivial map. For the natural action $\Gl_1$, we show that the first and second cohomology groups $H_{\Gl_1}^i(\autgp,\grpp),i=1,2$ vanish for any partition $\ul{\gl}$ for an odd prime $p$. For the trivial action $\Gl_2$ we show that, for an odd prime $p$,  the first cohomology group $H_{\Gl_2}^1(\autgp,\grpp)$ and for an odd prime $p\neq 3$, the second cohomology group $H_{\Gl_2}^2(\autgp,\grpp)$ vanish if and only if the difference between two successive parts of the partition $\ul{\gl}$ is at most one. This is done by proving that the $\mod p$ cohomologies $H^i(\autgp,\Z/p\Z),i=1,2$ vanish if and only if the difference between two successive parts of the partition $\ul{\gl}$ is at most one for an odd prime $p\neq 3$. The vanishing of the second cohomology in this context is proved using the extended Hochschild-Serre exact sequence for central extensions. 
	\end{abstract}
	\maketitle
\section{\bf{Introduction}}
The cohomology of finite groups is a very vast subject which has its interactions with other areas such as group theory, represention theory, homological algebra, number theory, K-theory, classifying spaces, group actions, characteristic classes, homotopy theory. It is because of its interactions with the other areas, the subject is of immense interest. The calculation of cohomology of finite groups is quite challenging as it involves a number of complicated ingredients. An article~\cite{MR2355780} by A.~Adem mentions in Section 4.3, a calculational method for the computation of cohomology of finite groups. The calculation method mentioned here is also elaborated in the survey article~\cite{MR1460209} by A.~Adem. The details of this method are mentioned in Section~\ref{sec:CMC}.

Among the finite groups, the class of symmetric groups and the class of general linear groups are particularly important and interesting as far as the cohomology computation is concerned. D.~Quillen~\cite{MR0315016} has computed the mod $l$ cohomology ring of $GL_n(\Z/p\Z)$ for a prime $l\neq p$ and gives partial results when $l=p$. Later B.~M.~Mann~\cite{MR0500961} has computed the mod $p$ cohomology ring of symmetric groups for an odd prime $p$. However, nothing much is known about the mod $p$ cohomology of the automorphism group $\autgp$ of a finite abelian $p$ group $\grpp$ except when the abelian $p$-group $\grpp$ is elementary. Even when $\ul{\gl}=(1^n)$ only a vanishing range for the cohomologies is known (see D.~Quillen~\cite{MR2355780}, E.~M.~Friedlander and B.~J.~Parshall~\cite{MR0722727}). So obtaining a vanishing criterion for $H^1_{Trivial\ Action}(\autgp,\Z/p\Z)$ or $H^2_{Trivial\ Action}(\autgp,\Z/p\Z)$ itself is a new result in this direction.

\subsection{\bf{Motivation}}
There is a general principle that when cohomology of groups is used to classify obstructions to constructions then $H^2$ classifies the isomorphism classes of structures up to suitable equivalence and $H^1$ acts simply transitively on the set of equivalence classes of automorphisms of a given structure. When we have vanishing theorems for $H^1$ which occurs in most important situations then structures being studied do not have ``nontrivial" automorphisms. When we have vanishing theorems for $H^2$ then there are no ``nontrivial" obstructions to constructions. In this article, as mentioned in the abstract, we study $H^i_{\Gl_j}(\autgp,\grpp)$ for $i,j=1,2$ and give vanishing results in the case of $\Gl_1$ for odd primes and give vanishing criteria in the case of $\Gl_2$ for odd primes $p\neq 3$. Here $\Gl_1$ is the natural action of $\autgp$ on $\grpp$ and $\Gl_2$ is the trivial action of $\autgp$ on $\grpp$. Henceforth we will specify the action more explicitly instead of using the notation $\Gl_1,\Gl_2$ to avoid any confusion.

Since this article concerns finite abelian p-groups and their automorphism groups which are characterized by partitions, the study of $H^1$ and $H^2$ is naturally of combinatorial interest. The vanishing criteria is expressed in terms of the combinatorics of the partitions.

In this article, for the computation of $H^i_{Trivial\ Action}(\autgp,\grpp),i=1,2$, naturally the mod $p$ cohomologies $H^i_{Trivial\ Action}(\autgp,\Z/p\Z),i=1,2$ play an important role. Since automorphism group of a finite abelian p-group is a generalization of a finite general linear group over $\Z/p\Z$, the mod $p$ cohomologies\linebreak $H^i_{Trivial\ Action}(GL_n(\Z/p\Z),\Z/p\Z),i=1,2$ of the finite general linear groups $GL_n(\Z/p\Z),n\geq 1$ is of special interest.

D.~Quillen in his article~\cite{MR0315016} gives a vanishing range for the mod $p$ cohomology of the finite general linear group and the range was further improved in the article~\cite{MR0722727} by E.~M.~Friedlander and B.~J.~Parshall. There is a landmark result due to D.~Quillen~\cite{MR0298694} which tells us that the $p$-elementary abelian subgroups can be used to understand most of the mod $p$ cohomology of a group. This is also mentioned in the survey article~\cite{MR1460209} by A.~Adem.

\subsection{\bf{Ideas of Main Results and Brief Summary of the Article}}  
The sections in this article are organized as follows. Section~\ref{sec:Prelim} contains the preliminaries required to understand most of the results of the article.  Sections~\ref{sec:H1NaturalAction},~\ref{sec:H2NaturalAction} give  vanishing results for $H^i_{Natural\ Action}(\autgp,\grpp),i=1,2$ in Theorems~\ref{theorem:H1NaturalAction},~\ref{theorem:H2NaturalAction} for odd primes respectively. Sections~\ref{sec:H1TrivialAction},~\ref{sec:H2TrivialAction} describe vanishing criteria for $H^i_{Trivial\ Action}(\autgp,\grpp),i=1,2$ in terms of partitions in Theorems~\ref{theorem:CriterionH1TrivialAction},~\ref{theorem:CriterionH2TrivialAction} respectively. The vanishing and nonvanishing results in the case $H^i_{Trivial\ Action}(\autgp,\grpp),i=1,2$ can be observed from the mod $p$ cohomologies 
$H^i_{Trivial\ Action}(\autgp,\Z/p\Z),i=1,2$.

The proofs of Theorems~\ref{theorem:H1NaturalAction},~\ref{theorem:H2NaturalAction} are not long, though not obvious and follows with some tricky arguments. The nonvanishing cases of $H^i_{Trivial\ Action}(\autgp,\grpp)$ follow somewhat easily though they are not straightforward. The vanishing case of $H^1_{Trivial\ Action}(\autgp,\grpp)$ given in Theorem~\ref{theorem:VanishingH1}  and the vanishing case of $H^2_{Trivial\ Action}(\autgp,\grpp)$ given in Theorem~\ref{theorem:VanishingH2} are involved.  Much of the paper is devoted to proving these two theorems. So we will mention the ideas involved in proving them.

The proof of Theorem~\ref{theorem:VanishingH1} requires a computation of the commutator subgroup of $\autgp$ in the case where the parts of the partition $\ul{\gl}$ differ by at most one. The proof of Theorem~\ref{theorem:VanishingH2} is a long one which occupies most of Section~\ref{sec:H2TrivialAction}. We mention briefly the method of the proof.

To give a vanishing range for $i\in \mbb{N}\cup\{0\}$ of the mod p cohomologies $H^i_{Trivial\ Action}$ $(GL_n(\Z/p\Z),\Z/p\Z)$ of a general linear group $GL_n(\Z/p\Z)$, D.~Quillen~\cite{MR0315016} considers the group of unipotent upper triangular matrices $U_n(\Z/p\Z)\subs GL_n(\Z/p\Z)$ and shows that the cohomology groups $H^i_{Trivial\ Action}(U_n(\Z/p\Z),\Z/p\Z)$ which are semisimple representations of the diagonal subgroup $T_n\subs GL_n(\Z/p\Z)$ has no nontrivial invariants, that is, $H^i_{Trivial\ Action}(U_n(\Z/p\Z),\Z/p\Z)^{T_n}=0$ for $i$ in a certain vanishing range. For this he uses the Poincar\'{e} series of $H^*_{Trivial\ Action}$ $(U_n(\Z/p\Z),\Z/p\Z)$ as a representation of $T_n$ and estimates the series using the Hochschild-Serre spectral sequences for central extensions. The central extensions are obtained from a chief series of the $p$-group $U_n(\Z/p\Z)$ which is nilpotent. He shows that the trivial character does not occur in the character of the representation $H^i_{Trivial\ Action}(U_n(\Z/p\Z),\Z/p\Z)$ for $i$ in a certain vanishing range.

For proving $H^2_{Trivial\ Action}(\autgp,\Z/p\Z)=0$ when the partition $\ul{\gl}$ has the property that any two consecutive parts of the partition differ by at most one, first we consider a suitable $p$-Sylow subgroup $\mcl{P}_{\ul{\gl}}\subs \autgp$ and observe that $H^2_{Trivial\ Action}(\mcl{P}_{\ul{\gl}},\Z/p\Z)$ is a semisimple representation for the restricted diagonal subgroup ${\mcl{D}_{\ul{\gl}}}\subs \autgp$ consisting of diagonal matrices whose orders divide $p-1$. Then we prove that there are no nontrivial $\mcl{D}_{\ul{\gl}}$ invariants in $H^2_{Trivial\ Action}(\mcl{P}_{\ul{\gl}},\Z/p\Z)$, that is, $H^2_{Trivial\ Action}(\mcl{P}_{\ul{\gl}},\Z/p\Z)^{\mcl{D}_{\ul{\gl}}}=0$.
Here instead of using the method of Poincar\'{e} series and estimating the series with Hochschild-Serre spectral sequence as given in D.~Quillen~\cite{MR0315016}, we use the extended Hochschild-Serre exact sequence for central extensions (a result due to N.~Iwahori and H.~Matsumoto~\cite{MR0180607}, Proposition $1.1$, Page 132) repeatedly. The central extensions here are obtained from a chief series $\mcl{N}^s_{\ul{\gl}}\trianglelefteq \mcl{P}_{\ul{\gl}}$ for $s\in S$ a totally ordered set. Then we compute the $\mcl{D}_{\ul{\gl}}$ invariants $H^2_{Trivial\ Action}(\frac{\mcl{P}_{\ul{\gl}}}{\mcl{N}^s_{\ul{\gl}}},\Z/p\Z)^{\mcl{D}_{\ul{\gl}}}$ using the extended Hochschild-Serre exact sequence and show that there are no nontrivial invariants $H^2_{Trivial\ Action}(\mcl{P}_{\ul{\gl}},\Z/p\Z)^{\mcl{D}_{\ul{\gl}}}$. The chief series computation and calculation of the $\mcl{D}_{\ul{\gl}}$ invariants are combinatorially technical as they involve the partition $\ul{\gl}$. 
\section{\bf{Preliminaries}}
\label{sec:Prelim}
In this section we mention the required preliminaries needed in this article. 

\subsection{\bf{Commutator Subgroup of $GL_n(\ZZ k)$}}
~\\
We begin with a proposition.
\begin{prop}
	\label{prop:ElementaryGeneration}
	Let $p$ be a prime and $n,k$ be two positive integers. The group $SL_n(\ZZ k)$ is generated by elementary matrices.
\end{prop}
\begin{proof}
	Clearly all elementary matrices have determinant $1$. 
	Let $A\in SL_n(\ZZ k)$. There exists an entry in the first column of $A$ which is a unit in the ring $\ZZ k$. By using an elementary matrix bring this entry to $11^{th}$ position. Now again using elementary matrices clear all the entries in the first row and first column and make them zero except the $11^{th}$ entry which is a unit. Now continue this process to the remaining submatrix and reduce $A$ to a diagonal matrix $\Diag(a_1,a_2,\cdots,a_n)$ of determinant one. Let $E_{ij}(\ga)=I+e_{ij}(\ga)$ for $1\leq i\neq j\leq n$ where $e_{ij}(\ga)$ is the $n\times n$ matrix with all zero entries except the $ij^{th}$ entry which is $\ga$. Now we observe the following.
	If \equ{B=E_{n\ n-1}(1-a_n^{-1})E_{n-1\ n}(-1)E_{n\ n-1}(1)E_{n\ n-1}(-a_n)E_{n-1\ n}(a_n^{-1})} then  
	$B\Diag(a_1,a_2,\cdots,a_n)=\Diag(a_1,a_2,\cdots,a_{n-1}a_n,1)$. Hence diagonal matrices of determinant one are product of elementary matrices. This completes the proof.	
\end{proof}
\begin{theorem}
\label{theorem:CommutatorSubgroup}
Let $p$ be a prime and $n,k$ be two positive integers. The commutator subgroup of $GL_n(\ZZ k)$ is $SL_n(\ZZ k)$ unless $p=2,n=2$ in which case the commutator subgroup is strictly contained in $SL_2(\Z/2^k\Z)$.
\end{theorem}
\begin{proof}
	Let $E_{ij}(\ga)=I+e_{ij}(\ga)$ for $1\leq i\neq j\leq n$ where $e_{ij}(\ga)$ is the $n\times n$ matrix with all zero entries except the $ij^{th}$ entry which is $\ga$. Then we have for $n\geq 3$ and $r\neq i\neq j\neq r$ we have $[E_{ir}(\ga),E_{rj}(\gb)]=E_{ij}(\ga\gb)$.
	Hence for $n\geq 3$, the commutator subgroup contains $SL_n(\ZZ k)$ using Proposition~\ref{prop:ElementaryGeneration}. Conversely since the group $\frac{GL_n(\ZZ k)}{SL_n(\ZZ k)}\cong (\ZZ k)^*$ is abelian via the determinant homomorphism, we have the commutator subgroup is contained in $SL_n(\ZZ k)$. Hence for $n\geq 3$, we have $SL_n(\ZZ k)=[GL_n(\ZZ k),GL_n(\ZZ k)]$.
	
	Now assume $n=2$. If $p>2$ then there exists $\gb\in \ZZ k$ such that both $\gb,\gb-1$ are invertible. Here we have \equa{E_{12}(\ga)&=\Diag(\gb,1)E_{12}((\gb-1)^{-1}\ga)\Diag(\gb^{-1},1)E_{12}(-(\gb-1)^{-1}\ga)\\&=[\Diag(\gb,1),E_{12}((\gb-1)^{-1}\ga)].}
	So $E_{12}(\ga)$ and similarly $E_{21}(\ga)$ are commutators. Hence again we have  $SL_2(\ZZ k)$ $=[GL_2(\ZZ k),GL_2(\ZZ k)]$.
	
	Now assume $n=2,p=2$.	Suppose $E_{12}(1)$ is in the commutator subgroup of $GL_2(\Z/2^k\Z)$ then reducing modulo $2$ we get that $E_{12}(1)$ is in the commutator subgroup of $GL_2(\Z/2\Z)\cong S_3$. But $E_{12}(1)$ has order $2$ in $GL_2(\Z/2\Z)$. The commutator subgroup $A_3$ of $S_3$ has no element of order $2$ which is a contradiction. Hence $[GL_2(\Z/2^k\Z),GL_2(\Z/2^k\Z)]\neq SL_2(\Z/2^k\Z)$. We infact have that 
	\equ{[GL_2(\Z/2^k\Z),GL_2(\Z/2^k\Z)]\sbnq SL_2(\Z/2^k\Z)}
	since $\frac{GL_2(\Z/2^k\Z)}{SL_2(\Z/2^k\Z)}\cong (\Z/2^k\Z)^*$ is abelian via the deteriminant homomophism.
\end{proof}
\subsection{\bf{A Vanishing Criterion for the second Cohomology}}
~\\
We state the theorem.
\begin{theorem}
	\label{theorem:SecondCohoTriv}
	Let $G$ be a finite group and $p$ be a prime. Then the following are equivalent.
	\begin{enumerate}
		\item $H^2_{Trivial\ Action}(G,A)=0$ for all finite abelian $p$-groups $A$.
		\item $H^2_{Trivial\ Action}(G,\Z/p\Z)=0$.
	\end{enumerate}
\end{theorem}
\begin{proof}
	Clearly $(1)\Ra (2)$. Now we prove $(2)\Ra (1)$. Suppose $H^2_{Trivial\ Action}(G,$ $\Z/p\Z)=0$. We prove $H^2_{Trivial\ Action}(G,A)=0$ by induction on $n$ where the cardinality of $A$ is $p^n$. For $n=1$ the assertion holds. Now assume the assertion holds for $n=m$. Let $\mid A \mid =p^{m+1}$. Let $B$ be an abelian subgroup of index $p$ in $A$. Then any $2$-cocycle $c:G\times G\lra A$ gives rise to a $2$-cocycle $\ol{c}:G\times G\lra \frac{A}{B}\cong \Z/p\Z$. Hence $\ol{c}$ is a coboundary. Let $v:G\lra A$ be such that $\ol{c}(g_1,g_2)=\ol{v}(g_1)+\ol{v}(g_2)-\ol{v}(g_1g_2)$ where $\ol{v}:G\lra \frac{A}{B}$ obtained from $v$. Let $\partial v:G\times G\lra A$ be defined by $\partial v(g_1,g_2)=v(g_1)+v(g_2)-v(g_1g_2)$, a 2-coboundary. Then we have $c-\partial v:G\times G \lra B\subs A$ is $2$-cocycle cohomologous to $c$. But by induction $c-\partial v$ is $2$-coboundary. Hence $c$ is a $2$-couboundary. Therefore $H^2_{Trivial\ Action}(G,A)=0$. This proves $(2)\Ra (1)$. Hence the proposition follows. 
\end{proof}
\subsection{\bf{On the Stable Cohomology Classes}}
~\\
We state the theorem.
\begin{theorem}
\label{theorem:BasicMultiple}
Let $r$ be a positive integer and $p$ be a prime.
Let $c:\Z/p^r\Z\times \Z/p^r\Z\lra \Z/p\Z$ be a $2$-cocycle, that is, $c\in Z^2_{Trivial\ Action}(\Z/p^r\Z,\Z/p\Z)$. Let $\gs\in (\Z/p^r\Z)^*$ such that $\gs^{p-1}=1$. Define another $2$-cocycle $d:\Z/p^r\Z\times \Z/p^r\Z\lra \Z/p\Z$ such that $d(x,y)=c(\gs x,\gs y)$. If $c$ is not cohomologous to zero, that is, $c$ is not a $2$-coboundary then $d$ is cohomologous to $c$ if and only if $\gs =1$. 
\end{theorem}
\begin{proof}
First note that $H^2_{Trivial\ Action}(\ZZ r,\ZZ {})=\ZZ {}$. This follows from the periodic resolution 
\equ{\cdots \os{N}{\lra} \Z[\ZZ r]\os{\gep-1}{\lra} \Z[\ZZ r]\os{N}{\lra} \Z[\ZZ r]\os{\gep-1}{\lra} \Z[\ZZ r]\lra \Z\lra 0,}
where $N=1+\gep+\gep^2+\cdots+\gep^{p^r-1}$ for $\gep$ a generator of $\ZZ r$, that is $\ZZ r=\langle \gep\rangle$.
The group $\ZZ {r+1}$ occurs as an extension in $p-1$ different ways. Let $k_1,k_2\in (\Z/p\Z)^*$. Consider the following diagram where $i^{k_1}(1)=p^rk_1$ and $i^{k_2}(1)=p^rk_2$. The map $\gp$ is reduction modulo $p^r$. 
\[
\begin{tikzcd}
	0 \arrow[r,""]& \Z/p\Z \arrow[hook]{r}{i^{k_1}} \arrow[]{d}{\Vert}[swap]{Id_{\Z/p\Z}} & \ZZ {r+1} \arrow[two heads]{r}{\gp} \arrow[]{d}{\cong}[swap]{\gf} & \Z/p^r\Z \arrow[r,""] \arrow[]{d}{\Vert}[swap]{Id_{\Z/p^r\Z}} & 0\\
	0 \arrow[r,""]& \Z/p\Z\arrow[hook]{r}{i^{k_2}} & \ZZ {r+1}\arrow[two heads]{r}{\gp}& \Z/p^r\Z \arrow[r,""] & 0
\end{tikzcd}
\]
If $\gf$ is an automorphism then there exists $s\in (\ZZ {r+1})^*$ such that $\gf(x)=sx$ for all $x\in \ZZ {r+1}$. Since 
$\gp\circ \gf=\gp$ we have $s\equiv 1\mod p^r$. On the other hand we have $\gf\circ i^{k_1}=i^{k_2}$. Therefore $p^rsk_1\equiv p^rk_2\mod p^{r+1}\Ra sk_1\equiv k_2\mod p \Ra k_1\equiv k_2\mod p$. So we get $p-1$ inequivalent extensions for $k\in (\Z/p\Z)^*$. 

Consider the standard $2$-cocycle $f:\Z/p^r\Z\times \Z/p^r\Z\lra \Z/p\Z$ of the standard extension
\equ{0\lra \Z/p\Z\os{i}{\lra} \ZZ {r+1}\os{\gp}{\us{\mod p^r}{\lra}} \Z/p^r\Z\lra 0}
where $i(1)=p^r,i(a)=p^ra,\gp(\us{i=0}{\os{r}{\sum}}a_ip^i)=\us{i=0}{\os{r-1}{\sum}}a_ip^i$. Here $f(a,b)=\lfloor \frac{a+b}{p^r}\rfloor$ for $a,b\in \{0,1,\cdots,p^r-1\}$ and $0\leq a+b\leq 2p^r-2$. Now $H^2_{Trivial\ Action}(\Z/p^r\Z,\Z/p\Z)$ $\cong \Z/p\Z$ and the $2$-cocycles $f,2f,\cdots (p-1)f$ are mutually non-cohomologous representing all the nonzero cohomology classes of  $H^2_{Trivial\ Action}(\Z/p^r\Z,\Z/p\Z)$.
For $1\leq k\leq p-1$, let $A_{kf}=\Z/p\Z\times \Z/p^r\Z$ as a set. Define a group structure on $A_{kf}$ as follows. For $(a,b),(a',b')\in A_{kf}$ 
\equ{(a,b)+(a',b')=(a+a'+kf(b,b'),b+b').}
Then the following diagram commutes for $k=1$.
\[
\begin{tikzcd}
	0 \arrow[r,""]& \Z/p\Z \arrow[hook]{r}{i} \arrow[]{d}{\Vert}[swap]{Id_{\Z/p\Z}} & \ZZ {r+1} \arrow[two heads]{r}{\gp} \arrow[]{d}{\cong}[swap]{\gf} & \Z/p^r\Z \arrow[r,""] \arrow[]{d}{\Vert}[swap]{Id_{\Z/p^r\Z}} & 0\\
	0 \arrow[r,""]& \Z/p\Z\arrow[hook]{r}{i_1} & A_f\arrow[two heads]{r}{\gp_1}& \Z/p^r\Z \arrow[r,""] & 0
\end{tikzcd}
\]
where $i_1(a)=(a,0),\gf(\us{i=0}{\os{r}{\sum}}a_ip^i)=(a_r,\us{i=0}{\os{r-1}{\sum}}a_ip^i), 0\leq a_0,a_1,\cdots,a_r\leq p-1,\gp_1(a,b)=b$. Let $\gk\in (\ZZ r)^*$ be such that $\gk^{p-1}\equiv 1\mod p^r$. Let $k\in (\ZZ {})^*$ be such that $\gk\equiv k\mod p$. Let $B_{\gk,f}=\Z/p\Z\times \Z/p^r\Z$ as a set. Define a group structure on $B_{\gk,f}$ as follows. For $(a,b),(a',b')\in B_{\gk,f}$ 
\equ{(a,b)+(a',b')=(a+a'+f(\gk b,\gk b'),b+b').}
Let $\ti{\gk}\in (\ZZ {r+1})^*$ be such that $\ti{\gk}\equiv \gk\mod p^r$. 

Consider the following diagram
\[
\begin{tikzcd}
	0 \arrow[r,""]& \Z/p\Z \arrow[hook]{r}{i_k} \arrow[]{d}{\Vert}[swap]{Id_{\Z/p\Z}} & A_{kf} \arrow[two heads]{r}{\gp_k} \arrow[]{d}{\cong}[swap]{\psi} & \Z/p^r\Z \arrow[r,""] \arrow[]{d}{\Vert}[swap]{Id_{\Z/p^r\Z}} & 0\\
	0 \arrow[r,""]& \Z/p\Z \arrow[hook]{r}{j_k} \arrow[]{d}{\Vert}[swap]{Id_{\Z/p\Z}} & A_f \arrow[two heads]{r}{\gp_1} \arrow[]{d}{\cong}[swap]{\gf^{-1}} & \Z/p^r\Z \arrow[r,""] \arrow[]{d}{\Vert}[swap]{Id_{\Z/p^r\Z}} & 0\\
	0 \arrow[r,""]& \Z/p\Z \arrow[hook]{r}{l_k} \arrow[]{d}{\Vert}[swap]{Id_{\Z/p\Z}} & \ZZ {r+1} \arrow[two heads]{r}{\gp} \arrow[]{d}{\cong}[swap]{\gf_{\ti{\gk}}} & \Z/p^r\Z \arrow[r,""] \arrow[]{d}{\Vert}[swap]{Id_{\Z/p^r\Z}} & 0\\
	0 \arrow[r,""]& \Z/p\Z \arrow[hook]{r}{i} \arrow[]{d}{\Vert}[swap]{Id_{\Z/p\Z}} & \ZZ {r+1} \arrow[two heads]{r}{\gb_{\gk}} \arrow[]{d}{\cong}[swap]{\gf} & \Z/p^r\Z \arrow[r,""] \arrow[]{d}{\Vert}[swap]{Id_{\Z/p^r\Z}} & 0\\
	0 \arrow[r,""]& \Z/p\Z \arrow[hook]{r}{i_1} \arrow[]{d}{\Vert}[swap]{Id_{\Z/p\Z}} & A_f \arrow[two heads]{r}{\gga_{\gk}} \arrow[]{d}{\cong}[swap]{\gm} & \Z/p^r\Z \arrow[r,""] \arrow[]{d}{\Vert}[swap]{Id_{\Z/p^r\Z}} & 0\\
	0 \arrow[r,""]& \Z/p\Z\arrow[hook]{r}{m_k} & B_{\gk,f}\arrow[two heads]{r}{\gd_{\gk}}& \Z/p^r\Z \arrow[r,""] & 0	
\end{tikzcd}
\]
where 
\begin{itemize}
	\item $i_k(a)=(a,0),j_k(a)=(k^{-1}a,0),l_k(a)=p^rk^{-1}a,m_k(a)=(a,0)$,
	\item $\gp_k(a,b)=b,\gb_{\gk}(x)=(\ti{\gk})^{-1}x \mod p^r=\gp\circ \gf^{-1}_{\ti{\gk}}(x),\gga_{\gk}(a,b)=\gk^{-1}b$, $\gd_{\gk}(a,b)=b$,
	\item $\psi(a,b)=(k^{-1}a,b),\gf_{\ti{\gk}}(x)=\ti{\gk}x,\gm(a,b)=(a,\gk^{-1}b)$.
\end{itemize}
The above diagram commutes and all vertical maps are group isomorphisms and all horizontal maps are group homomorphisms. 
The above diagram gives an equivalence of extenstions $A_{kf}$ and $B_{\gk,f}$. So we have the cocycles $kf$ and $f_{\gk}$ are cohomologous where $f_{\gk}(a,b)=f(\gk a,\gk b)$. Therefore the set $\{f_{\gk_1},f_{\gk_2},\cdots, f_{\gk_{p-1}}\}$ represent distinct nontrivial cohomology classes in $H^2_{Trivial\ Action}(\Z/p^r\Z,\Z/p\Z)$ where $\gk_i\equiv i\mod p$. So we have for any $1\leq s,t\leq p-1,\gs,\gt\in (\ZZ r)^*$ such that $\gs\equiv s\mod p$ and $\gt\equiv t\mod p$, the cocycles 
\equ{sf_{\gt},tf_{\gs},stf,f_{\gs\gt}} represent the same cohomology class. Finally note that we have 
\equ{\Z/(p-1)\Z\cong \{\gk\in (\ZZ r)^*\mid \gk^{p-1}\equiv 1\mod p^r\}\os{\mod p}{\us{\cong}{\lra}} (\Z/p\Z)^* }
is an isomorphism.

We have that $c$ is a $2$-coboundary if and only if $d$ is a $2$-coboundary. 	
Assume that $c$ is not a $2$-coboundary. So there exists $1\leq t\leq p-1$ such that $c=tf$. So we have $d=tf_{\gs}=stf$. Hence 
$c$ and $d$ represent the same class if and only if $\gs=1$. This proves the theorem.
\end{proof}
\subsection{\bf{Calculational Methods and Computations}}
\label{sec:CMC}
~\\
Calculating the cohomology of finite groups can be quite challenging, as it involves a number of complicated ingredients. For a prime $p$, we outline one technique of computing the $\mod p$ cohomology which is used in this article. Here in this section we assume that $G$ is a finite group which acts trivially on $\Z/p\Z$.

Let $\mcl{F}$ denote a family of subgroups of $G$ which satisfies the following two properties.
\begin{enumerate}
	\item  If $H\in \mcl{F}, K\subs H$, then $K\in \mcl{F}$.
	\item If $g\in G,H\in \mcl{F}$ then $gHg^{-1}\in \mcl{F}$.
\end{enumerate}
Then we can define for $i=1,2$

\equa{&\us{H\in \mcl{F}}{\lim}H^i_{Trivial\ Action}(H,\Z/p\Z)=\{(\ga_H)\mid \ga_{K}=res^H_{K}\ga_H, \text{ if }K\subs H,\\ & \ga_{K}=c_g(\ga_H) \text{ if }K=gHg^{-1}\}\subseteq \us{H\in \mcl{F}}{\prod}H^i_{Trivial\ Action}(H,\Z/p\Z),}
where $c_g:H^i_{Trivial\ Action}(H,\Z/p\Z)\lra H^i_{Trivial\ Action}(gHg^{-1},\Z/p\Z)$ is the conjugation map which sends the cohomology class of $c\in Z^i_{Trivial\ Action}(H,\Z/p\Z)$ to the cohomology class of $d\in Z^i_{Trivial\ Action}(gHg^{-1},\Z/p\Z)$ where $d(g_1,\cdots,g_i)=c(g^{-1}g_1g,\cdots,g^{-1}g_ig)$. Here $g_i\in gHg^{-1},g\in G$.
\begin{remark}
	\label{remark:TrivialAction}
Let $G$ be a finite group, $N$ be a normal subgroup of $G$, then $G$ acts on $H^i_{Trivial\ Action}(N,\Z/p\Z)$ for $i=1,2$ via the action map $g\lra c_g$. It is standard result in group cohomology that $c_n$ acts like identity for all $n\in N$. Hence $H^i_{Trivial\ Action}(N,\Z/p\Z)$ becomes a $\frac GN$-module for $i=1,2$.
\end{remark}
\begin{theorem}
	\label{theorem:CartanEilenberg}
	Let $S_p(G)$ denote the family of all $p$-subgroups of $G$. Then for $i=1,2$ the restrictions induce an isomorphism 
	\equ{H^i_{Trivial\ Action}(G,\Z/p\Z)\cong \us{P\in S_p(G)}{\lim}H^i_{Trivial\ Action}(P,\Z/p\Z).}
\end{theorem}
\begin{proof}
	For a proof of this theorem, refer to G.~Karpilovsky~\cite{MR1183469}, Chapter $9$ on Group Cohomology or refer to H.~Cartan, S.~Eilenberg~\cite{MR1731415}.  
\end{proof}
In order to compute $H^2_{Trivial\ Action}(G,\Z/p\Z)$ we need to find out a collection of subgroups $H_i,1\leq i\leq l$ such that the map 
\equ{H^2_{Trivial\ Action}(G,\Z/p\Z)\lra \us{i=1}{\os{l}{\oplus}}H^2_{Trivial\ Action}(H_i,\Z/p\Z)} is injective in which case this collection is said to detect the cohomology. The philosphy used here is:

\begin{center}
	\fbox{\begin{varwidth}{\dimexpr\textwidth-2\fboxsep-2\fboxrule\relax}
			{\it Reduce to the Sylow $p$-subgroup via the Cartan–Eilenberg result, and then 
			combine information about the cohomology of $p$-groups with stability conditions.}
	\end{varwidth}}
\end{center}
For a group $G$ and a subgroup $H\subseteq G$, for $i=1,2$ we say an element $x\in H^i_{Trivial\ Action}(H,\Z/p\Z)$ is {\bf $G$-stable} or simply stable if 
\equ{Res^H_{gHg^{-1}\cap H}(x)=Res^{gHg^{-1}}_{gHg^{-1}\cap H}(c_g(x)) \text{ for all }g\in G.}
\begin{remark}
For a group $G$ and subgroup $H\subseteq G$, even though for $i=1,2$ the space $H^i_{Trivial\ Action}(H,\Z/p\Z)$ need not be a $G$-module, we can define a {\bf $G$-stable submodule} consisting of {\bf $G$-stable elements} of $H^i_{Trivial\ Action}(H,\Z/p\Z)$.
\end{remark}
\begin{theorem}
\label{theorem:StabilityConditions}
Let $G$ be a finite group acting trivially on $\Z/p\Z$ for a prime $p$. Let $P\subseteq G$ be a $p$-Sylow subgroup of $G$. Then for $i=1,2$, the cohomology group $H^i_{Trivial\ Action}(G,\Z/p\Z)$ restricts isomorphically onto the {\bf $G$-stable submodule} of
$H^i_{Trivial\ Action}(P,\Z/p\Z)$. 
\end{theorem}
\begin{proof}
For a proof of the theorem refer to G.~Karpilovsky~\cite{MR1183469}, Chapter $9$, Section $2$, Proposition $2.5(iv)$ on pages $399-400$ and G.~Karpilovsky~\cite{MR1183469}, Chapter $9$, Section $5$, Theorem $5.3$ on page $411$. 
\end{proof}
\subsection{\bf{The extended Hochschild-Serre Exact Sequence for Central Extensions}}
~\\
First we mention a remark similar to Remark~\ref{remark:TrivialAction} but for general actions.
\begin{remark}
Let $G$ be a group and $N$ be a normal subgroup of $G$. Let $A$ be an abelian group on which $G$ acts via the map $g\lra \ ^g\bullet:a\lra\ ^ga$. Then $G$ acts on $H^i(N,A),i=1,2$ via the action map $g\lra c_g:H^i(N,A)\lra H^i(N,A)$ where $c_g$ sends the cohomology class of $c\in Z^i(N,A)$ to the cohomology class of $d\in Z^i(N,A)$ where $d(g_1,\cdots,g_i)=\ ^gc(g^{-1}g_1g,\cdots,g^{-1}g_ig)$. So $H^i(N,A)$ is a $G$-module. In fact $N$ acts trivially on $H^i(N,A)$ which turns $H^i(N,A)$ into a $\frac GN$-module for $i=1,2$.
\end{remark}
The Hochschild-Serre exact sequence is given in the following theorem.
\begin{theorem}
Let $G$ be a group and $N$ be a normal subgroup of $G$. Let $A$ be an abelian group on which $G$ acts. Then the sequence 
\equ{1\lra H^1(\frac GN,A^N)\os{Inf}{\lra} H^1(G,A)\os{Res}{\lra}H^1(N,A)^G\os{Tra}{\lra}H^2(\frac GN,A^N)\os{Inf}{\lra}H^2(G,A)}
is exact where $Inf$ is the inflation map, $Res$ is the restriction map and $Tra$ is the transgression map.
\end{theorem}
\begin{proof}
For a proof of the theorem refer to G.~Karpilovsky~\cite{MR1215935}, Chapter $1$, Section $1$, Theorem $1.12$ on page $16$ or refer to D.~Benson~\cite{MR1156302}, page $110$ on the inflation restriction sequence which is obtained as a consequence of the spectral sequence of a group extension.
\end{proof}
Let $G_1,G_2$ be two groups and $A$ be an abelian group. We define an abelian group $P(G_1,G_2,A)$ as follows:
\equan{BilinearMap}{P(G_1,G_2,A)&=\{f:G_1\times G_2\lra A\mid f(xy,z)=f(x,z)+f(y,z),\\& f(x,zw)=f(x,z)+f(x,w) \text{ for all }x,y\in G_1, z,w\in G_2\}.}
The set $P(G_1,G_2,A)$ is a group under addition induced from that of $A$.
\begin{prop}
	\label{prop:ThetaMap}
Let $G_1,G_2$ be two subgroups of $G$ such that $xy=yx$ for all $x\in G_1,y\in G_2$. Let $A$ be an abelian group on which $G$ acts trivially. For any given $c\in Z^2_{Trivial\ Action}(G,A)$, let $\gb:G_1\times G_2\lra A$ be defined by $\gb(x,y)=c(x,y)-c(y,x)$ for $x\in G_1,y\in G_2$. Then $\gb\in P(G_1,G_2,A)$ and $\gb$ depends only on the cohomology class $[c]\in H^2_{Trivial\ Action}(G,A)$ and the map $\gth:[c]\lra \gb_c=\gb$ from $H^2_{Trivial\ Action}(G,A)$ to $P(G_1,G_2,A)$ is a well defined group homomorphism. 
\end{prop}
\begin{proof}
For a proof of the proposition refer to G.~Karpilovsky~\cite{MR1215935}, Chapter $1$, Section $2$, Lemma $2.2$ on page $19$.
\end{proof}
\begin{theorem}
\label{theorem:ProductCohomology}
Let $G_1,G_2$ be two groups and $A$ be an abelian group on which $G_1,G_2,G_1\times G_2$ act trivially. Then 
\equ{H^2_{Trivial\ Action}(G_1\times G_2,A)\cong H^2_{Trivial\ Action}(G_1,A)\times H^2_{Trivial\ Action}(G_2,A)\times P(G_1,G_2,A).}
The isomorphism being 
\equ{[c]\lra \big(Res_1([c]),Res_2([c]),\gb_{[c]}\big)}
where $Res_i:H^2_{Trivial\ Action}(G_1\times G_2,A)\lra H^2_{Trivial\ Action}(G_i,A)$ are the restriction maps for $i=1,2$ and $\gb_{[c]}(x,y)=c\big((x,e_{G_2}),(e_{G_1},y)\big)-c\big((e_{G_1},y),(x,e_{G_2})\big)$ for all $x\in G_1,y\in G_2$ where $c\in Z^2_{Trivial\ Action}(G_1\times G_2,A)$ is any cocycle representing the cohomology class $[c]\in H^2_{Trivial\ Action}(G_1\times G_2,A)$. (Here $e_{G_i}$ is the identity element of the group $G_i$, i=1,2.) 
\end{theorem}
\begin{proof}
For a proof of the theorem refer to G.~Karpilovsky~\cite{MR1215935}, Chapter $1$, Section $2$, Theorem $2.3$ on page $20$.
\end{proof}
As a consequence of the above theorem we have the following corollary.
\begin{cor}
\label{cor:ProductCohomology}
Let $G_i,1\leq i\leq n$ be groups. Let $A$ be an abelian group on which $G_i$ act trivially for $1\leq i\leq n$. Then 
\equ{H^2_{Trivial\ Action}(\us{i=1}{\os{n}{\prod}}G_i,A)\cong \us{i=1}{\os{n}{\prod}}H^2_{Trivial\ Action}(G_i,A)\times \us{1\leq i<j\leq n}{\prod}P(G_i,G_j,A)}
with the isomorphism being
\equ{[c]\lra \big((Res_i([c]))_{1\leq i\leq n},(\gb_{ij})_{1\leq i<j\leq n}\big)}
where $Res_i:H^2_{Trivial\ Action}(\us{i=1}{\os{n}{\prod}}G_i,A) \lra H^2_{Trivial\ Action}(G_i,A)$ is the restriction map for $1\leq i\leq n$ and $\gb_{ij}:H^2_{Trivial\ Action}(\us{i=1}{\os{n}{\prod}}G_i,A)\lra P(G_i,G_j,A)$ is a map defined as $\gb_{ij}([c])(x,y)=c\big((e_{G_1},\cdots,e_{G_{i-1}},x,e_{G_{i+1}},\cdots,e_{G_n}),(e_{G_1},\cdots,e_{G_{j-1}},y,e_{G_{j+1}},\cdots,e_{G_n})\big)$ $-c\big((e_{G_1},\cdots,e_{G_{j-1}},y,e_{G_{j+1}},\cdots,e_{G_n}),(e_{G_1},\cdots,e_{G_{i-1}},x,e_{G_{i+1}},\cdots,e_{G_n})\big)$ for $x\in G_i$, $y\in G_j$ for $1\leq i<j\leq n$ where $c\in Z^2_{Trivial\ Action}(\us{i=1}{\os{n}{\prod}}G_i,A)$ is any cocycle representing the cohomology class $[c]$. (Here $e_{G_i}$ is the identity element of the group $G_i, 1\leq i\leq n$).
\end{cor}
In the following theorem we describe the extended Hochschild-Serre exact sequence for central extensions.
\begin{theorem}
	\label{theorem:HSESequence}
Let $G$ be a finite group and $H\subseteq G$ be a central subgroup, that is, $H\subseteq Z(G)$. Let $A$ be a finite abelian group on which $G$ acts trivially. Then the following sequence 
\equa{0&\lra \Hom(\frac{G}{H},A)\os{Inf}{\lra}\Hom(G,A)\os{Res}{\lra}\Hom(H,A)\os{Tra}{\lra}H^2_{Trivial\ Action}(\frac{G}{H},A)\\&\os{Inf}{\lra}H^2_{Trivial\ Action}(G,A) \os{\gt=Res\times \gth}{\lra}H^2_{Trivial\ Action}(H,A) \times P(G,H,A)}
is exact where $Inf$ is the inflation map, $Res$ is the restriction map, $Tra$ is the transgression map and $\gth:H^2_{Trivial\ Action}(G,A)\lra P(G,H,A)$ is the map defined as $\gth([c])(g,h)=c(g,h)-c(h,g)$ for $g\in G,h\in H$ with $c\in Z^2_{Trivial\ Action}(G,A)$ being any cocycle which represents the cohomology class $[c]$. 
\end{theorem}
\begin{proof}
For a proof of the theorem, refer to G.~Karpilovsky~\cite{MR1215935}, Chapter $1$, Section $1$, Corollary $1.13$ on page $17$ and G.~Karpilovsky~\cite{MR1215935}, Chapter $1$, Section $2$, Theorem $2.8$ on page $23$.
\end{proof}
\subsection{\bf{An Algebraic Identity}}
In this section we prove an algebraic identity which is useful later.
Let $X$ be a variable (indeterminate) such that $X^2=0$.
Let $A,B$ be two $n\times n$ matrices with entries $a_{ij},b_{ij}$ for $n\geq i>j\geq 1$ and with entries $Xa_{ij},Xb_{ij}$ for $1\leq i<j\leq n$ respectively and with diagonal entries being $0$ in both of them, where $a_{ij},b_{ij},1\leq i\neq j\leq n$ are variables. All the variables $a_{ij},b_{ij},1\leq i\neq j\leq n,X$ commute with each other.
Then $A$ and $B$ are nilpotent matrices because $X^2=0\Ra A^{2n}=B^{2n}=0$. Moreover we have $\Tra(A^i)=\Tra(B^i)=0$ for $i=1,i\geq n+1$.  Let $g^a=A+\Diag(1+Xa_{11},1+Xa_{22},\cdots,1+Xa_{nn}),g^b=B+\Diag(1+Xb_{11},1+Xb_{22},\cdots,1+Xb_{nn})$ where the elements $a_{ii},b_{ii}$ appearing in the diagonals of $g^a,g^b$ are also commuting variables which also commute with all the remaining variables.
Then we have $g^c=g^ag^b=(L_A+XU_A+XD_A+I_n)(L_B+XU_B+XD_B+I_n)$ where $L_A,L_B$ are the strictly lower triangular part of $A,B$, and
$XU_A,XU_B$ are the strictly upper triangular part of $A,B$, and $I_n+XD_A,I_n+XD_B$ are the diagonal parts of $g^a,g^b$ respectively.
So $X^2=0\Ra g^c=(L_A+L_B+L_AL_B)+X(L_AU_B+U_AL_B+U_A+U_B)+X(L_AD_B+D_AL_B)+X(D_A+D_B)+I_n$. Now define a matrix 
$C=[c_{ij}]_{1\leq i,j\leq n}$ where $c_{ii}=0$ for $1\leq i\leq n$, $c_{ij}=(L_A+L_B+L_AL_B)_{ij}$ for $n\geq i>j\geq 1$ the strictly lower triangular part of $g^c$ where we have ignored the coefficients of $X$ in these entries, $c_{ij}=X(L_AU_B+U_AL_B+U_A+U_B)_{ij}$ for $1\leq i<j\leq n$ the strictly upper triangular part of $g^c$.  Note that $C$ is also nilpotent, $C^{2n}=0, \Tra(C^i)=0$ for $i=1,i\geq n+1$.
\begin{prop}
	\label{prop:Identity}
With notations as above we have for $n\geq 2$,
\begin{enumerate}
	\item $1+\us{i=2}{\os{n}{\sum}}(-1)^{i-1}\Tra(\frac{N^i}{i})=\Det(I_n+N)$ for $N=A,B,C$.
	\item $\Det(I_n+A)+\Det(I_n+B)=1+\Det(I_n+A)\Det(I_n+B)$.
	\item $\Det(I_n+A)\Det(I_n+B)=\Det(I_n+C)+\Tra(AB)$.
	\item \equ{\Tra(AB)+\us{i=2}{\os{n}{\sum}}(-1)^{i-1}\Tra(\frac{C^i}{i})-\us{i=2}{\os{n}{\sum}}(-1)^{i-1}\Tra(\frac{A^i}{i})-\us{i=2}{\os{n}{\sum}}(-1)^{i-1}\Tra(\frac{B^i}{i})=0.}
\end{enumerate}
\end{prop}
\begin{proof}
We prove $(1)$ for $N=A$. Note $X^2=0$. Hence we have for $i\geq 2$, \equ{\Tra(A^i)=iX\bigg(\us{n\geq j_1>j_2>\ldots>j_i\geq 1}{\sum}a_{j_1j_2}a_{j_2j_3}\cdots a_{j_{i-1}j_i}a_{j_ij_1}\bigg).}
Note that the product $Xa_{j_1j_2}a_{j_2j_3}\cdots a_{j_{i-1}j_i}a_{j_ij_1}$ appears $i$ times in the trace of $A^i$.

Now for $2\leq i\leq n$ let $S^i=\{\gs\in S_n\mid \gs \text{ is an }i\text{-cycle of the form }(j_1j_2\cdots j_i)\text{ where }$ $n\geq j_1>j_2 >\cdots>j_i\geq 1\}$.
Since $X^2=0$ we have \equ{\Det(I_n+A)=1+X\us{i=2}{\os{n}{\sum}}\bigg(\us{\gs=(j_1j_2\cdots j_i)\in S^i}{\sum}sign(\gs)a_{j_1j_2}a_{j_2j_3}\cdots a_{j_{i-1}j_i}a_{j_ij_1}\bigg).}

Now it follows that $1+\us{i=2}{\os{n}{\sum}}(-1)^{i-1}\Tra(\frac{A^i}{i})=\Det(I_n+A)$. Similarly $(1)$ follows for $N=B,C$.

We prove $(2)$. This is clear because we have $(1+Xu)(1+Xv)=1+Xu+Xv$ since $X^2=0$. So $(2)$ follows using $(1)$. 

We prove $(3)$. Notice that for $g^a=A+I_n+X\Diag(a_{11},a_{22},\cdots,a_{nn})$ we have,
since $X^2=0$ \equan{DiagTrace}{\Det(g^a)=\Det(I_n+A)+X(\us{i=1}{\os{n}{\sum}}a_{ii}).}
Here $X(\us{i=1}{\os{n}{\sum}}a_{ii})=\Tra\big(X\Diag(a_{11},a_{22},\cdots,a_{nn})\big)$.
Also we have \equan{DetLower}{\Det(g^a+XL_M)=\Det(g^a)} for any strictly lower triangular matrix $L_M$ as this follows since $X^2=0$.
Now \equ{\Det(I_n+A)\Det(I_n+B)=\Det(I_n+A+B+AB).} 
The matrices $C$ and $I_n+A+B+AB$ differ in diagonal entries by a mulitple of $X$ contributed by the term $AB$ and in the strictly lower triangular entries a multiple of $X$ again contributed by the term $AB$. Hence we get using Equations~\ref{Eq:DiagTrace},~\ref{Eq:DetLower} that 
\equ{\Det(I_n+A)\Det(I_n+B)=\Det(I_n+A+B+AB)=\Det(I_n+C)+\Tra(AB).}
This proves $(3)$.

Now $(4)$ is an immediate consequence of $(1),(2),(3)$. Hence the proposition is proved.
\end{proof}
\begin{theorem}
	\label{theorem:Identity}
Let the matrices $A,B,C$ be as defined before with variables $a_{ij},b_{ij}, 1\leq i\neq j\leq n$.
\equa{\us{1\leq i<j\leq n}{\sum}a_{ij}b_{ji}&+a_{ji}b_{ij}+\us{i=2}{\os{n}{\sum}}(-1)^{i-1}\bigg(\us{n\geq j_1>j_2>\ldots>j_i\geq 1}{\sum}\big(c_{j_1j_2}c_{j_2j_3}\cdots c_{j_{i-1}j_i}c_{j_ij_1}\\&-a_{j_1j_2}a_{j_2j_3}\cdots a_{j_{i-1}j_i}a_{j_ij_1}-b_{j_1j_2}b_{j_2j_3}\cdots b_{j_{i-1}j_i}b_{j_ij_1}\big)\bigg)=0.}
\end{theorem} 
\begin{proof}
The theorem follows from Proposition~\ref{prop:Identity}(4). The identity is a very general identity which holds in the ring 
$\Z[a_{ij},b_{ij}:1\leq i\neq j\leq n]$.
\end{proof}
\begin{remark}
We apply Theorem~\ref{theorem:Identity} in Theorem~\ref{theorem:InflationFromAbelinization} later.
\end{remark}
\section{\bf{Computation of $H^1_{Natural\ Action}(\autgp,\grpp)$}}
\label{sec:H1NaturalAction}
\begin{theorem}
	\label{theorem:H1NaturalAction}
	Let $A$ be an abelian group and $G=\Aut(A)$. Let $q\in \mathbb{Z}$ be such that $qId_A:A\lra A$ is an automorphism. Then $(q-1)$ annihilates $H^1_{Natural\ Action}(G,A)$. As a consequence, for a partition
	$\ugl$ and $\grpp=\us{i=1}{\os{k}{\oplus}}(\Z/p^{\gl_i}\Z)^{\gr_i}$, its associated finite abelian $p$-group where $p$ is a prime with $\autgp=\Aut(\grpp)$ its automorphism group, we obtain the following. 
	\begin{enumerate}[label=(\alph*)]
		\item For $p\neq 2, H_{Natural\ Action}^1(\autgp,\grpp)=0$.
		\item For $p=2,H_{Natural\ Action}^1(\autgp,\grpp)$ is a finite elementary abelian $2$-group. 
	\end{enumerate}
\end{theorem}
\begin{proof}
	Let $f\in Z^1_{Natural\ Action}(G,A)$. Then we have the following.
	\begin{enumerate}
		\item $f(Id_{A})=0$.
		\item $f(g^{-1})=-g^{-1}f(g)$.
		\item $(q-1)f\in B^1_{Natural\ Action}(G,A)$. 
	\end{enumerate}
We prove $(1)$. $f(gh)=gf(h)+f(g)$ for all $g,h\in G$. Hence $f(Id_A)=f(Id_AId_A)=Id_Af(Id_A)+f(Id_A)=2f(Id_A) \Ra f(Id_A)=0$. This proves $(1)$.

We prove $(2)$. $0=f(g^{-1}g)=gf(g^{-1})+f(g)$. So $f(g^{-1})=-g^{-1}f(g)$ for all $g\in G$. This proves $(2)$.

We prove $(3)$. $f(qId_{A})=f(qgg^{-1})=qgf(g^{-1})+gf(qId_{A})+f(g)\Ra (q-1)f(g)=gf(qId_{A})-f(qId_{A})$.
This implies that $(q-1)f\in B^1_{Natural\ Action}(G,A)$. This proves $(3)$.

Now let $q$ be a prime such that $qId_{\grpp}:\grpp\lra \grpp$ is an isomorphism. Here it is enough to choose $q=2$ if $p$ is odd and $q=3$ if $p=2$. Then $(a),(b)$ follow. This proves the theorem.
\end{proof}
\begin{remark}
For more details about the group $H^1_{Natural\ Action}(\autgp,\grpp)$ when $p=2$ and a description of a basis elements of the vector space $H^1_{Natural\ Action}(\autgp,\grpp)$, refer to W.~H.~Mills~\cite{MR0087657}.
\end{remark}
\section{\bf{Computation of $H^2_{Natural\ Action}(\autgp,\grpp)$}}
\label{sec:H2NaturalAction}
\begin{theorem}
	\label{theorem:H2NaturalAction}
	Let $A$ be an abelian group and $G=\Aut(A)$. Let $q\in \mathbb{Z}$ be such that $qId_A:A\lra A$ is an automorphism. Then $(q-1)^2$ annihilates $H^2_{Natural\ Action}(G,A)$. As a consequence, for a partition
	$\ugl$ and $\grpp=\us{i=1}{\os{k}{\oplus}}(\Z/p^{\gl_i}\Z)^{\gr_i}$, its associated finite abelian $p$-group where $p$ is a prime with $\autgp=\Aut(\grpp)$ its automorphism group, we obtain the following. 
	\begin{enumerate}[label=(\alph*)]
		\item For $p\neq 2, H_{Natural\ Action}^2(\autgp,\grpp)=0$.
		\item For $p=2,H_{Natural\ Action}^2(\autgp,\grpp)$ is a direct sum of finitely many copies of $\Z/2\Z$ and finitely many copies of $\Z/4\Z$. 
	\end{enumerate}
\end{theorem}
\begin{proof}
	We have $H_{Natural\ Action}^2(G,A)=\frac{Z_{Natural\ Action}^2(G,A)}{B_{Natural\ Action}^2(G,A)}$ where $Z_{Natural\ Action}^2(G,A)=\{u:G\times G\lra A\mid xu(y,z)+u(x,yz)=u(x,y)+u(xy,z)\}$,  $B_{Natural\ Action}^2(G,A)=\{u:G\times G\lra A\mid u(x,y)=xv(y)+v(x)-v(xy)\text{ for some }v:G\lra A\}$.
	Clearly if  $u(x,y)=xv(y)+v(x)-v(xy)\text{ for some }v:G\lra A$ then we have $xu(y,z)+u(x,yz)=u(x,y)+u(xy,z)$. Conversely let \equan{Six}{xu(y,z)+u(x,yz)=u(x,y)+u(xy,z)} for some $u:G\times G \lra A$. Then define $v:G\lra A$ as 
	\equ{v(x)=u(qId_A,x)-u(x,qId_A).}
	We have by substituting $x=qId_A$ in Equation~\ref{Eq:Six}
	\equan{Seven}{qu(y,z)+u(qId_A,yz)=u(qId_A,y)+u(qy,z).}
	By substituting $z=qId_A$ we have
	\equ{qu(y,qId_A)+u(qId_A,qy)=u(qId_A,y)+u(qy,qId_A)}
	which implies
	\equan{Eight}{(q-1)u(y,qId_A)=v(y)-v(qy)}
	and
	\equan{Nine}{(q-1)u(qId_A,y)=qv(y)-v(qy).}
	Replacing the variable $y$ by the variable $x$ and the variable $z$ by the variable $y$ in Equation~\ref{Eq:Seven} we get
	\equan{Ten}{qu(x,y)+u(qId_A,xy)=u(qId_A,x)+u(qx,y).}
	Using Equations~\ref{Eq:Eight},~\ref{Eq:Nine} we get 
	\equan{Eleven}{q(q-1)u(x,y)+qv(xy)-v(qxy)=qv(x)-v(qx)+(q-1)u(qx,y).}
	Subtituting $y=qId_A$ in Equation~\ref{Eq:Six}
	we get 
	\equ{xu(qId_A,z)+u(x,qz)=u(x,qId_A)+u(qx,z).}
	Now replace variable $z$ in the previous equation by the variable $y$ to get 
	\equ{xu(qId_A,y)+u(x,qy)=u(x,qId_A)+u(qx,y).}
	Using Equations~\ref{Eq:Eight},~\ref{Eq:Nine} we get 
	\equan{Twelve}{qxv(y)-xv(qy)+(q-1)u(x,qy)=v(x)-v(qx)+(q-1)u(qx,y).}
	By substituting $z=qId_A$ in Equation~\ref{Eq:Six} we get 
	\equ{xu(y,qId_A)+u(x,qy)=u(x,y)+u(xy,qId_A).}
	Using Equations~\ref{Eq:Eight},~\ref{Eq:Nine} we get 
	\equan{Thirteen}{xv(y)-xv(qy)+(q-1)u(x,qy)=(q-1)u(x,y)+v(xy)-v(qxy).}
	Now eliminating $(q-1)u(qx,y)$ and $(q-1)u(x,qy)$ from Equations~\ref{Eq:Eleven},~\ref{Eq:Twelve},~\ref{Eq:Thirteen} and solving for $u(x,y)$ we get
	\equ{(q-1)^2u(x,y)=(q-1)xv(y)+(q-1)v(x)-(q-1)v(xy).}
	
	This proves that $(q-1)^2u:G\times G\lra A$ is a 2-coboundary, that is, $(q-1)^2u\in B_{Natural\ Action}^2(G,A)$. Therefore we have 
	that $H_{Natural\ Action}^2(G,A)$ is annihilated by $(q-1)^2$. 
	
	Now let $q$ be a prime such that $qId_{\grpp}:\grpp\lra \grpp$ is an isomorphism. Here it is enough to choose $q=2$ if $p$ is odd and $q=3$ if $p=2$. Then $(a),(b)$ follow. This proves the theorem.
\end{proof}
\begin{remark}
In case $p=2$ in the previous theorem, the number of copies of $\Z/2\Z$ and the number of copies of $\Z/4\Z$ that occurs in $H^2_{Natural\ Action}(\autgp,\grpp)$ is still not known. We provide two examples below illustrating the remark.
\end{remark}
\begin{example}
	Let $A=\Z/4\Z$ and $G=\Z/2\Z$. Then there are at least two possibilities for the extension $E$ of $G$ by $A$ which gives rise to the natural action of $G=\Aut(A)$ on $A$. They are $E=D_8,Q_8$. Hence $H_{Natural\ Action}^2(G,A)\neq 0$.
\end{example}
\begin{example}
	Let $A=(\Z/2\Z)^2$ and $G=GL_2(\Z/2\Z)$. Then any extension $E$ of $G$ by $A$ which gives rise to the natural action of $G=\Aut(A)$ on $A$ is a nonabelian group of order $24$ with trivial center. Hence $E\cong S_4$ and there is only one normal subgroup $V_4$ of $S_4$ which is isomorphic to $(\Z/2\Z)^2$. Any automorphism of $V_4$ extends to an automorphism of $E=S_4$. Also note that all automorphisms of $S_4,S_3$ are inner automorphisms and we have $S_3\cong Inn(S_3)= \Aut(S_3),S_4\cong Inn(S_4)= \Aut(S_4)$. Any exact sequence with $E=S_4$ looks like 
	\equ{0\lra (\Z/2\Z)^2 \os{i\circ \gc}{\hookrightarrow} S_4 \os{\gf\circ \gp}{\twoheadrightarrow}GL_2(\Z/2\Z)\lra  1}
	where $\gc$ is an automorphism $(\Z/2\Z)^2$ and $\gf$ is an automorphism of $GL_2(\Z/2\Z)=S_3$ and $i((\Z/2\Z)^2)=V_4$.
	But the following two extensions are equivalent.
	
	\[
	\begin{tikzcd}
		0 \arrow[r,""]& (\Z/2\Z)^2 \arrow[hook]{r}{i} \arrow[]{d}{\Vert}[swap]{Id_A} & S_4 \arrow[two heads]{r}{\gf\circ \gp} \arrow[]{d}{\cong}[swap]{\ti{\psi}} & GL_2(\Z/2\Z) \arrow[r,""] \arrow[]{d}{\Vert}[swap]{Id_G} & 1\\
		0 \arrow[r,""]& (\Z/2\Z)^2 \arrow[hook]{r}{i\circ \psi} & S_4\arrow[two heads]{r}{\gf\circ \gp}& GL_2(\Z/2\Z) \arrow[r,""] & 1
	\end{tikzcd}
	\]
	where $\ti{\psi}$ is an extension of $\psi$. Let $g\in S_4$ and $\psi_g$ be inner automophism of $S_4$ induced by $g$. If $\gf_1,\gf_2$ are two automorphisms of $GL_2(\Z/2\Z)$ then the following two extensions 
	\[
	\begin{tikzcd}
		0 \arrow[r,""]& (\Z/2\Z)^2 \arrow[hook]{r}{i} \arrow[]{d}{\Vert}[swap]{Id_A} & S_4 \arrow[two heads]{r}{\gf_1\circ \gp} \arrow[]{d}{\cong}[swap]{\psi_g} & GL_2(\Z/2\Z) \arrow[r,""] \arrow[]{d}{\Vert}[swap]{Id_G} & 1\\
		0 \arrow[r,""]& (\Z/2\Z)^2 \arrow[hook]{r}{i} & S_4\arrow[two heads]{r}{\gf_2\circ \gp}& GL_2(\Z/2\Z) \arrow[r,""] & 1
	\end{tikzcd}
	\]
	are equivalent if and only if $g\in V_4=\{Identity,(12)(34),(13)(24),(14)(23)\}\subs S_4$ and $\gf_1=\gf_2$. Among all exact sequences of the form 
	\equ{0\lra (\Z/2\Z)^2 \os{i}{\hookrightarrow} S_4 \os{\gf\circ \gp}{\twoheadrightarrow}GL_2(\Z/2\Z)\lra  1}
	the group $GL_2(\Z/2\Z)$ induces natural action on $(\Z/2\Z)^2$ if and only if $\gf$ is the identity automorphism.
	Hence there is essentially only one way in which $S_4$ appears as an extension for the natural action of $GL_2(\Z/2\Z)$ on $(\Z/2\Z)^2$ which is the semidirect product $(\Z/2\Z)^2 \rtimes GL_2(\Z/2\Z)$.
	So we have $H^2_{Natural\ Action}(G,A)$ is a group of order $1$, that is, $H^2_{Natural\ Action}(G,A)=0$.  
\end{example}
\section{\bf{Computation of $H^1_{Trivial\ Action}(\autgp,\grpp)$}}
\label{sec:H1TrivialAction}
\begin{prop}
	\label{prop:Index}
	Let $\ugl$ be a partition. Let $\grpp$ be the finite abelian $p$-group associated to $\ul{\gl}$, where $p$ is a prime. Let $\autgp=\Aut(\grpp)$ be its automorphism group. Then the following are equivalent. 
	\begin{enumerate}
		\item $H^1_{Trivial\ Action}(\autgp,\grpp)\neq 0$.
		\item $\Hom(\autgp,\grpp)\neq 0$.
		\item $\Hom(\autgp,\Z/p\Z)\neq 0$, that is, there exists a normal subgroup $N\trianglelefteq\autgp$ of index $p$, that is, $[\autgp: N]=p$. 
		\item The order of the abelian group $\frac{\autgp}{[\autgp,\autgp]}$ is divisible by $p$.
	\end{enumerate}
\end{prop}
\begin{proof}
	The proof is immediate.
\end{proof}
\begin{prop}
	\label{prop:pIndex}
~\\
	\begin{enumerate}	
		\item We have $p$ divides $\mid \frac{GL_n(\ZZ k)}{[GL_n(\ZZ k),GL_n(\ZZ k)]}\mid $ if $k\geq 2$ for any prime $p$ for all $n\geq 1$. 
		\item For $k=1$ and $p$ odd we have $p$ does not divide $\mid\frac{GL_n(\ZZ {})}{[GL_n(\ZZ {}),GL_n(\ZZ {})]}\mid$ for all $n\geq 1$.
		\item For $k=1,p=2$, we have $2$ does not divide $\mid\frac{GL_n(\Z/2\Z)}{[GL_n(\Z/2\Z),GL_n(\Z/2\Z)]}\mid$ if and only if $n\geq 3$ or $n=1$.
	\end{enumerate}
\end{prop}
\begin{proof}
	We prove $(1)$. If $k\geq 2$ and either $p\neq 2$ or $n\neq 2$ then the commutator subgroup is $SL_n(\ZZ k)$ using Theorem~\ref{theorem:CommutatorSubgroup}. Hence we have 
	\equ{\mid \frac{GL_n(\ZZ k)}{SL_n(\ZZ k)}\mid =p^{k-1}(p-1)\Ra p \text{ divides }  \vert \frac{GL_n(\ZZ k)}{[GL_n(\ZZ k),GL_n(\ZZ k)]}\vert.} 
	If $k\geq 2, n=2$ and $p=2$ then  $[GL_2(\Z/2^k\Z),GL_2(\Z/2^k\Z)]\sbnq SL_2(\Z/2^k\Z)$ using Theorem~\ref{theorem:CommutatorSubgroup}. Hence we have 
	\equ{2^{k-1}=\mid \frac{GL_n(\ZZ k)}{SL_n(\ZZ k)}\mid \text{ divides } \mid \frac{GL_n(\ZZ k)}{[GL_n(\ZZ k),GL_n(\ZZ k)]}\mid.} Since $k-1\geq 1$ we have 
	$2\text{ divides } \mid \frac{GL_n(\ZZ k)}{[GL_n(\ZZ k),GL_n(\ZZ k)]}\mid$. This proves $(1)$.
	
	We prove $(2)$. For $k=1,p\neq 2$ we have that the commutator subgroup of $GL_n(\Z/p\Z)$ is $SL_n(\Z/p\Z)$ using Theorem~\ref{theorem:CommutatorSubgroup}. Hence $\vert \frac{GL_n(\ZZ {})}{[GL_n(\ZZ {}),GL_n(\ZZ {})]}\vert=p-1$. This proves $(2)$.
	
	Now we prove $(3)$. For $k=1,p=2$ and $n=1$, $GL_1(\Z/2\Z)$ is trivial. For $k=1,p=2,n\geq 3$, the commutator subgroup of $GL_n(\Z/2\Z)$ is $SL_n(\Z/2\Z)=GL_n(\Z/2\Z)$ using Theorem~\ref{theorem:CommutatorSubgroup}. Hence $\frac{GL_n(\Z/2\Z)}{[GL_n(\Z/2\Z),GL_n(\Z/2\Z)]}$ is trivial. For $k=1,n=2,p=2$, we have 
	$\frac{GL_2(\Z/2\Z)}{[GL_2(\Z/2\Z),GL_2(\Z/2\Z)]}\cong \Z/2\Z$. This proves $(3)$.
\end{proof}
\begin{theorem}
	\label{theorem:FirstCohomologyIsotypic}
Let $p$ be a prime and $n,k$ be two positive integers. Then \equ{H^1_{Trivial\ Action}(GL_n(\ZZ k),(\ZZ k)^n)\neq 0}
if and only if $k\geq 2$ or $k=1,p=2,n=2$.
\end{theorem}
\begin{proof}
This follows from Propositions~\ref{prop:Index},~\ref{prop:pIndex}. 
\end{proof}
Now we prove another required proposition.
\begin{prop}
	\label{prop:StrongApprox}
	Let $p$ be an odd prime, $k,n$ be positive integers. Let $M_{n}(\ZZ k)$ be the space of $n\times n$ matrices over the ring $\ZZ k$. Let $K_n(\ZZ k)$ be the kernel of the composed homomorphism $GL_n(\ZZ k)\os{\mod p}{\lra} GL_n(\Z/p\Z) \os{\Det}{\lra} (\Z/p\Z)^*$. If $g\in K_n(\ZZ k)$ then there exists \equ{g'\in SL_n(\ZZ k)=[GL_n(\ZZ k),GL_n(\ZZ k)]\text{ and a matrix }h\in M_{n}(\ZZ k)} such that \equ{g=g'+ph.}
\end{prop}
\begin{proof}
	We need to use the important fact that for $k\geq 2$ the reduction map $\gf:SL_n(\ZZ k)\lra SL_n(\ZZ {k-1})$ is surjective. This can be proved using the fact that both $SL_n(\ZZ k),SL_n(\ZZ {k-1})$ are generated by elementary matrices. Let $g\in K_n(\ZZ k)$ then \equ{g \mod p^{k-1}\in K_n(\ZZ {k-1}).}
	So by induction there exists a $g''\in SL_n(\ZZ {k-1})$ and $h''\in M_n(\ZZ {k-1})$ such that $g\mod p^{k-1}=g''+ph''$.
	Now there exists $g'\in SL_n(\ZZ k)$ such that $\gf(g')=g''$. Hence we have $g \mod p=g'\mod p\Ra g=g'+ph$ for some $h\in M_n(\ZZ k)$.
\end{proof}
\subsubsection{\bf{The Vanishing Case of $H^1_{Trivial\ Action}(\autgp,\grpp)$}}
\begin{theorem}
	\label{theorem:VanishingH1}
Let $\ugl$ be a partition such that $\gl_i=k-i+1$. Let $\grpp$ be the finite abelian $p$-group associated to $\ul{\gl}$, where $p$ is an odd prime. Let $\autgp=\Aut(\grpp)$ be its automorphism group. Then $H^1_{Trivial\ Action}(\autgp,\grpp)=\Hom(\autgp,\grpp)=0$.
\end{theorem}
\begin{proof}
	It suffices to show that $p$ does not divide the order of the group $\frac{\autgp}{[\autgp,\autgp]}$ using Proposition~\ref{prop:Index}.
	We show that the commutator subgroup $\autgp'=[\autgp,\autgp]$ of $\autgp$ is the kernel of the surjective homomorphism 
	\equ{\Gf:\autgp \lra (\Z/p\Z)^*\times (\Z/p\Z)^*\times \cdots \times (\Z/p\Z)^*,} 
	given as follows: If $g=[g_{mn}]_{1\leq m,n\leq k},g_{mn}:(\Z/p^{\gl_n}\Z)^{\gr_n}\lra (\Z/p^{\gl_m}\Z)^{\gr_m}$, 
	\equ{\Gf(g)=(\Det(g_{11})\mod p, \Det(g_{22})\mod p,\cdots, \Det(g_{kk})\mod p).}
	It is clear that $\Ker(\Gf)$ contains the commutator subgroup $\autgp'$. Now we show that $\Ker(\Gf)\subseteq \autgp'$. 
	
	Let $I_{\ul{\gl}}$ be the identity element of $\autgp$.
	First it is clear for $1\leq i< j\leq k$ the block elementary matrix $E_{ij}(p^{\gl_i-\gl_j}A)=I_{\ul{\gl}}+e_{ij}(p^{\gl_i-\gl_j}A)\in \autgp'$ where 
	$e_{ij}(p^{\gl_i-\gl_j}A)$ is a matrix with all zero blocks except for the $ij^{th}$ block entry which is $p^{\gl_i-\gl_j}A$ where 
	$p^{\gl_i-\gl_j}A:(\ZZ {\gl_j})^{\gr_j}\lra (\ZZ {\gl_i})^{\gr_i}$. 
	This is because if $I_l$ denotes the identity matrix in $GL_{\gr_l}(\ZZ {\gl_l})$ for $1\leq l\leq k$ then 
	\equa{E_{ij}(p^{\gl_i-\gl_j}A)=&\Diag(I_1,\cdots,I_{i-1},\gb I_i,I_{i+1},\cdots I_k)E_{ij}((\gb-1)^{-1}p^{\gl_i-\gl_j}A)\\&\Diag(I_1,\cdots,I_{i-1},\gb^{-1} I_i,I_{i+1},\cdots I_k)E_{ij}(-(\gb-1)^{-1}p^{\gl_i-\gl_j}A)\\
	=&[\Diag(I_1,\cdots,I_{i-1},\gb I_i,I_{i+1},\cdots I_k), E_{ij}((\gb-1)^{-1}p^{\gl_i-\gl_j}A)]}
	where $\gb\in (\ZZ {\gl_i})^*$ such that $\gb-1\in (\ZZ {\gl_i})^*$.
	
	It is also clear that for $1\leq j< i\leq k$ the block elementary matrix $E_{ij}(A)=I_{\ul{\gl}}+e_{ij}(A)\in \autgp'$ where 
	$e_{ij}(A)$ is a matrix with all zero blocks except for the $ij^{th}$ block entry which is $A$ where 
	$A:(\ZZ {\gl_j})^{\gr_j}\lra (\ZZ {\gl_i})^{\gr_i}$. This is because 
	\equa{E_{ij}(A)=&\Diag(I_1,\cdots,I_{j-1},\gb^{-1} I_j,I_{j+1},\cdots I_k)E_{ij}((\gb-1)^{-1}A)\\&\Diag(I_1,\cdots,I_{j-1},\gb I_j,I_{j+1},\cdots I_k)E_{ij}(-(\gb-1)^{-1}A)\\
	=&[\Diag(I_1,\cdots,I_{j-1},\gb^{-1} I_j,I_{j+1},\cdots I_k), E_{ij}((\gb-1)^{-1}A)]}
where $\gb\in (\ZZ {\gl_j})^*$ such that $\gb-1\in (\ZZ {\gl_j})^*$.
	
	Now by induction on $k$ we observe that the matrix $g=\Diag(g_{11},g_{22},\cdots,g_{kk})$ $\in \autgp'$ where $g_{11}\in SL_{\gr_1}(\ZZ {\gl_1})$ and $g_{ii}\in K_{\gr_i}(\ZZ {\gl_i})$ for $2\leq i\leq k$. 
	
	Now consider $g=\Diag(g_{11},I_2,I_3,\cdots,I_k)\in \autgp$ where $g_{11}\in K_{\gr_1}(\ZZ {\gl_1})$. We show that $g\in \autgp'$. 
	Using Proposition~\ref{prop:StrongApprox} there exists $g_{11}'\in SL_{\gr_1}(\ZZ {\gl_1})$ and $h\in M_{\gr_1}(\ZZ {\gl_1})$ such that $g_{11}=g_{11}'+ph$. Now the matrix \equ{\Diag(g_{11}',I_2,I_3,\cdots,I_k)\in \autgp'.}
	
	Choose \equ{pB:(\ZZ {\gl_2})^{\gr_2}\lra (\ZZ {\gl_1})^{\gr_1},C:(\ZZ {\gl_1})^{\gr_1}\lra (\ZZ {\gl_2})^{\gr_2}} such that \equ{pBCg_{11}'=pe_{ij}:(\ZZ {\gl_1})^{\gr_1}\lra (\ZZ {\gl_1})^{\gr_1}, 1\leq i,j\leq \gr_1.} Such a choice of matrices $B,C$ clearly exist. Note here $\gl_1-\gl_2=1$.
	
	Let $h_{ij}$ be the $ij^{th}$ element of $h\in M_{\gr_1}(\ZZ {\gl_1})$ for $1\leq i,j\leq \gr_1$.
	We have \equ{E_{21}(-C)E_{12}(ph_{ij}B)\Diag(g_{11}',I_2,I_3,\cdots,I_k)=\mattwo{\mattwo{g_{11}'}{ph_{ij}B}{-Cg_{11}'}{I_2-ph_{ij}CB}}{0}{0}{I_{\ul{\gm}}}}
	where $I_{\ul{\gm}}$ is the identity element of $\autmgp$ with $\ul{\gm}=(\gl_3^{\gr_3}>\cdots>\gl_k^{\gr_k})$.
	We have 
	\equa{&\mattwo{\mattwo{g_{11}'}{ph_{ij}B}{-Cg_{11}'}{I_2-ph_{ij}CB}}{0}{0}{I_{\ul{\gm}}}E_{21}(Cg_{11}')E_{12}(-ph_{ij}(g_{11}')^{-1}(I_1+ph_{ij}BC)^{-1}B)\\
		&=\mattwo{\mattwo{g_{11}'+ph_{ij}BCg_{11}'}{0}{0}{I_2-ph_{ij}CB}}{0}{0}{I_{\ul{\gm}}}.}
	Now the matrix $I_2-ph_{ij}CB\in K_{\gr_2}(\ZZ {\gl_2})$. Hence we get 
	\equa{&\mattwo{\mattwo{g_{11}'+ph_{ij}BCg_{11}'}{0}{0}{I_2-ph_{ij}CB}}{0}{0}{I_{\ul{\gm}}}\Diag(I_1,(I_2-ph_{ij}CB)^{-1},I_3,I_4,\cdots,I_k)\\
		&=\Diag(g_{11}'+ph_{ij}e_{ij},I_2,\cdots,I_k)\in \autgp'.}
	Now applying this procedure repeatedly for all $1\leq i,j\leq \gr_1$ we get that 
	\equ{\Diag(g_{11}'+p\us{i,j}{\sum}h_{ij}e_{ij},I_2,\cdots,I_k)=\Diag(g_{11}'+ph,I_2,\cdots,I_k)=\Diag(g_{11},I_2,\cdots,I_k)}
	belongs to $\autgp'$. Hence we get that for $g_{ii}\in K_{\gr_i}(\ZZ {\gl_i}),1\leq i\leq k$ the matrix 
	\equ{\Diag(g_{11},g_{22},\cdots,g_{kk})=\Diag(g_{11},I_2,\cdots,I_k)\Diag(I_1,g_{22},g_{33},\cdots,g_{kk})\in \autgp'.}
	
	Now if $g\in \Ker(\Gf)$ with diagonal entries $g_{ii}\in K_{\gr_i}(\ZZ {\gl_i})$ then by using elementary block matrices which are in $\autgp'$ we can reduce $g$ to the diagonal block matrix $\Diag(g_{11},g_{22},\cdots,g_{kk})$ which is also in $\autgp'$. Hence we have $g\in \autgp'$. Therefore we have proved 
	\equ{\Ker{\Gf}=\autgp'.}
	Hence the theorem follows.
\end{proof}
\begin{remark}
The above theorem does not hold if $p=2$ as the following example illustrates. Also refer to Theorem~\ref{theorem:FirstCohomologyIsotypic} for another such example where $k=1,n=2,p=2$.
\end{remark}
\begin{example}
	If $k=2,p=2$ and $\ul{\gl}=(2>1)$ then  $\Hom(\autgp,\grpp)\neq 0$. This is because $\autgp$ is a group of order $8$. Hence there exists a normal subgroup of order $4$ with quotient isomorphic to $\Z/2\Z$. 
\end{example}
\subsubsection{\bf{The Nonvanishing Case of $H^1_{Trivial\ Action}(\autgp,\grpp)$}}
\begin{theorem}
	\label{theorem:NonvanishingH1}
	Let $\ugl$ be a partition. Let $\grpp$ be the finite abelian $p$-group associated to $\ul{\gl}$, where $p$ is any prime.
	Let $\autgp=\Aut(\grpp)$ be its automorphism group. If $\gl_i-\gl_{i+1}\geq 2$ for some $1\leq i\leq k-1$ or if $\gl_k\geq 2$ then $H^1_{Trivial\ Action}(\autgp,\grpp)\neq 0$. 
\end{theorem}
\begin{proof}
	Define $\gl_{k+1}=0$. If $g\in \autgp$ then $g=[g_{mn}]_{1\leq m,n\leq k}$ where $g_{mn}:(\ZZ {\gl_n})^{\gr_n}\lra (\ZZ {\gl_m})^{\gr_m}$. Consider the homomorphism \equ{\gf:\autgp \lra GL_{\gr}(\ZZ {2})} where $\gr=\us{j=1}{\os{i}{\sum}}\gr_j$ given by 
	\equ{g=[g_{mn}]_{1\leq m,n\leq k} \lra \ol{g}=[g_{mn}]_{1\leq m,n\leq i}\mod p^2.}
	Since $\gl_m-\gl_n\geq 2$ if $n\geq {i+1}$ and $m\leq i$ the map $\gf$ is a homomorphism.
	Now consider the homomorphism $\gc:GL_{\gr}(\ZZ {2})\lra (\ZZ {2})^{*}=\Z/p(p-1)\Z$ given by the determinant. The composed map 
	$\gc\circ\gf$ maps $\autgp$ onto a cyclic group of order $p(p-1)$ which can further be mapped onto $\Z/p\Z$. Hence using Proposition~\ref{prop:Index} we conclude that $H^1_{Trivial\ Action}(\autgp,\grpp)\neq 0$.
\end{proof}
\subsubsection{\bf{The Vanishing/Nonvanishing Criterion for $H^1_{Trivial\ Action}(\autgp,\grpp)$ for Odd Primes}}
As a consquence of Theorems~\ref{theorem:VanishingH1},~\ref{theorem:NonvanishingH1} we have proved the following theorem for odd primes.
\begin{theorem}
	\label{theorem:CriterionH1TrivialAction}
Let $\ugl$ be a partition. Let $\grpp$ be the finite abelian $p$-group associated to $\ul{\gl}$, where $p$ is an odd prime. Let $\autgp=\Aut(\grpp)$ be its automorphism group. Then $H^1_{Trivial\ Action}(\autgp,\grpp)=0$ if and only if the difference between two successive parts of $\ul{\gl}$ is at most $1$.
\end{theorem}
\section{\bf{Computation of $H^2_{Trivial\ Action}(\autgp,\grpp)$}}
\label{sec:H2TrivialAction}
In this section we compute the second cohomology group for the trivial action of $\autgp$ on $\grpp$.
\subsection{\bf{The Nonvanishing Case of $H^2_{Trivial\ Action}(\autgp,\grpp)$}}
\begin{theorem}
	\label{theorem:NonvanishingH2}
	Let $\ugl$ be a partition. Let $\grpp$ be the finite abelian $p$-group associated to $\ul{\gl}$, where $p$ is a prime.
	Let $\autgp=\Aut(\grpp)$ be its automorphism group. If $\gl_i-\gl_{i+1}\geq 2$ for some $1\leq i\leq k-1$ or if $\gl_k\geq 2$ then $H^2_{Trivial\ Action}(\autgp,\grpp)\neq 0$.
\end{theorem}
\begin{proof}
	Define $\gl_{k+1}=0$. If $g\in \autgp$ then $g=[g_{mn}]_{1\leq m,n\leq k}$ where $g_{mn}:(\ZZ {\gl_n})^{\gr_n}\lra (\ZZ {\gl_m})^{\gr_m}$. Consider the homomorphism \equ{\gf:\autgp \lra GL_{\gr}(\ZZ {2})} where $\gr=\us{j=1}{\os{i}{\sum}}\gr_j$ given by 
	\equ{g=[g_{mn}]_{1\leq m,n\leq k} \lra \ol{g}=[g_{mn}]_{1\leq m,n\leq i}\mod p^2.}
	Since $\gl_m-\gl_n\geq 2$ if $n\geq {i+1}$ and $m\leq i$ the map $\gf$ is a homomorphism.
	Now consider the homomorphism $\gc:GL_{\gr}(\ZZ {2})\lra (\ZZ {2})^{*}=\Z/p(p-1)\Z$ given by the determinant. The composed map 
	$\gc\circ\gf$ maps $\autgp$ onto a cyclic group of order $p(p-1)$.
	Consider a $2$-cocycle $c:(\ZZ 2)^*\times (\ZZ 2)^*\lra \Z/p\Z$ for the nontrivial central extension 
	\equ{0\lra \Z/p\Z\lra \Z/p^2(p-1)\Z\lra (\ZZ 2)^*\lra 0.}
	Let $\gs:\Z/p\Z\lra \ZZ {\gl_k}$ be defined as the inclusion homomorphism: $\gs(1)=p^{\gl_k-1}$ and $\gs(a)=p^{\gl_k-1}a$.
	Define a group structure on the set $E=\ZZ {\gl_k}\times \autgp$ as follows.
	Let $(x,g),(x',g')\in E$. Then define the multiplication on $E$ as:
	\equ{(x,g).(x',g')=(x+x'+\gs(c(\gc\circ\gf(g),\gc\circ\gf(g'))),gg').}
	\begin{claim}
		Consider the following central extension,
		\equ{0\lra \ZZ {\gl_k}\lra E\lra \autgp \lra 0.}
		Then $E$ is a nontrivial central extension of $\autgp$ by $\Z/p^{\gl_k}\Z$.
	\end{claim}
	\begin{proof}[Proof of Claim]
		Consider the set $E_1=\Z/p\Z\times \autgp$. Define a group structure on the set $E_1=\ZZ {}\times \autgp$ as follows.
		Let $(x,g),(x',g')\in E_1$. Then define the multiplication on $E_1$ as:
		\equ{(x,g).(x',g')=(x+x'+c(\gc\circ\gf(g),\gc\circ\gf(g')),gg').}
		Then $E_1$ is a nontrivial central extension of $\autgp$ by $\Z/p\Z$ if and only if $E$ is a nontrivial central extenstion of $\autgp$ by $\ZZ {\gl_k}$. Now consider the diagonal subgroup $D=(\ZZ {\gl_i})^*\subs \us{l=1}{\os{k}{\oplus}}((\ZZ {\gl_l})^*)^{\gr_l}\subs \autgp$ where $i$ is as mentioned in the theorem. Suppose the extension is given as:
		\equ{0\lra \ZZ {}\lra E_1 \os{\gp}{\lra} \autgp \lra 0.}
		Then we have a central extension.
		\equ{0\lra \ZZ {}\lra \gp^{-1}(D) \os{\gp}{\lra} D \lra 0.}
		Note that $\gp^{-1}(D)$ is an abelian group and as a set $\gp^{-1}(D)=\Z/p\Z\times (\ZZ {\gl_i})^*$.
		We have a group homomorphism $\Z/p\Z\times (\ZZ {\gl_i})^* \os{\mod p^2}{\lra} \Z/p\Z\times (\ZZ {2})^*$ where the group strucure in both of them are given by the respective $2$-cocycles.
		So we have a commutative diagram
		\[
		\begin{tikzcd}
			0 \arrow[r,""]& \Z/p\Z \arrow[hook]{r}{i} \arrow[]{d}{\Vert}[swap]{Id_{\Z/p\Z}} & \gp^{-1}(D) \arrow[two heads]{r}{\gp} \arrow[]{d}{}[swap]{\mod p^2} & D \arrow[r,""] \arrow[]{d}{\mod p^2}[swap]{} & 0\\
			0 \arrow[r,""]& \Z/p\Z\arrow[hook]{r}{i_1} & \Z/p\Z\times (\Z/p^2\Z)^*\arrow[two heads]{r}{\gp_1}& (\Z/p^2\Z)^* \arrow[r,""] & 0
		\end{tikzcd}
		\]
		Hence now we observe that by a straight-forward calculation we have the set $\Z/p\Z\times (\Z/p^2\Z)^*$ as a group is isomorphic to $\Z/p^2(p-1)\Z$ which is a nontrivial central extension. Hence $\gp^{-1}(D)$ and $E_1$ cannot be trivial central extensions implying that $E$ is not a trivial extension. This proves the claim.
	\end{proof}
	Continuing with the proof of the theorem, we have that the extension $E$ gives rise to a nontrivial  central extension
	\equ{0\lra \grpp\lra \mcl{A}_{\ul{\gm}}\oplus E\lra \autgp\lra 0} where 
	$\ul{\gm}=\big(\gl_1^{\gr_1}>\gl_2^{\gr_2}>\cdots>\gl_k^{\gr_k-1}\big)$. Hence $H^2_{Trivial\ Action}(\autgp,\grpp)\neq 0$.
	
\end{proof}
\subsection{\bf{The Vanishing Case of $H^2_{Trivial\ Action}(\autgp,\grpp)$}}
~\\
Throughout this section we assume that $\ugl$ is a partition such that $\gl_i=k-i+1$. We also assume that $p$ is an odd prime and $p\neq 3$. However we will mention where we require that the prime $p$ must be odd and where we also require in addition that the odd prime $p$ is not equal to $3$.
\subsubsection{\bf{Chief Series of a $p$-Sylow Subgroup of $\autgp$}}
Let $\mcl{P}_{\ul{\gl}}\subs \autgp$  be the subgroup consisting of those automorphisms of $\grpp$ which are unipotent lower triangular matrices modulo $p$ in $GL_{\gr}(\Z/p\Z)$ where $\gr=\us{i=1}{\os{k}{\sum}}\gr_i$.  Then $\mcl{P}_{\ul{\gl}}$ is a $p$-Sylow subgroup of $\autgp$. We describe a chief series for $\mcl{P}_{\ul{\gl}}$.

First we describe a typical element $g^a\in \mcl{P}_{\ul{\gl}}$ where $g$ and $a$ are just two symbols.  
Let $g^a=[g^a_{mn}]_{1\leq m,n\leq k}$ where $g^a_{mn}:(\ZZ {\gl_n})^{\gr_n}\lra (\ZZ {\gl_m})^{\gr_m}$, that is, $g^a_{mn}=[(g^{ij}_{mn})^a]_{1\leq i\leq \gr_m,1\leq j\leq \gr_n}$ where we have 
\equ{(g^{ij}_{mn})^a=\begin{cases}
		1+a_{1,i,i,m,m}p+a_{2,i,i,m,m}p^2+\cdots+a_{\gl_m-1,i,i,m,m}p^{\gl_m-1}\\ \text{ if }m=n,i=j\\
		a_{0,i,j,m,m}+a_{1,i,j,m,m}p+a_{2,i,j,m,m}p^2+\cdots+a_{\gl_m-1,i,j,m,m}p^{\gl_m-1}\\ \text{ if }m=n,i>j\\
		a_{1,i,j,m,m}p+a_{2,i,j,m,m}p^2+\cdots+a_{\gl_m-1,i,j,m,m}p^{\gl_m-1}\\ \text{ if }m=n,i<j\\
		a_{0,i,j,m,n}+a_{1,i,j,m,n}p+a_{2,i,j,m,n}p^2+\cdots+a_{\gl_m-1,i,j,m,n}p^{\gl_m-1}\\ \text{ if }m>n\\
		a_{\gl_m-\gl_n,i,j,m,n}p^{\gl_m-\gl_n}+\cdots+a_{\gl_m-1,i,j,m,n}p^{\gl_m-1}\\ \text{ if }m<n\\
\end{cases}}
with $a_{l,i,j,m,n}\in \{0,1,2,\cdots,p-1\}$ for $0\leq l\leq \gl_m-1,1\leq i\leq \gr_m,1\leq j\leq \gr_n,1\leq m,n\leq k$.
The value $a_{0,i,i,m,m}$ is defined to be $1$ for $1\leq i\leq \gr_m,1\leq m\leq k$ and the value $a_{l,i,j,m,n}$ is defined to be zero for $l=0,1\leq m=n\leq k,1\leq i<j\leq \gr_m$ and for $0\leq l\leq \gl_m-\gl_n-1,1\leq i\leq \gr_m,1\leq j\leq \gr_n,1\leq m<n\leq k$. The matrix $g^a$ is of size $\gr\times \gr$. The position of $(g^{ij}_{mn})^a$ in the matrix $g^a$ is \equ{Pos((g^{ij}_{mn})^a)=\big(\us{f=1}{\os{m-1}{\sum}}\gr_f+i,\us{h=1}{\os{n-1}{\sum}}\gr_h+j\big).}
We observe that $a_{l,i,j,m,n}$ is the coefficient of $p^l$ in the $p$-adic expansion of the element $(g^{ij}_{mn})^a$ in the matrix $g^a$ which occurs in the position $Pos((g^{ij}_{mn})^a)$.  
We also write $Pos(a_{l,i,j,m,n})=\big(\us{f=1}{\os{m-1}{\sum}}\gr_f+i,\us{h=1}{\os{n-1}{\sum}}\gr_h+j\big)$ and sometimes when the element $g^a$ is understood by default then we write 
\equ{Pos(l,i,j,m,n)=\big(\us{f=1}{\os{m-1}{\sum}}\gr_f+i,\us{h=1}{\os{n-1}{\sum}}\gr_h+j\big).} 

Define a function $\gch:\mbb{N} \times \mbb{N}\lra \mbb{Z}$ as $\gch(x,y)=x-y$. Let $S'=\{(l,i,j,m,n)\mid 0\leq l\leq \gl_m-1,1\leq i\leq \gr_m,1\leq j\leq \gr_n,1\leq m,n\leq k\}$ and $S''=\{(0,i,j,m,m)\mid 1\leq i\leq j\leq \gr_m,1\leq m\leq k\}\bigcup\{(l,i,j,m,n)\mid 0\leq l\leq \gl_m-\gl_n-1,1\leq i\leq \gr_m,1\leq j\leq \gr_n,1\leq m<n\leq k\}$.
We define a total order $\leq_{TO}$ on the set $S=(S'\bs S'')\bigcup \{\es\}$ as follows. For $(l,i,j,m,n), (l',i',j',m',n')\in S\bs\{\es\}$, we say that 
\equ{(l,i,j,m,n)\leq_{TO} (l',i',j',m',n')}
\begin{itemize}
	\item if $l>l'$ or
	\item if $l=l'$ and $\gch(Pos(l,i,j,m,n))>\gch(Pos(l',i',j',m',n'))$ or
	\item if $l=l'$ and $\gch(Pos(l,i,j,m,n))=\gch(Pos(l',i',j',m',n'))$ and \equ{\us{f=1}{\os{m-1}{\sum}}\gr_f+i\geq \us{f=1}{\os{m'-1}{\sum}}\gr_f+i'.}
\end{itemize}
We also define that $\es\in S$ is the least element, that is, $\es\leq_{TO}(l,i,j,m,n)$ for all $(l,i,j,m,n)\in S\bs \{\es\}$. It is clear that $\leq_{TO}$ is a total order on the set $S$.

Now we define a chain of subgroups of $\mcl{P}_{\ul{\gl}}$ indexed by the set $S$. Define for $s\in S$,
\equ{\mcl{P}^{s}_{\ul{\gl}}=\{g^a\in \mcl{P}_{\ul{\gl}}\mid a_{s'}=0 \text{ for all }s'\in S \text{ and }s<_{TO}s'\}.}
Note that if $s=\es\in S$ then $\mcl{P}^{\es}_{\ul{\gl}}=\{1\}$, the trivial subgroup. From the definition of $\mcl{P}^{s}_{\ul{\gl}}$ it is clear that $\mcl{P}^{s}_{\ul{\gl}}\sbnq \mcl{P}^{s'}_{\ul{\gl}} \Llra s<_{TO}s'$.
\begin{theorem}
\label{theorem:Subgroup}
The subset $\mcl{P}^{s}_{\ul{\gl}}$ is indeed a subgroup for any $s\in S$.
\end{theorem}
\begin{proof}
\begin{figure}[h]
		\centering
		\includegraphics[width = 0.8\textwidth]{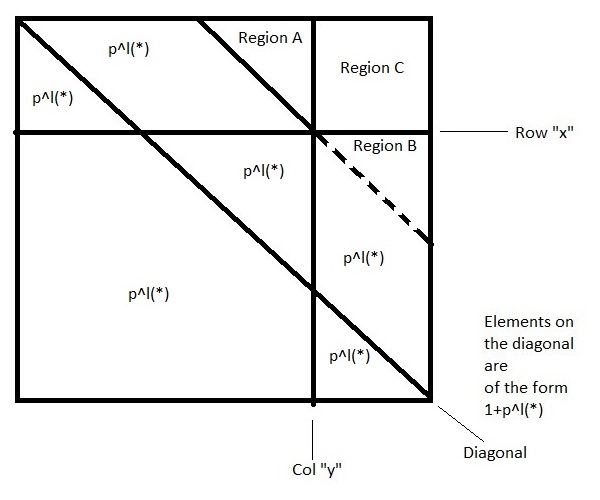}
		\caption{Matrix Form if $(x,y)$ corresponds to a strictly upper triangular entry}
		\label{fig:One}
\end{figure}
If $s=\es$ then the proof is trivial. So assume that $s=(l,i,j,m,n)\in S\bs\{\es\}$ and $Pos(s)=(x,y)$.  Suppose $\gch(Pos(s))=x-y<0$, that is, $(x,y)$ corresponds to a strictly upper triangular entry. Let $g^a,g^b\in \mcl{P}^s_{\ul{\gl}}$. Then certainly $g^c=g^a.g^b\in \mcl{P}_{\ul{\gl}}$. Moreover $g^a \equiv g^b\equiv I\mod p^l$. Hence $g^c\equiv I\mod p^l$.

Now consider Figure~\ref{fig:One}. 
Let region $A=\{(u,v)\in \mbb{N}\times \mbb{N}\mid 1\leq u< v\leq y, u-v\leq x-y<0\text{ and } (u,v)\neq (x,y)\}.$   Let region $B=\{(u,v)\in \mbb{N}\times \mbb{N}\mid x\leq u<v\leq \gr=\gr_1+\cdots+\gr_k, u-v<x-y\}$. The dotted boundary line of region $B$ is not included in region $B$. Let region $C=\{(u,v)\in \mbb{N}\times \mbb{N}\mid 1\leq u<x,y< v\leq \gr=\gr_1+\cdots+\gr_k  \}$.
To prove that $P^s_{\ul{\gl}}$ is closed under group multiplication, it is enough to prove that, the entries in regions $A,B,C$ are of the form $p^{l+1}(*)$. 

For $u<v$, the $(u,v)^{th}$ entry of $g^c=g^a.g^b$ is of the form 
\equan{One}{(g^c)_{(u,v)}=\us{z=1}{\os{u}{\sum}}(g^a)_{(u,z)}(g^b)_{(z,v)}+\us{z=u+1}{\os{v}{\sum}}(g^a)_{(u,z)}(g^b)_{(z,v)}+\us{z=v+1}{\os{\gr}{\sum}}(g^a)_{(u,z)}(g^b)_{(z,v)}.}
If $(u,v)\in$ region $A$ or region $C$ then in the first sum of the RHS of~\ref{Eq:One}, for $1\leq z\leq u, (z,v)\in$ region $A$ or region $C$ and hence $(g^b)_{(z,v)}$ is of the form $p^{l+1}(*)$. If $(u,v)\in$ region $B$ then in the first sum of the RHS of~\ref{Eq:One}, for $1\leq z\leq u, (z,v)\in$ region $B$ or region $C$. So $(g^b)_{(z,v)}$ is of the form $p^{l+1}(*)$.

In the second sum of the RHS of~\ref{Eq:One}, for $u+1\leq z<v$ we have $(g^a)_{(u,z)}$ is of the form $p(*)$ since it is lower triangular modulo $p$ and $(g^b)_{(z,v)}$ is of the form $p^l(*)$. In the second sum of the RHS of~\ref{Eq:One}, for $z=v$, $(g^a)_{(u,v)}$ is of the form $p^{l+1}(*)$ since $(u,v)$ belongs to region $A$ or region $B$ or region $C$. 

In the third sum of the RHS of~\ref{Eq:One}, we have
$(g^a)_{(u,z)}$ is of the form $p(*)$ since it is lower triangular modulo $p$ and $(g^b)_{(z,v)}$ is of the form $p^l(*)$. So we conclude that $(g^c)_{(u,v)}$ is of the form $p^{l+1}(*)$. 

The proof is similar if $(x,y)$ corresponds to a strictly lower triangular entry or a diagonal entry. Hence $\mcl{P}^s_{\ul{\gl}}$ is closed under group multiplication. 

Now we prove that $\mcl{P}^s_{\ul{\gl}}$ is closed under inverses. This is immediate because any nonempty finite subset of a group which is also closed under group multiplication is closed under inverses and contains the identity element. Hence $\mcl{P}^s_{\ul{\gl}}$ is a subgroup.
\end{proof}
\begin{remark}
\label{remark:MultiplyEntry}
We also observe in the above proof that if $g^a,g^b\in \mcl{P}^s_{\ul{\gl}},s\in S\bs\{\es\},s=(l,i,j,m,n),Pos(s)=(x,y)$ and $g^c=g^a.g^b$ then the entry 
\begin{enumerate}[label=(\roman*)]
	\item $(g^c)_{(x,y)}\equiv(g^a)_{(x,y)}+(g^b)_{(x,y)} \mod p^{l+1}$, that is, $c_{l,i,j,m,n}\equiv a_{l,i,j,m,n}+b_{l,i,j,m,n} \mod p$ if $x\neq y$.
	\item $(g^c)_{(x,y)}-1\equiv (g^a)_{(x,y)}-1+(g^b)_{(x,y)}-1 \mod p^{l+1}$, that is, $(g^c)_{(x,y)}\equiv 1+p^la_{l,i,j,m,n}+p^lb_{l,i,j,m,n}\mod p^{l+1},(g^a)_{(x,y)} \equiv 1+p^la_{l,i,j,m,n}\mod p^{l+1}$, $(g^b)_{(x,y)} \equiv 1+p^lb_{l,i,j,m,n}\mod p^{l+1}$  if $x=y$.
\end{enumerate}
\end{remark}
\begin{remark}
In the proof of Theorem~\ref{theorem:Subgroup}, $p$ can be any prime. 
\end{remark}
\begin{theorem}
\label{theorem:CosetRepresentatives}
Let $t_1,t_2\in S$ and $t_1<t_2$. Let $S_{t_i}=\{t\in S\mid t\leq t_i,t\neq \es\},i=1,2$. Then 
\begin{enumerate}
\item the set $C_{t_1}^{t_2}=\{g^c\mid g^c\in \mcl{P}^{t_2}_{\ul{\gl}},c_t=0\text{ for all }t\in S_{t_1}\}$ forms a set of distinct left and right coset representatives for the subgroup $\mcl{P}^{t_1}_{\ul{\gl}}\subs \mcl{P}^{t_2}_{\ul{\gl}}$.
\item For $g^c\in C^{t_2}_{t_1}$, $g^d\in g^c\mcl{P}^{t_1}_{\ul{\gl}}$ or $g^d\in \mcl{P}^{t_1}_{\ul{\gl}}g^c$ if and only if $c_t=d_t$ for all $t_1<t\leq t_2$. 
\item For $g^c,g^d\in \mcl{P}_{\ul{\gl}}$, $g^c\mcl{P}^{t_1}_{\ul{\gl}}=g^d\mcl{P}^{t_1}_{\ul{\gl}}$ or $\mcl{P}^{t_1}_{\ul{\gl}}g^c=\mcl{P}^{t_1}_{\ul{\gl}}g^d$ if and only if $c_t=d_t$ for all $t_1<t$. 
\item The index $[\mcl{P}^{t_2}_{\ul{\gl}}: \mcl{P}^{t_1}_{\ul{\gl}}]=p^{\mid S_{t_2}\bs S_{t_1}\mid}$.
\end{enumerate}
\end{theorem}
\begin{proof}
For $s=\es$ the proof is trivial. Assume that $s\neq \es$. Let $s_1$ be the predecessor of $s$ in the set $S$ where $Pos(s)=(x,y)$. Then using Remark~\ref{remark:MultiplyEntry} (i),(ii), we have that the set of elementary or diagonal matrices $\{E_{(x,y)}(p^la_{l,i,j,m,n})\mid 0\leq a_{l,i,j,m,n}\leq p-1\}$ form a set of distinct left and right coset representatives for the subgroup $\mcl{P}^{s_1}_{\ul{\gl}}\subs \mcl{P}^{s}_{\ul{\gl}}$. Here $E_{(x,y)}(p^la_{l,i,j,m,n})=I+e_{(x,y)}(p^la_{l,i,j,m,n})$ where $e_{(x,y)}(p^la_{l,i,j,m,n})$ is the matrix of zeroes except for the $(x,y)^{th}$ entry which is $p^la_{(l,i,j,m,n)}$. Note that if $(l,i,j,m,n)$ is in diagonal position, that is, $x=y$ then $l>0$ and $E_{(x,y)}(p^la_{l,i,j,m,n})$ is a diagonal matrix. As a consequence we have $[\mcl{P}^{s}_{\ul{\gl}}:\mcl{P}^{s_1}_{\ul{\gl}}]=p$.

We prove the theorem by induction on the cardinality $\mid S_{t_2}\bs S_{t_1}\mid$. If $\mid S_{t_2}\bs S_{t_1}\mid=1$ then it follows from Remark~\ref{remark:MultiplyEntry} (i),(ii). Suppose $\mid S_{t_2}\bs S_{t_1}\mid>1$. Let $t_1<t_3<t_2$ where $t_3$ is the successor of $t_1$ and $g^c,g^d\in C^{t_2}_{t_1}$. Suppose the two left cosets $g^c\mcl{P}^{t_1}_{\ul{\gl}},g^d\mcl{P}^{t_1}_{\ul{\gl}}$ are equal then $g^c\mcl{P}^{t_3}_{\ul{\gl}}=g^d\mcl{P}^{t_3}_{\ul{\gl}}$. Hence we have $c_t=d_t$ for all $t_3<t\leq t_2$ by induction on $\mid S_{t_2}\bs S_{t_3}\mid=\mid S_{t_2}\bs S_{t_1}\mid-1$. A similar conclusion follows when we consider right cosets. We now prove $c_{t_3}=d_{t_3}$. Let $t_4$ be the predecessor of $t_2=(l_2,i_2,j_2,m_2,n_2)$. Let $Pos(t_2)=(u_2,v_2)$. Then we have 
\begin{itemize}
	\item $E_{(u_2,v_2)}(-p^{l_2}c_{l_2,i_2,j_2,m_2,n_2})g^c, E_{(u_2,v_2)}(-p^{l_2}d_{l_2,i_2,j_2,m_2,n_2})g^d\in \mcl{P}^{t_4}_{\ul{\gl}}$,
	\item $g^cE_{(u_2,v_2)}(-p^{l_2}c_{l_2,i_2,j_2,m_2,n_2}), g^dE_{(u_2,v_2)}(-p^{l_2}d_{l_2,i_2,j_2,m_2,n_2})\in \mcl{P}^{t_4}_{\ul{\gl}}$,
	\item $E_{(u_2,v_2)}(-p^{l_2}c_{l_2,i_2,j_2,m_2,n_2})g^c\mcl{P}^{t_1}_{\ul{\gl}}=E_{(u_2,v_2)}(-p^{l_2}d_{l_2,i_2,j_2,m_2,n_2})g^d\mcl{P}^{t_1}_{\ul{\gl}}$,
	\item $\mcl{P}^{t_1}_{\ul{\gl}}g^cE_{(u_2,v_2)}(-p^{l_2}c_{l_2,i_2,j_2,m_2,n_2})=\mcl{P}^{t_1}_{\ul{\gl}}g^dE_{(u_2,v_2)}(-p^{l_2}d_{l_2,i_2,j_2,m_2,n_2})$,
	\item $c_{l_2,i_2,j_2,m_2,n_2}=d_{l_2,i_2,j_2,m_2,n_2}$.
\end{itemize}
Hence by induction on $\mid S_{t_4}\bs S_{t_1}\mid=\mid S_{t_2}\bs S_{t_1}\mid-1$, we have that, if $g^a=$\linebreak $E_{(u_2,v_2)}(-p^{l_2}c_{l_2,i_2,j_2,m_2,n_2})g^c$ and 
$g^b=E_{(u_2,v_2)}(-p^{l_2}d_{l_2,i_2,j_2,m_2,n_2})g^d$ then $a_t=b_t$ for all $t_1<t\leq t_4$. Similarly if $g^A=g^cE_{(u_2,v_2)}(-p^{l_2}c_{l_2,i_2,j_2,m_2,n_2})$
and $g^B=$\linebreak$g^dE_{(u_2,v_2)}(-p^{l_2}d_{l_2,i_2,j_2,m_2,n_2})$ then $A_t=B_t$ for all $t_1<t\leq t_4$. We observe that $g^a$ and $g^c$ differ only in the $(u_2)^{th}$-row. Also $g^A,g^c$ differ only in the $(v_2)^{th}$-column. A similar conclusion follows for the pairs $g^b,g^d$ and $g^B,g^d$. So if $Pos(t_3)\neq (u_2,v_2)=Pos(t_2)$ then $c_{t_3}=d_{t_3}$. 

Now we assume that $Pos(t_3)=Pos(t_2)=(u_2,v_2)$. Let $t_3=(l_3,i_3,j_3,m_3,n_3)$. We need to prove that the coefficient of $p^{l_3}$ in the $p$-adic expansions of $(g^c)_{(u_2,v_2)},(g^d)_{(u_2,v_2)}$ are equal.  We have 
\equan{Two}{(g^a)_{(u_2,v_2)}&=(g^c)_{(u_2,v_2)}-p^{l_2}c_{l_2,i_2,j_2,m_2,n_2}(g^c)_{(v_2,v_2)}\\ \Ra (g^c)_{(u_2,v_2)}&=(g^a)_{(u_2,v_2)}+p^{l_2}c_{l_2,i_2,j_2,m_2,n_2}(g^c)_{(v_2,v_2)}}
and 
\equan{Three}{(g^b)_{(u_2,v_2)}&=(g^d)_{(u_2,v_2)}-p^{l_2}d_{l_2,i_2,j_2,m_2,n_2}(g^d)_{(v_2,v_2)}\\ \Ra
	(g^d)_{(u_2,v_2)}&=(g^b)_{(u_2,v_2)}+p^{l_2}d_{l_2,i_2,j_2,m_2,n_2}(g^d)_{(v_2,v_2)}.}
There are two cases: $l_2>0$ and $l_2=0$. Suppose $l_2>0$. In Equations~\ref{Eq:Two},~\ref{Eq:Three},
the coefficients of $p^{\ti{l}}$ in $(g^a)_{(u_2,v_2)}$ and $(g^b)_{(u_2,v_2)}$ are equal for all $\ti{l}\leq l_3$ since $a_t=b_t$ for all $t_1<t$. Also the coefficients of $p^{\ti{l}}$ in $(g^c)_{(v_2,v_2)},(g^d)_{(v_2,v_2)}$ are equal for $\ti{l}<l_3$ since $c_t=d_t$ for all $t_3<t$. Hence in this case $l_2>0$ we have $c_{t_3}=d_{t_3}$.

Suppose $l_2=0$ then $u_2> v_2$ because the coefficient of $p^0$ which can be nonzero and not equal to one can occur only in the strictly lower triangular entries. Hence the coefficients of $p^{\ti{l}}$ in $(g^c)_{(v_2,v_2)},(g^d)_{(v_2,v_2)}$ are equal for $\ti{l}\leq l_3$ and the coefficients of $p^{\ti{l}}$ in $(g^a)_{(u_2,v_2)}$ and $(g^b)_{(u_2,v_2)}$ are equal for all $\ti{l}\leq l_3$. Again here we have $c_{t_3}=d_{t_3}$. The converse statements in (2),(3) also follow in a similar way. This proves the theorem.
\end{proof}
\begin{remark}
	In the proof of Theorem~\ref{theorem:CosetRepresentatives}, $p$ can be any prime. 
\end{remark}
\begin{theorem}
	\label{theorem:Chain}
	We have
	\begin{enumerate}
		\item $\mcl{P}^{s}_{\ul{\gl}}$ is a normal subgroup of $\mcl{P}_{\ul{\gl}}$ for $s\in S$.
		\item If $s'$ is the successor element of $s$ in $S$  then  
		\equ{\frac{\mcl{P}^{s'}_{\ul{\gl}}}{\mcl{P}^{s}_{\ul{\gl}}} \cong \Z/p\Z.}
		\item If $s'$ is the successor element of $s$ in $S$ then there is a central extension 
		\equan{CentralExtension}{0\lra \frac{\mcl{P}^{s'}_{\ul{\gl}}}{\mcl{P}^{s}_{\ul{\gl}}} \lra \frac{\mcl{P}_{\ul{\gl}}}{\mcl{P}^{s}_{\ul{\gl}}}\lra \frac{\mcl{P}_{\ul{\gl}}}{\mcl{P}^{s'}_{\ul{\gl}}}\lra 0.} 
	\end{enumerate}
\end{theorem}
\begin{proof}
We prove Theorem~\ref{theorem:Chain} by induction on the elements of the totally ordered set $S\bs\{\es\}$. Let $(l,i,j,m,n)=s\in S\bs\{\es\}$ and $s_1\in S$ be the predecessor of $s$. Let $Pos(s)=(x,y)$. We can assume that $\mcl{P}^{s_1}_{\ul{\gl}}$ is a normal subgroup of $\mcl{P}_{\ul{\gl}}$ by induction. Now let $g^b\in \mcl{P}_{\ul{\gl}}$. We have $g^bE_{(x,y)}(p^la_{l,i,j,m,n})=g^b+g^be_{(x,y)}(p^la_{l,i,j,m,n})$ and $E_{(x,y)}(p^la_{l,i,j,m,n})g^b=g^b+e_{(x,y)}(p^la_{l,i,j,m,n})g^b$.  The elements of the matrices $g^bE_{(x,y)}(p^la_{l,i,j,m,n})$ and $E_{(x,y)}(p^la_{l,i,j,m,n})g^b$ are the same except for the $x^{th}$-row and $y^{th}$-column.
In the $x^{th}$-row we have for any $1\leq z\leq \gr=\gr_1+\cdots+\gr_k$,
\equ{(E_{(x,y)}(p^la_{l,i,j,m,n})g^b)_{(x,z)}=g^b_{(x,z)}+p^la_{l,i,j,m,n}g^b_{(y,z)}.}
If $z>y$ then $p^{l+1}$ divides $p^la_{l,i,j,m,n}g^b_{(y,z)}$. If $z=y$ then $(E_{(x,y)}(p^la_{l,i,j,m,n})g^b)_{(x,y)}=g^b_{(x,y)}+p^la_{l,i,j,m,n}+p^{l+1}(*)$ and if $z<y$ then $x-z>x-y=\gch(Pos(s))$.
In the $y^{th}$-column we have for any $1\leq w\leq \gr=\gr_1+\cdots+\gr_k$,
\equ{(g^bE_{(x,y)}(p^la_{l,i,j,m,n}))_{(w,y)}=g^b_{(w,y)}+p^la_{l,i,j,m,n}g^b_{(w,x)}.}
Here if $w<x$ then $p^{l+1}$ divides $p^la_{l,i,j,m,n}g^b_{(w,x)}$. If $w=x$ then
 $(g^bE_{(x,y)}(p^la_{l,i,j,m,n}))_{(x,y)}$ $=g^b_{(x,y)}+p^la_{l,i,j,m,n}+p^{l+1}(*)$ and if $w>x$ then $w-y>x-y=\gch(Pos(s))$.
Hence we conclude that in the matrices $g^C=E_{(x,y)}(p^la_{l,i,j,m,n})g^b$ and $g^D=g^bE_{(x,y)}(p^la_{l,i,j,m,n})$, we have 
$C_t=D_t$ for all $s_1<t$. So using Theorem~\ref{theorem:CosetRepresentatives} we conclude that $g^C\mcl{P}^{s_1}_{\ul{\gl}}=g^D\mcl{P}^{s_1}_{\ul{\gl}}$ and $\mcl{P}^{s_1}_{\ul{\gl}}g^C=\mcl{P}^{s_1}_{\ul{\gl}}g^D$.
Hence $\frac{\mcl{P}^{s}_{\ul{\gl}}}{\mcl{P}^{s_1}_{\ul{\gl}}}$ is a central normal subgroup of $\frac{\mcl{P}_{\ul{\gl}}}{\mcl{P}^{s_1}_{\ul{\gl}}}$. Hence Theorem~\ref{theorem:Chain}$(1),(2),(3)$ follow.  This completes the proof of Theorem~\ref{theorem:Chain}. 
\end{proof}
\begin{remark}
The chain of subgroups $\mcl{P}^s_{\ul{\gl}},s\in S$ is a chief series for $\mcl{P}_{\ul{\gl}}$.
\end{remark}
\begin{remark}
	In the proof of Theorem~\ref{theorem:Chain}, $p$ can be any prime. 
\end{remark}
\subsubsection{\bf{Commutator Subgroup of $\mcl{P}_{\ul{\gl}}\subs \autgp$}}
\begin{theorem}
	\label{theorem:Commutator}
	Let $\mcl{P}'_{\ul{\gl}}=[\mcl{P}_{\ul{\gl}},\mcl{P}_{\ul{\gl}}]$ be the commutator subgroup of $\mcl{P}_{\ul{\gl}}\subs \autgp$. If $k>1$, define \equan{Abelianization}{T&=\{(0,i,j,m,n)\in S\mid \gch(Pos(0,i,j,m,n))=1\}\\
		&\bigcup \{(1,1,\gr_{m+1},m,m+1)\in S,1\leq m\leq k-1\}}
	and if $k=1$, define \equ{T=\{(0,i,j,1,1)\in S\mid \gch(Pos(0,i,j,1,1))=1\}.}
	For $g^a,g^b,g^c\in \mcl{P}_{\ul{\gl}}$, suppose $g^c=g^a.g^b$.  Then we have 
	\begin{enumerate}
		\item \equ{(l,i,j,m,n)\in T\Ra c_{(l,i,j,m,n)}\equiv a_{(l,i,j,m,n)}+b_{(l,i,j,m,n)}\mod p,}
		\item \equ{g^a\in \mcl{P}'_{\ul{\gl}} \Llra \text{ for all } (l,i,j,m,n)\in T, a_{(l,i,j,m,n)}=0.}
		\item \equ{\frac{\mcl{P}_{\ul{\gl}}}{\mcl{P}'_{\ul{\gl}}}\cong (\Z/p\Z)^{\mid T\mid}.}	
	\end{enumerate} 
\end{theorem}
\begin{proof}
We prove $(1)$.
Let $s\in T$. Clearly if $s=(0,i,j,m,n)$ such that $\gch(Pos(s))=1$ then $s$ appears in the leading subdiagonal or first subdiagonal (just below the diagonal). $a_s,b_s,c_s$ being the first coefficients in the $p$-adic expansion of the entry in position $Pos(s)$, it is clear that $c_s\equiv a_s+b_s\mod p$. Now assume that $s=(1,1,\gr_{m+1},m,m+1)$ for some $1\leq m\leq k-1$. Let $Pos(s)=(x,y)$ where $x=\us{f=1}{\os{m-1}{\sum}}\gr_f+1,y=\us{h=1}{\os{m}{\sum}}\gr_h+\gr_{m+1}$. Then $\gch(Pos(s))=1-\gr_m-\gr_{m+1}<0$. Hence $s$ appears in a strictly upper triangular entry, that is, $(x,y)$ corresponds to a strictly upper triangular entry. Moreover $(x,y)$ is the position  corresponding to the top right corner of the $(m,m+1)^{th}$ block matrix. We have 
\equ{g^c_{(x,y)}=\us{z=1}{\os{\gr}{\sum}}g^a_{(x,z)}g^b_{(z,y)}.}
So for $1\leq z<x$, we have $p^2\mid g^b_{(z,y)}$ since $2\leq \gl_t-\gl_{m+1}=m+1-t$ for all $t<m$. For $x<z<y$ we have 
$p\mid g^a_{(x,z)},p\mid g^b_{(z,y)}$ since $g^a,g^b$ are lower triangular modulo $p$. If $z>y$ then we have $p^2\mid g^a_{(x,z)}$ since $2\leq \gl_m-\gl_t=t-m$ for all $t>m+1$. For $z=x$ we have $g^a_{(x,x)}g^b_{(x,y)}=pb_s+p^2(*)$ and for $z=y$ we have $g^a_{(x,y)}g^b_{(y,y)}=pa_s+p^2(*)$. Hence we conclude that $c_s\equiv a_s+b_s\mod p$. This proves $(1)$.

We prove $(2)$. For $g^a\in \mcl{P}_{\ul{\gl}}$ if $g^b=(g^a)^{-1}$ then using $(1)$ we obtain that $b_s\equiv -a_s\mod p$ for all $s\in T$. Hence again using $(1)$ for $g^a,g^b\in \mcl{P}_{\ul{\gl}}$ if $g^c=[g^a,g^b]=g^ag^b(g^a)^{-1}(g^b)^{-1}$ then $c_s=0$ for all $s\in T$. Hence for $g^a\in \mcl{P}'_{\ul{\gl}}=[\mcl{P}_{\ul{\gl}},\mcl{P}_{\ul{\gl}}], a_s=0$ for all $s\in T$,
that is, if we define the subgroup $\mcl{C}_{\ul{\gl}}=\{g^a\in \mcl{P}_{\ul{\gl}}\mid a_s=0\text{ for all }s\in T\}$, then $\mcl{P}'_{\ul{\gl}}\subseteq \mcl{C}_{\ul{\gl}}$.

Now we prove the converse in $(2)$, that is, $\mcl{C}_{\ul{\gl}}\subseteq \mcl{P}'_{\ul{\gl}}$. Let $s\in S\bs T, s=(l,i,j,m,n)$ and let $Pos(s)=(x,y)$. Consider the matrix $E_{(x,y)}(p^la_s)=I+e_{(x,y)}(p^la_s)$. We show that it is a commutator or product of commutators. 

First assume that $x<y$, that is, $(x,y)$ corresponds to a strictly upper triangular entry. Then we have $m\leq n$. 

If $n=m$ then $i<j$ and $\gr_m\geq 2$ and position $(x,y)$ is in the diagonal block $(m,m)$. If $m=k$ then the entry in the $(x,y)^{th}$ position is zero. So $m<k\Ra y<\gr=\gr_1+\cdots+\gr_k$. We have $E_{(x,y)}(p^la_s)=[E_{(x,y+1)}(p^la_s),E_{(y+1,y)}(1)]$. Position $(x,y+1)$ occurs in either block $(m,m)$ or block $(m,m+1)$. So $p^l$ is allowed in position $(x,y+1)$ if $p^l$ is allowed in position $(x,y)$. 

If $n>m+1$ then $\gl_m-\gl_n=n-m\geq 2$ and $x+1<y$. So $l\geq 2$. We have $E_{(x,y)}(p^la_s)=[E_{(x,x+1)}(pa_s),E_{(x+1,y)}(p^{l-1})]$. Note that position $(x+1,y)$ occurs in either block $(m,n)$ or $(m+1,n)$. So $p^{l-1}$ is allowed in position $(x+1,y)$.

Suppose $n=m+1$. If $i>1$ then position $(x-1,y)$ is also in the block $(m,m+1)$. So $E_{(x,y)}(p^la_s)=[E_{(x,x-1)}(1),E_{(x-1,y)}(p^la_s)]$. If $j<\gr_{m+1}$ then position $(x,y+1)$ is also in the block $(m,m+1)$.
So $E_{(x,y)}(p^la_s)=[E_{(x,y+1)}(p^la_s),E_{(y+1,y)}(1)]$. So assume $i=1,j=\gr_{m+1}$. If $l\geq 2$ then positive integer $3$ is a part of $\ul{\gl}$ and $k\geq 3$. So if $m+1<k$ then So $E_{(x,y)}(p^la_s)=[E_{(x,y+1)}(p^la_s),E_{(y+1,y)}(1)]$. Position $(x,y+1)$ occurs in block $(m,m+2)$ and $p^l$ is allowed in position $(x,y+1)$ since $l\geq 2$. If $m+1=k\geq 3$ then $m\geq 2$. So $E_{(x,y)}(p^la_s)=[E_{(x,x-1)}(1),E_{(x-1,y)}(p^la_s)]$. Position $(x-1,y)$ occurs in block $(m-1,m+1)=(k-2,k)$ and $p^l$ is allowed in position $(x-1,y)$ since $l\geq 2$.

Finally in the scenario $x<y$, we are left with the case $n=m+1,i=1,j=\gr_{m+1},l=1$, that is, $s=(1,1,\gr_{m+1},m,m+1)$. But then we have that $s\in T$.

Now consider the scenario $x>y$, that is, $(x,y)$ corresponds to a strictly lower triangular entry.  Here we have $m\geq n$.

If $m=n$, that is, position $(x,y)$ is in the diagonal block $(m,m)$ then $i>j$ and $\gr_m\geq 2$. If $\gr_m=2$ then $i=2,j=1,x-y=1$.
If $l=0$ then $s=(0,2,1,m,m)\in T$. So assume $l>0$. In this case $k>1$ and $m<k$. So position $(x+1,y)$ occurs in block $(m+1,m)$. So $E_{(x,y)}(p^la_s)=[E_{(x,x+1)}(pa_s),E_{(x+1,y)}(p^{l-1})]$. Note that $p^{l-1}$ is allowed in position $(x+1,y)$ if $p^l$ is allowed in position $(x,y)$. If $\gr_m>2$, then there exists $z\neq x,z\neq y$ such that both positions $(x,z),(z,y)$ are in block $(m,m)$. If $x>z>y$ then $E_{(x,y)}(p^la_s)=[E_{(x,z)}(1),E_{(z,y)}(p^{l}a_s)]$. Now we assume that $x-y=1$ and $l>0$. If $x>y>z$ then $E_{(x,y)}(p^la_s)=[E_{(x,z)}(1),E_{(z,y)}(p^{l}a_s)]$. If $z>x>y$ then $E_{(x,y)}(p^la_s)=[E_{(x,z)}(p^{l}a_s),E_{(z,y)}(1)]$. If $x-y=1$ and $l=0$ then $s=(0,i,j=i-1,m,m)\in T$.

If $m>n+1$ then position $(x,y+1)$ is either in block $(m,n)$ or $(m,n+1)$. Also $x>y+1>y$. We have $E_{(x,y)}(p^la_s)=[E_{(x,y+1)}(p^la_s),E_{(y+1,y)}(1)]$. Note that $p^l$ is allowed in position $(x,y+1)$.

If $m=n+1, x-y>1$ then position $(x,y+1)$ is either in block $(m,n)$ or $(m,n+1)$. Here again we have $E_{(x,y)}(p^la_s)=[E_{(x,y+1)}(p^la_s),E_{(y+1,y)}(1)]$. If $x-y=1$ and $l=0$ then $(0,i,j,m,n)\in T$. If $x-y=1, l>0$ then $k>1,m<k$. So position $(x+1,y)$ is in either in block $(m,n)$ or $(m+1,n)$. We have $E_{(x,y)}(p^la_s)=[E_{(x,x+1)}(p),E_{(x+1,y)}(p^{l-1}a_s)]$. Note that $p^{l-1}$ is allowed in position $(x+1,y)$.

Now we consider the case $x=y$, that is, $m=n,i=j$ and $s=(l,i,i,m,m)$ where $l>0$. Consider first $E_{(x,x)}(p^la_s)$ for $l\geq 2$. In this case position $(x+2,x+2)$ appears in the matrix since $x+2\leq \gr=\gr_1+\cdots+\gr_k$. So position $(x+2,x)$ is in one of the blocks $(m,m),(m+1,m),(m+2,m)$. Similarly $(x,x+2)$ is in one of the blocks $(m,m),(m,m+1),(m,m+2)$. Also $(x+2,x+2)$ is in one of the blocks $(m,m),(m+1,m+1),(m+2,m+2)$. For $2\times 2$ matrices we have, for any symbol $a$ and $a^{-1}$ denoting the inverse of $a$,
\equ{\mattwo{a}{0}{0}{1}=\mattwo{1}{0}{0}{a}\mattwo{1}{1-a}{0}{1}\mattwo{1}{0}{-1}{1}\mattwo{1}{1-a^{-1}}{0}{1}\mattwo{1}{0}{a}{1}.}
Assume that $E_{(z,z)}(p^l(*)),l\geq 2$ is a product of commutators for all $x<z\leq\gr=\gr_1+\cdots+\gr_k$. We prove that $E_{(x,x)}(p^l(*))$ is a product of commutators for all $l\geq 2$.
We have using the above symbolic identity \equa{&E_{(x,x)}(p^la_s)=\\&E_{(x+2,x+2)}(p^la_s)E_{(x,x+2)}(-p^la_s)E_{(x+2,x)}(-1)E_{(x,x+2)}(1-(1+p^la_s)^{-1})E_{(x+2,x)}(p^la_s).}
Note that here $p^l$ for $l\geq 2$ is allowed in position $(x,x+2)$. Here if $l$ is very large then $p^l$ in that position becomes zero. We have $E_{(x,x+2)}(-p^la_s)$ is a commutator, $E_{(x+2,x)}(-1),E_{(x+2,x)}(p^la_s)$ are commutators. Also note that $1-(1+p^la_s)^{-1}=1-p^la_s+p^{l+1}(*)$ with $l\geq 2$. Hence \equ{E_{(x,x+2)}(1-(1+p^la_s)^{-1})=E_{(x,x+2)}(-p^la_s)E_{(x,x+2)}(p^{l+1}(*))} is a product of commutators which we have already shown. 
So $E_{(x,x)}(p^la_s)$ is a product of commutators for $l\geq 2$. So from this we can show that $E_{(x,x)}(p^l(*))$ is a product of commutators for all $l\geq 2$. This is because if $0\leq a<p-1, b$ is any non-negative integer then we have
\equa{E_{(x,x)}(p^la+p^{l+1}b)&=E_{(x,x)}(p^la)E_{(x,x)}\big(-1+\frac{(1+p^la+p^{l+1}b)}{(1+p^la)}\big)\\
	&=E_{(x,x)}(p^la)E_{(x,x)}(p^{l+1}(*)).}  

For any diagonal matrix $D\in \mcl{P}_{\ul{\gl}}$ if the diagonal entries are of the form $1+p^2(*)$ then $D$ is a product of commutators. Now consider $E_{(x,x)}(pa_s)$ where $s=(1,i,i,m,m)$ with $Pos(s)=(x,x)$. We can assume that $x<\gr=\gr_1+\cdots+\gr_k$. Position $(x+1,x)$ occurs either in block $(m,m)$ or $(m+1,m)$. Position $(x,x+1)$ occurs either in block $(m,m)$ or $(m,m+1)$. Position $(x+1,x+1)$ occurs either in block $(m,m)$ or $(m+1,m+1)$. Assume that $E_{(z,z)}(p(*))$ is a product of commutators for all $x<z\leq\gr=\gr_1+\cdots+\gr_k$. We prove that $E_{(x,x)}(p(*))$ is a product of commutators.    
First $E_{(x,x)}(pa_s)$ can be expressed as the following long product of matrices.
\equa{&E_{(x,x)}(pa_s)=E_{(x+1,x)}\big(-pa_s(1+pa_s+p^2a^2_s)^{-1}\big)[E_{(x,x+1)}(pa_s),E_{(x+1,x)}(1)]\\
	&E_{(x,x+1)}\big(p^2a_s^2(1+pa_s+p^2a^2_s)^{-1}\big)E_{(x,x)}\big(-1+(1+pa_s)(1+pa_s+p^2a^2_s)^{-1}\big)\\
&E_{(x+1,x+1)}(pa_s+p^2a_s^2).} 
This can be checked by first evaluating the commutator $[E_{(x,x+1)}(pa_s),E_{(x+1,x)}(1)]$ and then reducing it by elementary matrices to the matrix $E_{(x,x)}(pa_s)$.
Note that here $E_{(x,x)}\big(-1+(1+pa_s)(1+pa_s+p^2a^2_s)^{-1}\big)=E_{(x,x)}\big(p^2(*)\big)$ which we have proved is a product of commutators.
Moreover $E_{(x+1,x+1)}(pa_s+p^2a_s^2)=E_{(x+1,x+1)}(p(*))$ is a product of commutators by assumption. Also $E_{(x+1,x)}\big(-pa_s(1+pa_s+p^2a^2_s)^{-1}\big)=E_{(x+1,x)}(p(*))$ is a product of commutators, $E_{(x,x+1)}\big(p^2a_s^2$ $(1+pa_s+p^2a^2_s)^{-1}\big)=E_{(x,x+1)}(p^2(*))$ is a product of commutators. So $E_{(x,x)}(pa_s)$ is a product of commutators and hence
$E_{(x,x)}(p(*))$ is a product of commutators. So all diagonal matrices in $\mcl{P}_{\ul{\gl}}$ are product of commutators. 

Now every matrix in $\mcl{C}_{\ul{\gl}}$ is a product of elementary matrices $E_{(x,y)}(p^la_s)$ for $s=(i,i,j,m,n)\in S\bs T$. This can be observed as follows. We consider a matrix $g^a\in \mcl{C}_{\ul{\gl}}$ and reduce the diagonal blocks of $g^a$ to identity matrices using elementary matrices $E_{(x,y)}(p^la_s)$ with $s=(l,i,j,m,n)\in S\bs T$ and $Pos(s)$ occurs in the diagonal blocks only. Then we can reduce $g^a$ further to identity matrix in $\mcl{C}_{\ul{\gl}}$ using $E_{(x,y)}(p^la_s)$ for $s=(i,i,j,m,n)\in S\bs T$ with $Pos(s)$ occuring in the nondiagonal blocks. 
Hence we have proved that $\mcl{C}_{\ul{\gl}}\subseteq \mcl{P}'_{\ul{\gl}}$. This proves the converse in $(2)$.

$(3)$ is an immediate consequence of $(2)$. This completes the proof of Theorem~\ref{theorem:Commutator}.
\end{proof} 
\begin{remark}
	In the proof of Theorem~\ref{theorem:Commutator}, $p$ can be any prime. 
\end{remark}
\subsubsection{\bf{Modified Total Order on the Set $S$}}
Consider the two totally ordered subsets $(T,\leq_{TO})$ and $(S\bs T,\leq_{TO})$ of the set $(S,\leq_{TO})$ where $T$ is as defined in Equation~\ref{Eq:Abelianization}. The modified total order $\leq_{MTO}$ on the set $S$ is defined as follows. Let $s,s'\in S$. We say 
$s\leq_{MTO}s'$ if  
\begin{itemize}
	\item $s,s'\in S\bs T$ and $s\leq_{TO}s'$ or 
	\item $s,s'\in T$ and $s\leq_{TO}s'$ or  
	\item $s\in S\bs T,s'\in T$.
\end{itemize}
\begin{remark}
	The modified total order $\leq_{MTO}$ is introduced on the set $S$, so that the elements of $T\subs S$ become the larger elements of $S$. This is going to be useful later because for $t\in T,g^a\in \mcl{P}'_{\ul{\gl}},a_t=0$. 
\end{remark}
Now we define another chain of subgroups of $\mcl{N}_{\ul{\gl}}$ indexed by the set $S$. Define for $s\in S$
\equ{\mcl{N}^{s}_{\ul{\gl}}=\{g^a\in \mcl{P}_{\ul{\gl}}\mid a_{s'}=0 \text{ for all }s'\in S \text{ and }s<_{MTO}s'\}.}

\begin{theorem}
	\label{theorem:ModifiedChain}
	We have
	\begin{enumerate}
		\item $\mcl{N}^{s}_{\ul{\gl}}$ is a normal subgroup of $\mcl{P}_{\ul{\gl}}$ for $s\in S$.
		\item \equ{\mcl{N}^{s}_{\ul{\gl}}\sbnq \mcl{N}^{s'}_{\ul{\gl}} \Llra s<_{MTO}s'.}
		\item If $s'$ is the successor element of $s$ in $(S,\leq_{MTO})$  then  
		\equ{\frac{\mcl{N}^{s'}_{\ul{\gl}}}{\mcl{N}^{s}_{\ul{\gl}}} \cong \Z/p\Z.}
		\item If $s'$ is the successor element of $s$ in $(S,\leq_{MTO})$ then there is a central extension 
		\equan{ModifiedCentralExtension}{0\lra \frac{\mcl{N}^{s'}_{\ul{\gl}}}{\mcl{N}^{s}_{\ul{\gl}}} \lra \frac{\mcl{P}_{\ul{\gl}}}{\mcl{N}^{s}_{\ul{\gl}}}\lra \frac{\mcl{P}_{\ul{\gl}}}{\mcl{N}^{s'}_{\ul{\gl}}}\lra 0.} 
	\end{enumerate}
\end{theorem}
\begin{proof}
We prove $(1)$.	Let $T_s=\{t\in T\mid t\leq_{TO} s\}$. Let $S_s=\{t\in S\mid t\leq_{TO}s\}$. Then we have $T_s\subs S_s$,
\equ{\{t\in S\mid t\leq_{MTO}s\}=S_s\bs T_s.}
Now we conclude that for $s\in S\bs T, \mcl{N}^s_{\ul{\gl}}=\mcl{P}^s_{\ul{\gl}}\cap \mcl{P}'_{\ul{\gl}}$ using Theorem~\ref{theorem:Commutator}(2). Hence it is a normal subgroup being the intersection of two normal subgroups. If $s_0$ is the maximal element of $S\bs T$ then $\mcl{N}^{s_0}_{\ul{\gl}}=\mcl{P}'_{\ul{\gl}}$. So $\frac{\mcl{P}_{\ul{\gl}}}{\mcl{N}^{s_0}_{\ul{\gl}}}$ is abelian. Hence if $t'\in T$ and $t$ is its predecessor in $(S,\leq_{MTO})$ then $\frac{\mcl{N}^{t'}_{\ul{\gl}}}{\mcl{N}^t_{\ul{\gl}}}$ is a subgroup of the abelian group $\frac{\mcl{P}_{\ul{\gl}}}{\mcl{N}^t_{\ul{\gl}}}\cong\frac{\frac{\mcl{P}_{\ul{\gl}}}{\mcl{N}^{s_0}_{\ul{\gl}}}}{\frac{\mcl{N}^t_{\ul{\gl}}}{\mcl{N}^{s_0}_{\ul{\gl}}}}$ and hence $\frac{\mcl{N}^{t'}_{\ul{\gl}}}{\mcl{N}^t_{\ul{\gl}}}$ is normal in $\frac{\mcl{P}_{\ul{\gl}}}{\mcl{N}^t_{\ul{\gl}}}$. So we conclude that $\mcl{N}^{t'}_{\ul{\gl}}$ is normal in $\mcl{P}_{\ul{\gl}}$ for all $t'\in T$. This proves $(1)$.

$(2)$ is clear. For $t\leq_{MTO} t', t'$ the successor of $t$, the group $\mcl{N}^{t}_{\ul{\gl}}$ has index $p$ in $\mcl{N}^{t'}_{\ul{\gl}}$. So $(3)$ follows.  

We prove $(4)$. If $s'\in T$ then the proof is clear. If $s,s'\in S\bs T$ and $s'$ is the successor of $s$ in $(S,\leq_{MTO})$ then for all $\ti{s}\in S$, such that $s<_{TO} \ti{s}<_{TO} s'$ we have $\ti{s}\in T$.
Let $g\in \mcl{P}_{\ul{\gl}},h\in \mcl{N}^{s'}_{\ul{\gl}}\subseteq \mcl{P}^{s'}_{\ul{\gl}}$. Let $t\leq_{TO} s'$ be the predecessor of $s'$ in $(S,\leq_{TO})$. Then using Theorem~\ref{theorem:Chain}(3), we have $\frac{\mcl{P}^{s'}_{\ul{\gl}}}{\mcl{P}^{t}_{\ul{\gl}}}$ is a central subgroup of $\frac{\mcl{P}_{\ul{\gl}}}{\mcl{P}^{t}_{\ul{\gl}}}$. So the cosets $g\mcl{P}^{t}_{\ul{\gl}}, h\mcl{P}^t_{\ul{\gl}}$ commute, that is, $ghg^{-1}h^{-1}\in \mcl{P}^{t}_{\ul{\gl}}$ which implies that if $g^a=ghg^{-1}h^{-1}$ then $a_{s'}=0$.
We also have $ghg^{-1}h^{-1}\in \mcl{P}'_{\ul{\gl}}$. So $a_{\ti{s}}=0$ for all $s<_{TO}\ti{s}<_{TO}s'$ using Theorem~\ref{theorem:Commutator}(2). This implies that 
$ghg^{-1}h^{-1}\in \mcl{P}^s_{\ul{\gl}}$. Being a commutator $ghg^{-1}h^{-1}\in \mcl{P}^s_{\ul{\gl}}\cap \mcl{P}'_{\ul{\gl}}$, that is, $ghg^{-1}h^{-1}\in \mcl{N}^s_{\ul{\gl}}$ for all $g\in \mcl{P}_{\ul{\gl}},h\in \mcl{N}^{s'}_{\ul{\gl}}$. Hence the cosets $g\mcl{N}^s_{\ul{\gl}},h\mcl{N}^s_{\ul{\gl}}$ commute. Consequentially we have 
$\frac{\mcl{N}^{s'}_{\ul{\gl}}}{\mcl{N}^s_{\ul{\gl}}}$ is a central subgroup of $\frac{\mcl{P}_{\ul{\gl}}}{\mcl{N}^s_{\ul{\gl}}}$. This proves $(4)$. 

Hence the theorem follows.
\end{proof}
\begin{remark}
	In the proof of Theorem~\ref{theorem:ModifiedChain}, $p$ can be any prime. 
\end{remark}
\subsubsection{\bf{The Action of the Restricted Diagonal Subgroup on the Cohomology}}
Let $\mcl{D}_{\ul{\gl}}\subs \autgp$ be the subgroup of diagonal matrices whose orders divide $p-1$.
Then we have 
\begin{itemize}
	\item $\mcl{D}_{\ul{\gl}}\cap \mcl{P}_{\ul{\gl}}=\{1\}$, the trivial subgroup,
	\item $\mcl{D}_{\ul{\gl}}$ normalizes $\mcl{P}_{\ul{\gl}}$, that is, $\mcl{D}_{\ul{\gl}}\subs N_{\autgp}(\mcl{P}_{\ul{\gl}})$.
\end{itemize}
As mentioned in Remark~\ref{remark:TrivialAction} the conjugation action of $\mcl{D}_{\ul{\gl}}$ on $\mcl{P}_{\ul{\gl}}$ gives rise to an action of $\mcl{D}_{\ul{\gl}}$ on the spaces $H^2_{Trivial\ Action}(\frac{\mcl{N}^{s'}_{\ul{\gl}}}{\mcl{N}^{s}_{\ul{\gl}}},\Z/p\Z)$ for $s\leq_{MTO}s'$ and 
$H^2_{Trivial\ Action}(\frac{\mcl{P}_{\ul{\gl}}}{\mcl{N}^{s}_{\ul{\gl}}},\Z/p\Z)$ for $s\in S$. Since these are vector spaces over $\Z/p\Z$, they are semisimple representations of $\mcl{D}_{\ul{\gl}}$ since $p$ does not divide the cardinality of $\mcl{D}_{\ul{\gl}}$. Using Theorem~\ref{theorem:StabilityConditions}, in order to prove $H^2_{Trivial\ Action}(\autgp,\Z/p\Z)$ is zero it is enough to prove that the trivial subrepresentation \equ{H^2_{Trivial\ Action}(\mcl{P}_{\ul{\gl}},\Z/p\Z)^{\mcl{D}_{\ul{\gl}}}\subseteq H^2_{Trivial\ Action}(\mcl{P}_{\ul{\gl}},\Z/p\Z)} is zero.

\subsubsection{\bf{The Repeated Use of the Extended Hochschild-Serre Exact Sequence}}
To compute the space $H^2_{Trivial\ Action}(\mcl{P}_{\ul{\gl}},\Z/p\Z)^{\mcl{D}_{\ul{\gl}}}$ we use the extended Hochschild-Serre exact sequences for the central extensions given in~\ref{Eq:ModifiedCentralExtension}. Actually we use the following part of the exact sequence, for any $s,s'\in S\bs T$ such that $s'$ is the successor element of $s$ in $(S,\leq_{MTO})$.
\equa{H^2_{Trivial\ Action}(\frac{\mcl{P}_{\ul{\gl}}}{\mcl{N}^{s'}_{\ul{\gl}}},\Z/p\Z)&\os{Inf}{\lra}H^2_{Trivial\ Action}(\frac{\mcl{P}_{\ul{\gl}}}{\mcl{N}^s_{\ul{\gl}}},\Z/p\Z)\\&\os{\gt=Res\times \gth}{\lra}H^2_{Trivial\ Action}(\frac{\mcl{N}^{s'}_{\ul{\gl}}}{\mcl{N}^s_{\ul{\gl}}},\Z/p\Z)\times P(\frac{\mcl{P}_{\ul{\gl}}}{\mcl{N}^s_{\ul{\gl}}},\frac{\mcl{N}^{s'}_{\ul{\gl}}}{\mcl{N}^s_{\ul{\gl}}},\Z/p\Z).}
\begin{theorem}
	\label{theorem:KernelRes}
Let $p$ be an odd prime and let $s,s'\in S$ be such that $s'$ is the successor element of $s$ in $(S,\leq_{MTO})$. Consider the restriction map $H^2_{Trivial\ Action}(\frac{\mcl{P}_{\ul{\gl}}}{\mcl{N}^s_{\ul{\gl}}},\Z/p\Z) \os{Res}{\lra} H^2_{Trivial\ Action}(\frac{\mcl{N}^{s'}_{\ul{\gl}}}{\mcl{N}^s_{\ul{\gl}}},\Z/p\Z)$. Then we have 
\equ{H^2_{Trivial\ Action}(\frac{\mcl{P}_{\ul{\gl}}}{\mcl{N}^s_{\ul{\gl}}},\Z/p\Z)^{\mcl{D}_{\ul{\gl}}}\subseteq \Ker(Res).}
\end{theorem}
\begin{proof}
If $Pos(s')=(x,y)$ and $1\leq x\neq y\leq \gr=\gr_1+\cdots+\gr_k$ then we have $H^2_{Trivial\ Action}(\frac{\mcl{N}^{s'}_{\ul{\gl}}}{\mcl{N}^s_{\ul{\gl}}},\Z/p\Z)^{\mcl{D}_{\ul{\gl}}}=0$ using Theorem~\ref{theorem:BasicMultiple} since $p$ is odd. Hence \equ{Res(H^2_{Trivial\ Action}(\frac{\mcl{P}_{\ul{\gl}}}{\mcl{N}^s_{\ul{\gl}}},\Z/p\Z)^{\mcl{D}_{\ul{\gl}}})\subseteq H^2_{Trivial\ Action}(\frac{\mcl{N}^{s'}_{\ul{\gl}}}{\mcl{N}^s_{\ul{\gl}}},\Z/p\Z)^{\mcl{D}_{\ul{\gl}}}=0.}
If $Pos(s')=(x',x')$ for some $1\leq x'\leq \gr=\gr_1+\cdots+\gr_k$ where $s'=(l',i',i',m',m')$ then we have $l'>0 \Ra m'<k,1\leq x\leq \gr_1+\cdots+\gr_{k-1}$. So position $(x',x'+1)$ is in block $(m',m')$ or $(m',m'+1)$, position $(x'+1,x')$ is in block $(m',m')$ or $(m'+1,m')$, position $(x'+1,x'+1)$ is in block $(m',m')$ or $(m'+1,m'+1)$. 
Let $t\in S$ be such that $Pos(t)=(x',x'+1)$ and the first coordinate of $t$ is $l$ which is the first coordinate of $s'$. Then $s'<_{MTO}t,s'<_{TO}t$. Let the symbol $A$ denote the entry in the position $(x'+1,x')$ of any typical matrix in $\mcl{P}_{\ul{\gl}}$. Consider the subgroup \equ{\mcl{C}^{s'}_{\ul{\gl}}=\{g^a+e_{(x',x'+1)}(p^{l'}\ti{a}_t)+e_{(x'+1,x')}(A)\mid 0\leq \ti{a}_t\leq p-1,g^a\in \mcl{N}^{s'}_{\ul{\gl}}\}\subs \mcl{P}_{\ul{\gl}}.}
This is indeed a subgroup, because if $g_1=g^a+e_{(x',x'+1)}(p^{l'}\ti{a}_t)+e_{(x'+1,x')}(A),g_2=g^b+e_{(x',x'+1)}(p^{l'}\ti{b}_t)+e_{(x'+1,x')}(B)$ be two elements in $\mcl{C}^{s'}_{\ul{\gl}}$. Then we have $g_1g_2\in \mcl{C}^{s'}_{\ul{\gl}}$. We can see this as follows.
\equa{g_1g_2&=g^ag^b+g^ae_{(x',x'+1)}(p^{l'}\ti{b}_t)+g^ae_{(x'+1,x')}(B)\\&+e_{(x',x'+1)}(p^{l'}\ti{a}_t)g^b+e_{(x',x'+1)}(p^{l'}\ti{a}_t)e_{(x',x'+1)}(p^{l'}\ti{b}_t)+e_{(x',x'+1)}(p^{l'}\ti{a}_t)e_{(x'+1,x')}(B)\\
&+e_{(x'+1,x')}(A)g^b+e_{(x'+1,x')}(A)e_{(x',x'+1)}(p^{l'}\ti{b}_t)+e_{(x'+1,x')}(A)e_{(x'+1,x')}(B)\\
&=g^c+e_{(x',x'+1)}(p^{l'}\ti{c}_t)+e_{(x'+1,x')}(C),}
where $g^c\in \mcl{N}^{s'}_{\ul{\gl}},c_{s'}\equiv a_{s'}+b_{s'}+\ti{a}_tB\mod p,\ti{c}_t\equiv\ti{a}_t+\ti{b}_t\mod p,C\equiv A+B\mod p^{l'}$.
Since the nonempty finite subset $\mcl{C}^{s'}_{\ul{\gl}}$ is closed under group multiplication, it is also a subgroup. 
Now $s\in S\bs T$ is the predecessor of $s'$. We have $\mcl{N}^{s}_{\ul{\gl}}\trianglelefteq \mcl{C}^{s'}_{\ul{\gl}}$. The group multiplication in $\frac{\mcl{C}^{s'}_{\ul{\gl}}}{\mcl{N}^{s}_{\ul{\gl}}}$ is given in terms of $2\times 2$ matrices as follows. 
 \equa{&\mattwo{1+p^{l'}a_{s'}\mod p^{l'+1}}{p^l\ti{a}_{t}\mod p^{l'+1}}{A\mod p^l}{1\mod p^{l'}}\mattwo{1+p^{l'}b_{s'}\mod p^{l'+1}}{p^{l'}\ti{b}_{t}\mod p^{l'+1}}{B\mod p^{l'}}{1\mod p^{l'}}\\&=\mattwo{1+p^{l'}(a_{s'}+b_{s'}+\ti{a}_tB)\mod p^{l'+1}}{p^{l'}(\ti{a}_{t}+\ti{b}_{t})\mod p^{l'+1}}{A+B\mod p^{l'}}{1\mod p^{l'}}.}
 We consider the normal abelian subgroup $\mcl{A}^{s'}_{\ul{\gl}}\trianglelefteq\frac{\mcl{C}^{s'}_{\ul{\gl}}}{\mcl{N}^{s}_{\ul{\gl}}}$ of upper triangular $2\times 2$ matrices in $\frac{\mcl{C}^{s'}_{\ul{\gl}}}{\mcl{N}^{s}_{\ul{\gl}}}$, that is,
  
 \equa{\mcl{A}^{s'}_{\ul{\gl}}&=\bigg\{\mattwo{1+p^{l'}a_{s'}\mod p^{l'+1}}{p^{l'}\ti{a}_{t}\mod p^{l'+1}}{0\mod p^{l'}}{1\mod p^{l'}}\mid 0\leq a_{s'},\ti{a}_t\leq p-1\bigg\}\\&\cong \Z/p\Z\times \Z/p\Z.} 
 Now consider the map $Res$ as the following composition of maps $Res_2,Res_1$, that is, $Res=Res_2\circ Res_1$ where,
 \equa{H^2_{Trivial\ Action}(\frac{\mcl{P}_{\ul{\gl}}}{\mcl{N}^s_{\ul{\gl}}},\Z/p\Z) &\os{Res_1}{\lra} H^2_{Trivial\ Action}(\mcl{A}^{s'}_{\ul{\gl}},\Z/p\Z)\\ &\os{Res_2}{\lra} H^2_{Trivial\ Action}(\frac{\mcl{N}^{s'}_{\ul{\gl}}}{\mcl{N}^s_{\ul{\gl}}},\Z/p\Z).}
 To prove the theorem it is enough to show that 
\equ{Res_1(H^2_{Trivial\ Action}(\frac{\mcl{P}_{\ul{\gl}}}{\mcl{N}^s_{\ul{\gl}}},\Z/p\Z)^{\mcl{D}_{\ul{\gl}}})=0.}
Let $[c]=Res_1([d])$ where $[d]\in H^2_{Trivial\ Action}(\frac{\mcl{P}_{\ul{\gl}}}{\mcl{N}^s_{\ul{\gl}}},\Z/p\Z)^{\mcl{D}_{\ul{\gl}}}$ then cohomology class $[d]$ and hence $[c]$ does not change under conjugation action induced by the elementary invertible matrix\equ{\ol{E}_{(x'+1,x')}(C)=\mattwo{1 \mod p^{l'+1}}{0\mod p^{l'+1}}{C\mod p^{l'}}{1\mod p^{l'}}\in \frac{\mcl{C}^{s'}_{\ul{\gl}}}{\mcl{N}^{s}_{\ul{\gl}}}}
using Remark~\ref{remark:TrivialAction}.
We observe that the conjugation on $\mcl{A}^{s'}_{\ul{\gl}}$ is given as follows.
\equa{\ol{E}_{(x'+1,x')}(C)&\mattwo{1+p^{l'}a_{s'}\mod p^{l'+1}}{p^{l'}\ti{a}_{t}\mod p^{{l'}+1}}{0\mod p^{l'}}{1\mod p^{l'}}\ol{E}_{(x'+1,x')}(-C)=\\&
\mattwo{1+p^{l'}(a_{s'}-\ti{a}_tC)\mod p^{l'+1}}{p^{l'}\ti{a}_t\mod p^{l'+1}}{0\mod p^{l'}}{1\mod p^{l'}}.}
We also observe that cohomology class $[c]\in H^2_{Trivial\ Action}(\mcl{A}^{s'}_{\ul{\gl}},\Z/p\Z)^{\mcl{D}_{\ul{\gl}}}$.
We simplify the notation by denoting elements in $\mcl{A}^{s'}_{\ul{\gl}}=\{(a_{s'},\ti{a}_t)\in \Z/p\Z\times \Z/p\Z\}$. 
Using Theorem~\ref{theorem:ProductCohomology}, we have 
$c$ is cohomologous to $c_1+c_2+\ti{\gb}$ where $c_1$ and $c_2$ are cocycles obtained by restriction of the cocycle $c$ to each of the components of $\mcl{A}^{s'}_{\ul{\gl}}$ and $\ti{\gb}:\Z/p\Z\times \Z/p\Z\lra \Z/p\Z$ is a bilinear map which in particular is given as: For all $a_{s'},\ti{a}_t,b_{s'},\ti{b}_t\in \Z/p\Z$, $\ti{\gb}(a_s,\ti{b}_t)=\gb a_s\ti{b}_t$,
\equ{c((a_{s'},\ti{a}_t),(b_{s'},\ti{b}_t))\approx_{cohomologous} c_1(a_{s'},b_{s'})+c_2(\ti{a}_t,\ti{b}_t)+\gb a_{s'}\ti{b}_t}
for some $\gb\in \Z/p\Z$. Now the action of $\mcl{D}_{\ul{\gl}}$ acts nontrivially on $c_2$ since $t$ is in a nondiagonal position of any matrix in $\mcl{P}_{\ul{\gl}}$, where as $\mcl{D}_{\ul{\gl}}$ acts trivially on $c_1$ since $s'$ is in a diagonal position. Using Theorem~\ref{theorem:BasicMultiple}, since $p$ is odd, we have $c_2$ is cohomologous to zero. Now the map $\ti{\gb}$ depends only on the cohomology class of $[c]$ using Theorem~\ref{theorem:ProductCohomology}. Hence we have by using the invariance of $[c]$ under the action of $\mcl{D}_{\ul{\gl}}$, $\gb a_{s'}\ti{b}_t=\gb \gga a_{s'}\ti{b}_t$ for any $\gga \in (\Z/p\Z)^*$. Now we conclude that $\gb=0\Ra \ti{\gb}=0$ by choosing $ \gga\in (\Z/p\Z)^*,\gga\neq 1$ since $p$ is odd. Hence 
\equ{c((a_{s'},\ti{a}_t),(b_{s'},\ti{b}_t))\approx_{cohomologous} c_1(a_{s'},b_{s'}).}
We use conjugation by $\ol{E}_{(x'+1,x')}(C)$ for $C\not\equiv 0\mod p$ and conclude that 
\equ{c((a_{s'},\ti{a}_t),(b_{s'},\ti{b}_t))\approx_{cohomologous}c_1(a_{s'},b_{s'})\approx_{cohomologous} c_1(a_{s'}-\ti{a}_tC,b_{s'}-\ti{b}_tC),}
that is, $c_1(a_{s'}-\ti{a}_tC,b_{s'}-\ti{b}_tC)-c_1(a_{s'},b_{s'})=v(a_{s'}+b_{s'},\ti{a}_t+\ti{b}_t)-v(a_{s'},\ti{a}_t)-v(b_{s'},\ti{b}_t)$ for some $1$-cochain $v:\mcl{A}^{s'}_{\ul{\gl}}\lra \Z/p\Z$. 
Putting $a_{s'}=0=b_{s'},C=-1$ we conclude that  $c_1(\ti{a}_t,\ti{b}_t)$ is cohomologous to zero.
Hence we conclude that $[c]=0$. 

This proves the theorem.
\end{proof}
\begin{remark}
In the proof of Theorem~\ref{theorem:KernelRes} we only require that $p$ is an odd prime and $p$ can be equal to $3$.
\end{remark}
\begin{theorem}
	\label{theorem:GL21}
	Let $p$ be an odd prime.
	Consider the following $p$-group of order $p^2$.
	\equ{G_1=\{\mattwo{}{b}{c}{}\mid b,c\in \Z/p\Z\}}
	where the group multiplication is given by:
	\equ{\mattwo{}{b}{c}{}\mattwo{}{b'}{c'}{}=\mattwo{}{b+b'}{c+c'}{}.}
	Consider the action of the group $D=(\Z/p\Z)^*\times (\Z/p\Z)^*$ on $G_1$ as follows. For $t_i\in (\Z/p\Z)^*,i=1,2$,
	\equ{(t_1,t_2)\bullet\mattwo{}{b}{c}{}=\mattwo{}{\frac{t_1}{t_2}b}{\frac{t_2}{t_1}c}{}.}
	Then the subgroup of $D$-invariant cohomology classes of $H^2_{Trivial\ Action}(G_1,\Z/p\Z)$ is 
	\equ{H^2_{Trivial\ Action}(G_1,\Z/p\Z)^{D}=\{\gb [z]\mid \gb\in \Z/p\Z\}\cong \Z/p\Z}
	where $z\in Z^2_{Trivial\ Action}(G_1,\Z/p\Z)$ is defined as:
	\equ{z\bigg(\mattwo{}{b}{c}{},\mattwo{}{b'}{c'}{}\bigg)=bc'.} 
\end{theorem}
\begin{proof}
The theorem follows from an application of Theorem~\ref{theorem:ProductCohomology} and Theorem~\ref{theorem:BasicMultiple} and the fact that $p$ is odd.
\end{proof}
\begin{remark}
In the proof of Theorem~\ref{theorem:GL21} we only require that $p$ is an odd prime and $p$ can be equal to $3$.
\end{remark}
\begin{theorem}
	\label{theorem:GL21Second}
Let $\ul{\gl}=(2>1)$ and $p$ be an odd prime. Then \equ{H^2_{Trivial\ Action}(\autgp,\Z/p\Z)=0.}
\end{theorem}
\begin{proof}
We show that $H^2_{Trivial\ Action}(\mcl{P}_{\ul{\gl}},\Z/p\Z)^{\mcl{D}_{\ul{\gl}}}=0$.
The group multiplication in $\mcl{P}_{\ul{\gl}}$ is given as follows.
\equ{\mattwo {1+a p}{b p}{c}{1}\mattwo{1+a'p}{b'p}{c'}{1}=\mattwo{1+(a+a'+bc')p}{(b+b')p}{c+c'}{1}.}
Let $[y]\in H^2_{Trivial\ Action}(\mcl{P}_{\ul{\gl}},\Z/p\Z)^{\mcl{D}_{\ul{\gl}}}$ and consider $Res:H^2_{Trivial\ Action}(\mcl{P}_{\ul{\gl}},\Z/p\Z)$ $\lra H^2_{Trivial\ Action}(Z(\mcl{P}_{\ul{\gl}}),\Z/p\Z)$. We observe using Theorem~\ref{theorem:KernelRes} where we choose $s=\es,s'=(1,1,1,1,1)$, we get $[y]\in \Ker(Res)$. Consider the map $\gth:H^2_{Trivial\ Action}(\mcl{P}_{\ul{\gl}},\Z/p\Z)\lra P(\mcl{P}_{\ul{\gl}},Z(\mcl{P}_{\ul{\gl}}),\Z/p\Z)$. The class $[y]\in \Ker(\gth)$. This is because $[y]\in H^2_{Trivial\ Action}(\mcl{P}_{\ul{\gl}},\Z/p\Z)^{\mcl{D}_{\ul{\gl}}}$ and the positions of $a,b$ are not mutually symmetric about the diagonal and the positions of $a,c$ are also not mutually symmetric about the diagonal. 

Using the exactness of the sequence 
\equa{H^2_{Trivial\ Action}(&\frac{\mcl{P}_{\ul{\gl}}}{Z(\mcl{P}_{\ul{\gl}})},\Z/p\Z)\os{Inf}{\lra} H^2_{Trivial\ Action}(\mcl{P}_{\ul{\gl}},\Z/p\Z)\\&\os{Res\times \gth}{\lra}  H^2_{Trivial\ Action}(Z(\mcl{P}_{\ul{\gl}}),\Z/p\Z) \times P(\mcl{P}_{\ul{\gl}},Z(\mcl{P}_{\ul{\gl}}),\Z/p\Z),}
we get that there exists $[x]\in H^2_{Trivial\ Action}(\frac{\mcl{P}_{\ul{\gl}}}{Z(\mcl{P}_{\ul{\gl}})},\Z/p\Z)$ such that $Inf([x])=[y]$.

So we have by averaging $[x]$ with respect to the action of the group $\mcl{D}_{\ul{\gl}}$
\equ{Inf\bigg(\frac{1}{\mid \mcl{D}_{\ul{\gl}}\mid}\us{t\in \mcl{D}_{\ul{\gl}}}{\sum}t\bullet[x]\bigg)=[y]} and the cohomology class
\equ{\frac{1}{\mid \mcl{D}_{\ul{\gl}}\mid}\us{t\in \mcl{D}_{\ul{\gl}}}{\sum}t\bullet[x]\in H^2_{Trivial\ Action}(\frac{\mcl{P}_{\ul{\gl}}}{Z(\mcl{P}_{\ul{\gl}})},\Z/p\Z)^{\mcl{D}_{\ul{\gl}}}.}

Now the group $\frac{\mcl{P}_{\ul{\gl}}}{Z(\mcl{P}_{\ul{\gl}})} \cong G_1$ (defined in Theorem~\ref{theorem:GL21}) and the action of $\mcl{D}_{\ul{\gl}}$ on $\frac{\mcl{P}_{\ul{\gl}}}{Z(\mcl{P}_{\ul{\gl}})}$ and the action of $D$ in Theorem~\ref{theorem:GL21} on $G_1$ are compatible. Hence using Theorem~\ref{theorem:GL21} we get that $[y]=\gb [z]$ for some $\gb \in \Z/p\Z$ where 
\equ{z\bigg(\mattwo {1+a p}{b p}{c}{1},\mattwo{1+a'p}{b'p}{c'}{1}\bigg)=bc'.}
But we observe that $z$ is a 2-coboundary because if $v:\mcl{P}_{\ul{\gl}}\lra \Z/p\Z$ is a $1$-cochain defined as: 
$v(\mattwo {1+a p}{b p}{c}{1})=a$ then we have $v(gg')-v(g)-v(g')=bc'$ where $g=\mattwo {1+a p}{b p}{c}{1},g'=\mattwo {1+a' p}{b' p}{c'}{1}$. 

Hence we get $[y]=0$ and so $H^2_{Trivial\ Action}(\mcl{P}_{\ul{\gl}},\Z/p\Z)^{\mcl{D}_{\ul{\gl}}}=0$. This proves the theorem.
\end{proof}
\begin{remark}
	In the proof of Theorem~\ref{theorem:GL21Second} we only require that $p$ is an odd prime and $p$ can be equal to $3$.
\end{remark}
\begin{theorem}
	\label{theorem:GL211First}
Let $p$ be an odd prime. Consider the following $p$-group of order $p^4$.
\equ{G_2=\bigg\{\matthree{}{}{b}{c}{}{}{e}{f}{}\mid b,c,e,f\in \Z/p\Z\bigg\}}
where the group multiplication is given by: 
\equ{\matthree{}{}{b}{c}{}{}{e}{f}{}\matthree{}{}{b'}{c'}{}{}{e'}{f'}{}=\matthree{}{}{b+b'}{c+c'}{}{}{e+e'+fc'}{f+f'}{}.}
Consider the action of the group $D=(\Z/p\Z)^*\times (\Z/p\Z)^*\times (\Z/p\Z)^*$ on $G_2$ as follows. For $t_i\in (\Z/p\Z)^*,i=1,2,3$,
\equ{(t_1,t_2,t_3)\bullet \matthree{}{}{b}{c}{}{}{e}{f}{}=\matthree{}{}{\frac{t_1}{t_3}b}{\frac{t_2}{t_1}c}{}{}{\frac{t_3}{t_1}e}{\frac{t_3}{t_2}f}{}.}
Then the subgroup of $D$-invariant cohomology classes of $H^2_{Trivial\ Action}(G_2,\Z/p\Z)$ is trivial, that is,
\equ{H^2_{Trivial\ Action}(G_2,\Z/p\Z)^{D}=0.}
\end{theorem}
\begin{proof}
Let $[y]\in H^2_{Trivial\ Action}(G_2,\Z/p\Z)^{D}$. Consider the central subgroup $Z=\{\matthree{}{}{0}{0}{}{}{e}{0}{}\mid e\in \Z/p\Z\}$ and the map $\gth:H^2_{Trivial\ Action}(G_2,\Z/p\Z)\lra P(G_2,Z,\Z/p\Z)$ as defined in Proposition~\ref{prop:ThetaMap}. We have 
\equ{\gth([y])\bigg(\matthree{}{}{b}{c}{}{}{e}{f}{},\matthree{}{}{0}{0}{}{}{e'}{0}{}\bigg)=\ga b e'+\gb ce'+\gga fe'}
for some  $\ga,\gb,\gga\in \Z/p\Z$. By the invariance of $[y]$ under the action of the group $D$ we get that $\gb=\gga=0$.
Hence 
\equan{GL211First}{\gth([y])\bigg(\matthree{}{}{b}{c}{}{}{e}{f}{},\matthree{}{}{0}{0}{}{}{e'}{0}{}\bigg)=\ga b e'.}
Consider the normal abelian subgroup $H=\{\matthree{}{}{b}{c}{}{}{e}{0}{}\mid b,c e\in \Z/p\Z\}$ and the map $Res:H^2_{Trivial\ Action}(G_2,\Z/p\Z)\lra H^2_{Trivial\ Action}(H,\Z/p\Z)$.
We have by using Corollary~\ref{cor:ProductCohomology}
\equa{Res([y])\bigg(\matthree{}{}{b}{c}{}{}{e}{0}{},\matthree{}{}{b'}{c'}{}{}{e'}{0}{}\bigg)&\approx_{cohomologous} s_1(b,b')+s_2(c,c')+s_3(e,e')\\&+\ga'be'+\gb'ce'}
for some $\ga',\gb'\in \Z/p\Z$, where $s_i,i=1,2,3$ are the restrictions of $y$ to the respective components. Now using the invariance of $[y]$ under the action of $D$ and using Theorem~\ref{theorem:BasicMultiple}, since $p$ is odd, we conclude that $s_i$ are cohomologous to zero for $i=1,2,3$ and $\gb'=0$. So we have 
\equ{Res([y])\bigg(\matthree{}{}{b}{c}{}{}{e}{0}{},\matthree{}{}{b'}{c'}{}{}{e'}{0}{}\bigg)\approx_{cohomologous} \ga'be'.}
Now consider the map $\gth_0:H^2_{Trivial\ Action}(H,\Z/p\Z)\lra P(H,K,\Z/p\Z)$ as defined in Proposition~\ref{prop:ThetaMap}, where $K=\{\matthree{}{}{0}{c}{}{}{e}{0}{}\mid c,e\in \Z/p\Z\}$. We have 
\equ{\gth_0(Res([y]))\bigg(\matthree{}{}{b}{c}{}{}{e}{0}{},\matthree{}{}{0}{c'}{}{}{e'}{0}{}\bigg)=\ga'be'.}
Substituting $c'=0$ in the above equation and substuting $f=0$ in Equation~\ref{Eq:GL211First} we get that $\ga'=\ga$.

Now $Res([y])$ is also invariant under the conjugation action of the element $\matthree{}{}{0}{0}{}{}{0}{\ti{f}}{}$.
We observe that 
\equ{\matthree{}{}{0}{0}{}{}{0}{\ti{f}}{}\matthree{}{}{b}{c}{}{}{e}{0}{}\matthree{}{}{0}{0}{}{}{0}{-\ti{f}}{}=\matthree{}{}{b}{c}{}{}{e+\ti{f}c}{0}{}}
\equ{\matthree{}{}{0}{0}{}{}{0}{\ti{f}}{}\matthree{}{}{0}{c'}{}{}{e'}{0}{}\matthree{}{}{0}{0}{}{}{0}{-\ti{f}}{}=\matthree{}{}{0}{c'}{}{}{e'+\ti{f}c'}{0}{}}
So we get that 
\equ{\ga be'=\gth_0(Res([y]))\bigg(\matthree{}{}{b}{c}{}{}{e+\ti{f}c}{0}{},\matthree{}{}{0}{c'}{}{}{e'+\ti{f}c'}{0}{}\bigg)=\ga b(e'+\ti{f}c').}
This implies that $\ga=0$. Hence we get $\gth([y])=0$ in Equation~\ref{Eq:GL211First}. Now $[y]\in \Ker(\gth)$ and $[y]\in \Ker(Res: H^2_{Trivial\ Action}(G_2,\Z/p\Z)\lra H^2_{Trivial\ Action}(Z,\Z/p\Z))$ since $H^2_{Trivial\ Action}(Z,\Z/p\Z)^D=0$ and here $Res([y])\in H^2_{Trivial\ Action}(Z,\Z/p\Z)^D=0$.

Using the exactness of the sequence 
\equa{H^2_{Trivial\ Action}(&\frac{G_2}{Z},\Z/p\Z)\os{Inf}{\lra} H^2_{Trivial\ Action}(G_2,\Z/p\Z)\\&\os{Res\times \gth}{\lra}  H^2_{Trivial\ Action}(Z,\Z/p\Z) \times P(G_2,Z,\Z/p\Z),}
we get that there exists $[x]\in H^2_{Trivial\ Action}(\frac{G_2}{Z},\Z/p\Z)$ such that $Inf([x])=[y]$.

So we have by averaging $[x]$ with respect to the action of the group $D$
\equ{Inf\bigg(\frac{1}{\mid D\mid}\us{t\in D}{\sum}t\bullet[x]\bigg)=[y]} and the cohomology class
\equ{\frac{1}{\mid D\mid}\us{t\in D}{\sum}t\bullet[x]\in H^2_{Trivial\ Action}(\frac{G_2}{Z},\Z/p\Z)^{D}.}
But in $\frac{G_2}{Z}$ the positions of $b,c,f$ are nondiagonal and $D$ acts nontrivially. Also no two of $b,c,f$ are in mutually symmetric positions about the diagonal. Hence we conclude that  \equ{H^2_{Trivial\ Action}(\frac{G_2}{Z},\Z/p\Z)^{D}=0.} This implies that $[y]$ is the inflation of the zero cohomology class which implies $H^2_{Trivial\ Action}(G_2,\Z/p\Z)^{D}=0$. This proves the theorem.
\end{proof}
\begin{remark}
	In the proof of Theorem~\ref{theorem:GL211First} we only require that $p$ is an odd prime and $p$ can be equal to $3$.
\end{remark}
\begin{theorem}
	\label{theorem:GL211}
	Let $p$ be a prime. Consider the following $p$-group of order $p^5$.
	\equ{G_3=\bigg\{\matthree{}{a}{b}{c}{}{}{e}{f}{}\mid a,b,c,e,f\in \Z/p\Z\bigg\}}
	where the group multiplication is given by: 
	\equ{\matthree{}{a}{b}{c}{}{}{e}{f}{}\matthree{}{a'}{b'}{c'}{}{}{e'}{f'}{}=\matthree{}{a+a'+bf'}{b+b'}{c+c'}{}{}{e+e'+fc'}{f+f'}{}.}
	Consider the action of the group $D=(\Z/p\Z)^*\times (\Z/p\Z)^*\times (\Z/p\Z)^*$ on $G_3$ as follows. For $t_i\in (\Z/p\Z)^*,i=1,2,3$,
	\equ{(t_1,t_2,t_3)\bullet \matthree{}{a}{b}{c}{}{}{e}{f}{}=\matthree{}{\frac{t_1}{t_2}a}{\frac{t_1}{t_3}b}{\frac{t_2}{t_1}c}{}{}{\frac{t_3}{t_1}e}{\frac{t_3}{t_2}f}{}.}
	Then the subgroup of $D$-invariant cohomology classes of $H^2_{Trivial\ Action}(G_3,\Z/p\Z)$ is 
	\equ{H^2_{Trivial\ Action}(G_3,\Z/p\Z)^{D}=\{\gb [z]\mid \gb\in \Z/p\Z\}\cong \Z/p\Z}
	where $z\in Z^2_{Trivial\ Action}(G_3,\Z/p\Z)$ is defined as:
	\equ{z\bigg(\matthree{}{a}{b}{c}{}{}{e}{f}{},\matthree{}{a'}{b'}{c'}{}{}{e'}{f'}{}\bigg)=ac'+ be'.}
\end{theorem}
\begin{proof}
Clearly $z$ is a $2$-cocycle on $G_3$, that is, the cocycle condition is satisfied. This is because we have 
\equ{ (ac'+ be')+ (a+a'+bf')c''+ (b+b')e''= a'c''+b'e''+ a(c'+c'')+b(e'+e''+f'c'')}
for $a,b,c,e,f,a',b',c',e',f',a'',b'',c'',e'',f''\in \Z/p\Z$. We also observe that $[z]$ is invariant under the action of the group $D$. So $[z]\in H^2_{Trivial\ Action}(G_3,\Z/p\Z)^{D}$.
Let $[y]\in H^2_{Trivial\ Action}(G_3,\Z/p\Z)^{D}$. Consider the central subgroup $Z=\{\matthree{}{a}{0}{0}{}{}{0}{0}{}\mid a\in \Z/p\Z\}$ and the map $\gth:H^2_{Trivial\ Action}(G_3,\Z/p\Z)\lra P(G_3,Z,\Z/p\Z)$. We oberve that 
\equ{\gth([y])\bigg(\matthree{}{a}{b}{c}{}{}{e}{f}{},\matthree{}{a'}{0}{0}{}{}{0}{0}{}\bigg)=\gb ca'} for some $\gb\in \Z/p\Z$. We have
\equ{\gth([z])\bigg(\matthree{}{a}{b}{c}{}{}{e}{f}{},\matthree{}{a'}{0}{0}{}{}{0}{0}{}\bigg)=-ca'.}
So $[y+\gb z]\in \Ker(\gth)$. Also $[y+\gb z]\in \Ker(Res:H^2_{Trivial\ Action}(G_3,\Z/p\Z)\lra H^2_{Trivial\ Action}(Z,\Z/p\Z))$, since $H^2_{Trivial\ Action}(Z,\Z/p\Z)^D=0$ and $Res([y+\gb z])\in H^2_{Trivial\ Action}(Z,\Z/p\Z)^D=0$.   

Using the exact sequence 
\equa{H^2_{Trivial\ Action}(&\frac{G_3}{Z},\Z/p\Z)\os{Inf}{\lra} H^2_{Trivial\ Action}(G_3,\Z/p\Z)\\&\os{Res\times \gth}{\lra}  H^2_{Trivial\ Action}(Z,\Z/p\Z) \times P(G_3,Z,\Z/p\Z),}
we get that $[y]+[\gb z]=Inf([x])$ for some $[x]\in H^2_{Trivial\ Action}(\frac{G_3}{Z},\Z/p\Z)$. 
Now $[y]+[\gb z]\in H^2_{Trivial\ Action}(G_3,\Z/p\Z)^{D}$. So we have by averaging $[x]$ with respect to the action of the group $D$
\equ{Inf\bigg(\frac{1}{\mid D\mid}\us{t\in D}{\sum}t\bullet[x]\bigg)=[y]+[\gb z]} and the cohomology class
\equ{\frac{1}{\mid D\mid}\us{t\in D}{\sum}t\bullet[x]\in H^2_{Trivial\ Action}(\frac{G_3}{Z},\Z/p\Z)^{D}.}
We observe now that $\frac{G_3}{Z}\cong G_2$ and the actions of $D$ on $\frac{G_3}{Z}$ and $G_2$ are compatible.
Using Theorem~\ref{theorem:GL211} we get that  $H^2_{Trivial\ Action}(\frac{G_3}{Z},\Z/p\Z)^{D}=0$. Hence we have 
$[y]=-\gb [z]$.

Now the class $[z]$ is nonzero because $\gth([z])\neq 0$. Hence $H^2_{Trivial\ Action}(G_3,\Z/p\Z)^{D}$ $\cong \Z/p\Z$.
This proves the theorem.
\end{proof}
\begin{remark}
	In the proof of Theorem~\ref{theorem:GL211} we only require that $p$ is an odd prime and $p$ can be equal to $3$.
\end{remark}
\begin{theorem}
	\label{theorem:GL211Second}
Let $\ul{\gl}=(2^1>1^2)$ and $p$ be an odd prime. Then \equ{H^2_{Trivial\ Action}(\autgp,\Z/p\Z)=0.}
\end{theorem}
\begin{proof}
We show that $H^2_{Trivial\ Action}(\mcl{P}_{\ul{\gl}},\Z/p\Z)^{\mcl{D}_{\ul{\gl}}}=0$.
The group multiplication in $\mcl{P}_{\ul{\gl}}$ is given as follows.
\equa{\matthree {1+pa}{pb}{pc}{d}{1}{0}{e}{f}{1}&\matthree {1+pa'}{pb'}{p c'}{d'}{1}{0}{e'}{f'}{1}\\&=\matthree{1+p(a+a'+bd'+ce')}{(b+b'+cf')p}{p(c+c')}{d+d'}{1}{0}{e+e'+fd'}{f+f'}{1}.}
We observe that $Z(\mcl{P}_{\ul{\gl}})=\{\matthree {1+pa}{0}{0}{0}{1}{0}{0}{0}{1}\mid a\in \Z/p\Z\}$.

Let $[y]\in H^2_{Trivial\ Action}(\mcl{P}_{\ul{\gl}},\Z/p\Z)^{\mcl{D}_{\ul{\gl}}}$ and consider $Res:H^2_{Trivial\ Action}(\mcl{P}_{\ul{\gl}},\Z/p\Z)$ $\lra H^2_{Trivial\ Action}(Z(\mcl{P}_{\ul{\gl}}),\Z/p\Z)$. We observe using Theorem~\ref{theorem:KernelRes} where we choose $s=\es,s'=(1,1,1,1,1)$, we get $[y]\in \Ker(Res)$. Consider the map $\gth:H^2_{Trivial\ Action}(\mcl{P}_{\ul{\gl}},\Z/p\Z)\lra P(\mcl{P}_{\ul{\gl}},Z(\mcl{P}_{\ul{\gl}}),\Z/p\Z)$. The class $[y]\in \Ker(\gth)$. This is because $[y]\in H^2_{Trivial\ Action}(\mcl{P}_{\ul{\gl}},\Z/p\Z)^{\mcl{D}_{\ul{\gl}}}$ and the positions of $d,f,c$ are not symmetric about the diagonal with the position of $a$.

Using the exactness of the sequence 
\equa{H^2_{Trivial\ Action}(&\frac{\mcl{P}_{\ul{\gl}}}{Z(\mcl{P}_{\ul{\gl}})},\Z/p\Z)\os{Inf}{\lra} H^2_{Trivial\ Action}(\mcl{P}_{\ul{\gl}},\Z/p\Z)\\&\os{Res\times \gth}{\lra}  H^2_{Trivial\ Action}(Z(\mcl{P}_{\ul{\gl}}),\Z/p\Z) \times P(\mcl{P}_{\ul{\gl}},Z(\mcl{P}_{\ul{\gl}}),\Z/p\Z),}
we get that there exists $[x]\in H^2_{Trivial\ Action}(\frac{\mcl{P}_{\ul{\gl}}}{Z(\mcl{P}_{\ul{\gl}})},\Z/p\Z)$ such that $Inf([x])=[y]$.

So we have by averaging $[x]$ with respect to the action of the group $\mcl{D}_{\ul{\gl}}$
\equ{Inf\bigg(\frac{1}{\mid \mcl{D}_{\ul{\gl}}\mid}\us{t\in \mcl{D}_{\ul{\gl}}}{\sum}t\bullet[x]\bigg)=[y]} and the cohomology class
\equ{\frac{1}{\mid \mcl{D}_{\ul{\gl}}\mid}\us{t\in \mcl{D}_{\ul{\gl}}}{\sum}t\bullet[x]\in H^2_{Trivial\ Action}(\frac{\mcl{P}_{\ul{\gl}}}{Z(\mcl{P}_{\ul{\gl}})},\Z/p\Z)^{\mcl{D}_{\ul{\gl}}}.}

Now the group $\frac{\mcl{P}_{\ul{\gl}}}{Z(\mcl{P}_{\ul{\gl}})} \cong G_3$ (defined in Theorem~\ref{theorem:GL21}) and the action of $\mcl{D}_{\ul{\gl}}$ on $\frac{\mcl{P}_{\ul{\gl}}}{Z(\mcl{P}_{\ul{\gl}})}$ and the action of $D$ in Theorem~\ref{theorem:GL211} on $G_3$ are compatible. Hence using Theorem~\ref{theorem:GL211} we get that $[y]=\gb [z]$ for some $\gb \in \Z/p\Z$ where 
\equ{z\bigg(\matthree {1+pa}{pb}{pc}{d}{1}{0}{e}{f}{1},\matthree {1+pa'}{pb'}{p c'}{d'}{1}{0}{e'}{f'}{1}\bigg)=bd'+ce'.}
But we observe that $z$ is a 2-coboundary because if $v:\mcl{P}_{\ul{\gl}}\lra \Z/p\Z$ is a $1$-cochain defined as: 
$v(\matthree {1+pa}{pb}{pc}{d}{1}{0}{e}{f}{1})=a$ then we have $v(gg')-v(g)-v(g')=bd'+ce'$ where $g=\matthree {1+pa}{pb}{pc}{d}{1}{0}{e}{f}{1},g'=\matthree {1+pa'}{pb'}{p c'}{d'}{1}{0}{e'}{f'}{1}$. 

Hence we get $[y]=0$ and so $H^2_{Trivial\ Action}(\mcl{P}_{\ul{\gl}},\Z/p\Z)^{\mcl{D}_{\ul{\gl}}}=0$. This proves the theorem.
\end{proof}
\begin{remark}
	In the proof of Theorem~\ref{theorem:GL211Second} we only require that $p$ is an odd prime and $p$ can be equal to $3$.
\end{remark}
\begin{theorem}
\label{theorem:GL221}
Let $p$ be an odd prime. Consider the following $p$-group of order $p^6$.
\equ{G_4=\bigg\{\matthree{}{a}{b}{c}{}{d}{e}{f}{}\mid a,b,c,d,e,f\in \Z/p\Z\bigg\}}
where the group multiplication is given by: 
\equ{\matthree{}{a}{b}{c}{}{d}{e}{f}{}\matthree{}{a'}{b'}{c'}{}{d'}{e'}{f'}{}=\matthree{}{a+a'+bf'}{b+b'}{c+c'}{}{d+d'+cb'}{e+e'+fc'}{f+f'}{}.}
Consider the action of the group $D=(\Z/p\Z)^*\times (\Z/p\Z)^*\times (\Z/p\Z)^*$ on $G_4$ as follows. For $t_i\in (\Z/p\Z)^*,i=1,2,3,$
\equ{(t_1,t_2,t_3)\bullet \matthree{}{a}{b}{c}{}{d}{e}{f}{}=\matthree{}{\frac{t_1}{t_2}a}{\frac{t_1}{t_3}b}{\frac{t_2}{t_1}c}{}{\frac{t_2}{t_3}d}{\frac{t_3}{t_1}e}{\frac{t_3}{t_2}f}{}.}
Then the subgroup of $D$-invariant cohomology classes of $H^2_{Trivial\ Action}(G_4,\Z/p\Z)$ is 
\equ{H^2_{Trivial\ Action}(G_4,\Z/p\Z)^{D}=\{\ga[y]+\gb[z]\mid \ga,\gb\in \Z/p\Z\}\cong \Z/p\Z\oplus\Z/p\Z}
where $y,z\in Z^2_{Trivial\ Action}(G_4,\Z/p\Z)$ are defined as:
\equa{y\bigg(\matthree{}{a}{b}{c}{}{d}{e}{f}{},\matthree{}{a'}{b'}{c'}{}{d'}{e'}{f'}{}\bigg)&=ca'+df',\\
z\bigg(\matthree{}{a}{b}{c}{}{d}{e}{f}{},\matthree{}{a'}{b'}{c'}{}{d'}{e'}{f'}{}\bigg)&=ac'+be'.}
\end{theorem}
\begin{proof}
Clearly $y,z\in Z^2_{Trivial\ Action}(G_4,\Z/p\Z)$ are indeed cocycles and are invariant with respect to the action of $D$.
We just show below that $y$ is a cocycle. We have
\equ{ca'+df'+(c+c')a''+(d+d'+cb')f''=c'a''+d'f''+c(a'+a''+b'f'')+d(f'+f'')}
for $a,b,c,d,e,f,a',b',c',d',e',f',a'',b'',c'',d'',e'',f'' \in\Z/p\Z$.

Let $[x]\in H^2_{Trivial\ Action}(G_4,\Z/p\Z)^{D}$.
Let $\gth:H^2_{Trivial\ Action}(G_4,\Z/p\Z)^{D}\lra P(G_4,Z(G_4),\Z/p\Z)$ as defined in Proposition~\ref{prop:ThetaMap}.
Then we have 
\equa{\gth([x])\bigg(\matthree{}{a}{b}{c}{}{d}{e}{f}{},\matthree{}{a'}{0}{0}{}{d'}{e'}{0}{}\bigg)&=\ga fd'+\gb ca'+\gga be'+\gd fa'+\gm fe'+\gep cd'\\&+ \gt ce'+\gom ba'+\gn bd'.}
By invariance of $[x]$ under the action of $D$ we conclude that $\gd=\gm=\gep=\gt=\gom=\gn=0$.
So  
\equa{\gth([x])\bigg(\matthree{}{a}{b}{c}{}{d}{e}{f}{},\matthree{}{a'}{0}{0}{}{d'}{e'}{0}{}\bigg)&=\ga fd'+\gb ca'+\gga be'.}
We also have 
\equa{\gth([y])\bigg(\matthree{}{a}{b}{c}{}{d}{e}{f}{},\matthree{}{a'}{0}{0}{}{d'}{e'}{0}{}\bigg)&=ca'- fd'.}
So we get that 
\equa{\gth([x+\ga y])\bigg(\matthree{}{a}{b}{c}{}{d}{e}{f}{},\matthree{}{0}{0}{0}{}{d'}{0}{0}{}\bigg)&=0.}
We also have $[x+\ga y]\in \Ker(H^2_{Trivial\ Action}(G_4,\Z/p\Z)\lra H^2_{Trivial\ Action}(Z,\Z/p\Z))$ where $Z=\{\matthree{}{0}{0}{0}{}{d}{0}{0}{}\mid d\in \Z/p\Z\}$. This is because the position of $d$ in $Z$ is nondiagonal and $D$ acts nontrivially in this position. Therefore $H^2_{Trivial\ Action}(Z,$ $\Z/p\Z)^D=\{0\}$ and $Res([x+\ga y])\in H^2_{Trivial\ Action}(Z,\Z/p\Z)^D=\{0\}$.

 Now using the exactness of the following sequence
\equa{H^2_{Trivial\ Action}(\frac{G_4}{Z},\Z/p\Z)&\os{Inf}{\lra}H^2_{Trivial\ Action}(G_4,\Z/p\Z)\\
&\os{Res\times \gth}{\lra}H^2_{Trivial\ Action}(Z,\Z/p\Z)\times P(G_4,Z,\Z/p\Z)}
we conclude that the cohomology class $[x+\ga y]$ is in the image of the inflation map, that is, there exists, $[\ti{x}]\in H^2_{Trivial\ Action}(\frac{G_4}{Z},\Z/p\Z)$ such that $Inf([\ti{x}])=[x+\ga y]$. 

Now by averaging $[\ti{x}]$ under the action of $D$ we get that 
\equ{Inf(\frac{1}{\mid D\mid}\us{t\in D}{\sum}t\bullet[\ti{x}])=[x+\ga y]}
and $\frac{1}{\mid D\mid}\us{t\in D}{\sum}t\bullet[\ti{x}]\in H^2_{Trivial\ Action}(\frac{G_4}{Z},\Z/p\Z)^D$.
We observe that $\frac{G_4}{Z}\cong G_3$ where $G_3$ is as defined in Theorem~\ref{theorem:GL211} and the actions of $D$ on $\frac{G_4}{Z}$ and $G_3$ are compatible. Using Theorem~\ref{theorem:GL211}, we get that there exists $\gk\in \Z/p\Z$ such that $\frac{1}{\mid D\mid}\us{t\in D}{\sum}t\bullet[\ti{x}]=\gk [z]$ where $[z]$ is as defined in Theorem~\ref{theorem:GL211}. Upon inflation we get that $[x+\ga y]=\gk [z]$ where $[z]$ is as defined in Theorem~\ref{theorem:GL221}. In other words we have $[x]=-\ga [y]+\gk[z]$. 

Now the classes $[y]\neq 0\neq [z]$ because $\gth([y])\neq 0\neq \gth([z])$. We also have $[y]\neq \gs[z]$ for any $\gs\in \Z/p\Z$ because $\gth([y-\gs z])\neq 0$ for any $\gs\in \Z/p\Z$. Hence we conclude that 
$H^2_{Trivial\ Action}(G_4,\Z/p\Z)^D=\Z/p\Z\oplus\Z/p\Z$.
This proves the theorem.  
\end{proof}
\begin{remark}
	In the proof of Theorem~\ref{theorem:GL221} we only require that $p$ is an odd prime and $p$ can be equal to $3$.
\end{remark}
\begin{theorem}
	\label{theorem:GL221Second}
	Let $\ul{\gl}=(2^2>1^1)$ and $p$ be an odd prime and $p\neq 3$. Then \equ{H^2_{Trivial\ Action}(\autgp,\Z/p\Z)=0.}
\end{theorem}
\begin{proof}
We show that $H^2_{Trivial\ Action}(\mcl{P}_{\ul{\gl}},\Z/p\Z)^{\mcl{D}_{\ul{\gl}}}=0$.
The group multiplication in $\mcl{P}_{\ul{\gl}}$ is given as follows.
\equa{&\matthree {1+pa}{pb}{pc}{d+pe}{1+pf}{pg}{h}{i}{1}\matthree {1+pa'}{pb'}{pc'}{d'+pe'}{1+pf'}{pg'}{h'}{i'}{1}=\\&\matthree{1+p(a+a'+bd'+ch')}{(b+b'+ci')p}{p(c+c')}{d+d'+p(e+e'+da'+fd'+gh')}{1+p(f+f'+db'+gi')}{p(g+g'+dc')}{h+h'+id'}{i+i'}{1}.}
We observe that $Z(\mcl{P}_{\ul{\gl}})=\{\matthree {1}{0}{0}{pe}{1}{0}{0}{0}{1}\mid e\in \Z/p\Z\}=\mcl{N}^{s}_{\ul{\gl}}$
where $s=(1,2,$ $1,1,1)\in S$.

Let $[x]\in H^2_{Trivial\ Action}(\mcl{P}_{\ul{\gl}},\Z/p\Z)^{\mcl{D}_{\ul{\gl}}}$ and consider $Res:H^2_{Trivial\ Action}(\mcl{P}_{\ul{\gl}},\Z/p\Z)$ $\lra H^2_{Trivial\ Action}(Z(\mcl{P}_{\ul{\gl}}),\Z/p\Z)$. We observe using Theorem~\ref{theorem:KernelRes} where we choose $s=\es,s'=(1,1,1,1,1)$, we get $[x]\in \Ker(Res)$. Consider the map $\gth:H^2_{Trivial\ Action}(\mcl{P}_{\ul{\gl}},\Z/p\Z)\lra P(\mcl{P}_{\ul{\gl}},Z(\mcl{P}_{\ul{\gl}}),\Z/p\Z)$. The class $[x]\in \Ker(\gth)$. This is because $[x]\in H^2_{Trivial\ Action}(\mcl{P}_{\ul{\gl}},\Z/p\Z)^{\mcl{D}_{\ul{\gl}}}$ and the positions of $c,i$ are not symmetric about the diagonal with the position of $a$ and the position of $d,e$ are same, but $p\neq 3$. Note if $p\neq 3$ is an odd prime then there exists $t\in \Z/p\Z$ such that $0\neq t^2\neq 1$.

Using the exactness of the sequence 
\equa{H^2_{Trivial\ Action}(&\frac{\mcl{P}_{\ul{\gl}}}{Z(\mcl{P}_{\ul{\gl}})},\Z/p\Z)\os{Inf}{\lra} H^2_{Trivial\ Action}(\mcl{P}_{\ul{\gl}},\Z/p\Z)\\&\os{Res\times \gth}{\lra}  H^2_{Trivial\ Action}(Z(\mcl{P}_{\ul{\gl}}),\Z/p\Z) \times P(\mcl{P}_{\ul{\gl}},Z(\mcl{P}_{\ul{\gl}}),\Z/p\Z),}
we get that there exists $[\ti{x}]\in H^2_{Trivial\ Action}(\frac{\mcl{P}_{\ul{\gl}}}{Z(\mcl{P}_{\ul{\gl}})},\Z/p\Z)$ such that $Inf([\ti{x}])=[x]$.

So we have by averaging $[\ti{x}]$ with respect to the action of the group $\mcl{D}_{\ul{\gl}}$
\equ{Inf\bigg(\frac{1}{\mid \mcl{D}_{\ul{\gl}}\mid}\us{t\in \mcl{D}_{\ul{\gl}}}{\sum}t\bullet[\ti{x}]\bigg)=[x]} and the cohomology class
\equ{\frac{1}{\mid \mcl{D}_{\ul{\gl}}\mid}\us{t\in \mcl{D}_{\ul{\gl}}}{\sum}t\bullet[\ti{x}]\in H^2_{Trivial\ Action}(\frac{\mcl{P}_{\ul{\gl}}}{Z(\mcl{P}_{\ul{\gl}})},\Z/p\Z)^{\mcl{D}_{\ul{\gl}}}=H^2_{Trivial\ Action}(\frac{\mcl{P}_{\ul{\gl}}}{\mcl{N}^s_{\ul{\gl}}},\Z/p\Z)^{\mcl{D}_{\ul{\gl}}}.}

Hence we have \equ{Inf(H^2_{Trivial\ Action}(\frac{\mcl{P}_{\ul{\gl}}}{\mcl{N}^s_{\ul{\gl}}},\Z/p\Z)^{\mcl{D}_{\ul{\gl}}})=H^2_{Trivial\ Action}(\mcl{P}_{\ul{\gl}},\Z/p\Z)^{\mcl{D}_{\ul{\gl}}}.}

By repeating the steps similarly once again we get that for $s'=(1,2,2,1,1)$, the successor of $s$ in $(S,\leq_{MTO})$,
\equ{Inf(H^2_{Trivial\ Action}(\frac{\mcl{P}_{\ul{\gl}}}{\mcl{N}^{s'}_{\ul{\gl}}},\Z/p\Z)^{\mcl{D}_{\ul{\gl}}})=H^2_{Trivial\ Action}(\mcl{P}_{\ul{\gl}},\Z/p\Z)^{\mcl{D}_{\ul{\gl}}}.}

By repeating the steps similarly once again we get that for $s''=(1,1,1,1,1)$, the successor of $s'$ in $(S,\leq_{MTO})$,
\equ{Inf(H^2_{Trivial\ Action}(\frac{\mcl{P}_{\ul{\gl}}}{\mcl{N}^{s''}_{\ul{\gl}}},\Z/p\Z)^{\mcl{D}_{\ul{\gl}}})=H^2_{Trivial\ Action}(\mcl{P}_{\ul{\gl}},\Z/p\Z)^{\mcl{D}_{\ul{\gl}}}.}

Now we observe that 
\equ{\frac{\mcl{P}_{\ul{\gl}}}{\mcl{N}^{s''}_{\ul{\gl}}}=\bigg\{\matthree {1\mod p}{pb}{pc}{d\mod p}{1\mod p}{pg}{h}{i}{1}\mid b,d,c,g,h,i\in \Z/p\Z\bigg\}\cong G_4}
where $G_4$ is as defined in Theorem~\ref{theorem:GL221}. So using Theorem~\ref{theorem:GL221} we conclude that 
\equ{[x]=\ga [y]+\gb [z] \text{ for some }\ga,\gb\in \Z/p\Z}
where \equ{y\bigg(\matthree {1+pa}{pb}{pc}{d+pe}{1+pf}{pg}{h}{i}{1},\matthree {1+pa'}{pb'}{pc'}{d'+pe'}{1+pf'}{pg'}{h'}{i'}{1}\bigg)=db'+gi'}
and 
\equ{z\bigg(\matthree {1+pa}{pb}{pc}{d+pe}{1+pf}{pg}{h}{i}{1},\matthree {1+pa'}{pb'}{pc'}{d'+pe'}{1+pf'}{pg'}{h'}{i'}{1}\bigg)=bd'+ch'.}
But both $y$ and $z$ are $2$-coboundaries on $\mcl{P}_{\ul{\gl}}$ by observing the elements in positions $(1,1)$ and $(2,2)$ when we multiply two general elements of $\mcl{P}_{\ul{\gl}}$.

So $H^2_{Trivial\ Action}(\mcl{P}_{\ul{\gl}},\Z/p\Z)^{\mcl{D}_{\ul{\gl}}}=0$. This completes the proof of the theorem.
\end{proof} 
\begin{remark}
	In the proof of Theorem~\ref{theorem:GL221Second} we require that $p$ is an odd prime and $p$ is not equal to $3$.
\end{remark}
\begin{theorem}
\label{theorem:KernelThetaOne}
Let $p$ be an odd prime and
let $s,s'\in S\bs T$ be such that $s'$ is the successor element of $s$ in $(S,\leq_{MTO})$. Consider the  map $\gth:H^2_{Trivial\ Action}(\frac{\mcl{P}_{\ul{\gl}}}{\mcl{N}^s_{\ul{\gl}}},\Z/p\Z) \lra P(\frac{\mcl{P}_{\ul{\gl}}}{\mcl{N}^s_{\ul{\gl}}},\frac{\mcl{N}^{s'}_{\ul{\gl}}}{\mcl{N}^s_{\ul{\gl}}},\Z/p\Z)$ which is defined in Theorem~\ref{theorem:HSESequence}. 
Then we have 
 \equ{H^2_{Trivial\ Action}(\frac{\mcl{P}_{\ul{\gl}}}{\mcl{N}^s_{\ul{\gl}}},\Z/p\Z)^{\mcl{D}_{\ul{\gl}}}\subseteq \Ker(\gth),}
in the following cases.
\begin{enumerate}
	\item $Pos(s')=(x',x')$ for some $1\leq x'\leq \gr=\gr_1+\cdots+\gr_k$.
	\item $Pos(s')=(x',y')$ and both $(x',y'),(y',x')\nin \{Pos(t)\mid t\in T\}$.
	\item $Pos(s')=(x',y')\in \{Pos(t)\mid t\in T\}, (y',x')\nin \{Pos(t)\mid t\in T\}$ and $p\neq 3$ is an odd prime.
	\item $Pos(s')=(x',y'),x'-y'>1$ and $(y',x')\in \{Pos(t)\mid t\in T\}$.
	\item $Pos(s')=(x',y'),x'-y'=1$ and $(y',x')\in \{Pos(t)\mid t\in T\}$ and $p\neq 3$ is an odd prime.
\end{enumerate}
\end{theorem}
\begin{proof}
For $s\in S\bs T$, we have $\mcl{N}^s_{\ul{\gl}}\subseteq \mcl{P}'_{\ul{\gl}}$. So $\frac{\frac{\mcl{P}_{\ul{\gl}}}{\mcl{N}^s_{\ul{\gl}}}}{[\frac{\mcl{P}_{\ul{\gl}}}{\mcl{N}^s_{\ul{\gl}}},\frac{\mcl{P}_{\ul{\gl}}}{\mcl{N}^s_{\ul{\gl}}}]}\cong \frac{\mcl{P}_{\ul{\gl}}}{[\mcl{P}_{\ul{\gl}},\mcl{P}_{\ul{\gl}}]}\cong (\Z/p\Z)^{\mid T\mid}$ using Theorem~\ref{theorem:Commutator}(3). Hence we have \equ{P(\frac{\mcl{P}_{\ul{\gl}}}{\mcl{N}^s_{\ul{\gl}}},\frac{\mcl{N}^{s'}_{\ul{\gl}}}{\mcl{N}^s_{\ul{\gl}}},\Z/p\Z)\cong\bigg\{\ti{\gb}:\frac{\mcl{P}_{\ul{\gl}}}{[\mcl{P}_{\ul{\gl}},\mcl{P}_{\ul{\gl}}]}\times \frac{\mcl{N}^{s'}_{\ul{\gl}}}{\mcl{N}^s_{\ul{\gl}}}\lra \Z/p\Z\mid \ti{\gb} \text{ is bilinear}\bigg\}.}
If $[d]\in H^2_{Trivial\ Action}(\frac{\mcl{P}_{\ul{\gl}}}{\mcl{N}^s_{\ul{\gl}}},\Z/p\Z)^{\mcl{D}_{\ul{\gl}}}$ then using Theorem~\ref{theorem:HSESequence}, for $g^a\in \mcl{P}_{\ul{\gl}},g^b\in \mcl{N}^{s'}_{\ul{\gl}}$ 
\equa{\gth([d])(g^a\mcl{N}^s_{\ul{\gl}},g^b\mcl{N}^s_{\ul{\gl}})&=d(g^a\mcl{N}^s_{\ul{\gl}},g^b\mcl{N}^s_{\ul{\gl}})-
	d(g^b\mcl{N}^s_{\ul{\gl}},g^a\mcl{N}^s_{\ul{\gl}})\\&=\ti{\gb}((a_t)_{t\in T},b_{s'})=\us{t\in T}{\sum}\gb_t b_{s'}a_t \text{ for some }\gb_t\in \Z/p\Z,t\in T.}
Moreover $\ti{\gb}$ depends only on the cohomology class of $[d]$. Now we show that $\ti{\gb}=0$, that is, $\gb_t=0$ for all $t\in T$ in the cases $(1),(2),(3),(4),(5)$ mentioned.

We use the action of elements of $\mcl{D}_{\ul{\gl}}$ on $[d]$. Since $[d]$ is invariant under this action, we have 
the following. 
\begin{itemize}
	\item If $s'$ is in the diagonal position then $\gb_t=0$ for all $t\in T$, that is, $\ti{\gb}=0$.
	\item If $s'$ is not in the diagonal position and $Pos(s')=(x',y')$ then $\gb_t=0$ for all $t\in T$ such that
	\begin{enumerate}[label=(\roman*)]
		\item $Pos(t)\nin \{(x',y'),(y',x')\}$,
		\item $Pos(t)=Pos(s')=(x',y')$ since $p\neq 3$. Note if $p\neq 3$ then there exists $\gga\in (\Z/p\Z)^*$ such that $\gga^2\neq 1$.
	\end{enumerate} 
\end{itemize}
This establishes the theorem in cases $(1),(2),(3)$.

Now consider the case $(4): Pos(s')=(x',y'), x'-y'>1$ and there exists $t_0\in T$ such that $Pos(t_0)=(y',x')$. It remains to show that $\gb_{t_0}=0$ because 
in this case we have \equ{\gth([d])(g^a\mcl{N}^s_{\ul{\gl}},g^b\mcl{N}^s_{\ul{\gl}})=\gb_{t_0}a_{t_0}b_{s'}\text{ for }g^a\in \mcl{P}_{\ul{\gl}},g^b\in \mcl{N}^{s'}_{\ul{\gl}}.} 
If $x'-y'>1$ then $s'$ appears in a strictly lower triangular entry below the first subdiagonal. Here $t_0$ appears in a strictly upper triangular entry above the first superdiagonal. So the first coordinate of $t_0$ is $1$. Let the first coordinate of $s'$ be $l'\geq 0$.
Consider the positions in the set $\{y',x'-1,x'\}\times \{y',x'-1,x'\}$. Let $A_1,pA_2$ be typical entries of a matrix in $\mcl{P}_{\ul{\gl}}$ in positions $(x'-1,y')$ and $(y',x')$ respectively. Let $r\in S$ be such that the first coordinate of $r$ is $l'$ and $Pos(r)=(x',x'-1)$.
Consider the subgroup of $\mcl{F}^{s'}_{\ul{\gl}}=\{ g^a+e_{(x',y')}(pA_2)+e_{(x',x'-1)}(p^{l'}a_r)+e_{(x'-1,y')}(A_1)\mid 0\leq a_r\leq p-1,g^a\in \mcl{N}^{s'}_{\ul{\gl}}\}$. We can immediately see that $\mcl{F}^{s'}_{\ul{\gl}}$ is a subgroup and $\mcl{N}^s_{\ul{\gl}}\trianglelefteq \mcl{F}^{s'}_{\ul{\gl}}$ is a normal subgroup. The group multiplication in $\frac{\mcl{F}^{s'}_{\ul{\gl}}}{\mcl{N}^s_{\ul{\gl}}}$ is given in terms of $3\times 3$ matrices as follows.
For $g^a\mcl{N}^{s}_{\ul{\gl}},g^b\mcl{N}^{s}_{\ul{\gl}}\in \frac{\mcl{F}^{s'}_{\ul{\gl}}}{\mcl{N}^{s}_{\ul{\gl}}}$ the images of $g^a,g^b\in \mcl{F}^{s'}_{\ul{\gl}}$ the product is given as:
\fo{8}{8}{
	\equa{&\matthree{1\mod p^{l'+1}}{0 \mod p^{l'+1}}{pA_2 \mod p^{l'+1}}{A_1\mod p^{l'+1}}{1\mod p^{l'+1}}{0\mod p^{l'+1}}{p^{l'}a_{s'}\mod p^{l'+1}}{p^{l'}a_r\mod p^{l'+1}}{1\mod p^{l'+1}}\matthree{1\mod p^{l'+1}}{0 \mod p^{l'+1}}{pB_2 \mod p^{l'+1}}{B_1\mod p^{l'+1}}{1\mod p^{l'+1}}{0\mod p^{l'+1}}{p^{l'}b_{s'}\mod p^{l'+1}}{p^{l'}b_r\mod p^{l'+1}}{1\mod p^{l'+1}}\\
		&=\matthree{1\mod p^{l'+1}}{0 \mod p^{l'+1}}{p(A_2+B_2) \mod p^{l'+1}}{A_1+B_1\mod p^{l'+1}}{1\mod p^{l'+1}}{0\mod p^{l'+1}}{p^{l'}(a_{s'}+b_{s'}+a_rB_1)\mod p^{l'+1}}{p^{l'}(a_r+b_r)\mod p^{l'+1}}{1\mod p^{l'+1}}}
}
Here we have  \equ{\gth([d])(g^a\mcl{N}^s_{\ul{\gl}},g^b\mcl{N}^s_{\ul{\gl}})=\gb_{t_0}\ol{A}_2b_{s'}\text{ for }g^a\in \mcl{F}^{s'}_{\ul{\gl}},g^b\in \mcl{N}^{s'}_{\ul{\gl}}}
where $\ol{A}_2$ is the residue of $A_2$ modulo $p$.
Now consider the normal abelian subgroup $\mcl{H}^{s'}_{\ul{\gl}}$ of $\frac{\mcl{F}^{s'}_{\ul{\gl}}}{\mcl{N}^{s}_{\ul{\gl}}}$ given by $3\times 3$ matrices of the form  
\equ{\matthree{1\mod p^{l'+1}}{0 \mod p^{l'+1}}{pA_2 \mod p^{l'+1}}{A_1\mod p^{l'+1}}{1\mod p^{l'+1}}{0\mod p^{l'+1}}{p^{l'}a_{s'}\mod p^{l'+1}}{0\mod p^{l'+1}}{1\mod p^{l'+1}}}

where we have substituted $a_r=0$.
Consider the map $Res:H^2_{Trivial\ Action}(\frac{\mcl{P}_{\ul{\gl}}}{\mcl{N}^s_{\ul{\gl}}},\Z/p\Z)$ $\lra H^2_{Trivial\ Action}(\mcl{H}^{s'}_{\ul{\gl}},\Z/p\Z)$. We have using Corollary~\ref{cor:ProductCohomology}, for $g^a\mcl{N}^s_{\ul{\gl}},g^b\mcl{N}^s_{\ul{\gl}}\in \mcl{H}^{s'}_{\ul{\gl}}$
\equa{Res(d)(g^a\mcl{N}^s_{\ul{\gl}},&g^b\mcl{N}^s_{\ul{\gl}}) \approx_{cohomologous} s_1(A_1 \mod p^{l'+1},B_1\mod p^{l'+1})+s_2(a_{s'},b_{s'})\\&+s_3(pA_2\mod p^{l'+1},pB_2 \mod p^{l'+1})+\ga \ol{A}_1\ol{B}_2+\gga \ol{A}_1b_{s'}+\gb \ol{A}_2b_{s'}}
for some $\ga,\gb,\gga\in \Z/p\Z$ where $s_i, i=1,2,3$ are the restrictions of the cocycle $d$ in the respective components.
Since $[d]$ is invariant under the action of $\mcl{D}_{\ul{\gl}}$ we have $s_i,i=1,2,3$ are cohomolgous to zero and $\ga=\gga=0$.
Hence we have \equan{CohoClass}{Res(d)(g^a\mcl{N}^s_{\ul{\gl}},g^b\mcl{N}^s_{\ul{\gl}}) \approx_{cohomologous} \gb \ol{A}_2b_{s'} \text{ for }g^a\mcl{N}^s_{\ul{\gl}},g^b\mcl{N}^s_{\ul{\gl}}\in \mcl{H}^{s'}_{\ul{\gl}}.}
Let $\mcl{K}^{s'}$ be the abelian subgroup of $\mcl{H}^{s'}_{\ul{\gl}}$ consisting of matrices of the form 
\equ{\matthree{1\mod p^{l'+1}}{0 \mod p^{l'+1}}{0 \mod p^{l'+1}}{A_1\mod p^{l'+1}}{1\mod p^{l'+1}}{0\mod p^{l'+1}}{p^{l'}a_{s'}\mod p^{l'+1}}{0\mod p^{l'+1}}{1\mod p^{l'+1}}}
obtained by substituting $a_r=0$ and $A_2=0$ in matrices of $\frac{\mcl{F}^{s'}_{\ul{\gl}}}{\mcl{N}^{s}_{\ul{\gl}}}$.
Consider the map $\gth_0: H^2_{Trivial\ Action}(\mcl{H}^{s'}_{\ul{\gl}},\Z/p\Z)\lra P(\mcl{H}^{s'}_{\ul{\gl}},\mcl{K}^{s'}_{\ul{\gl}},\Z/p\Z)$ as defined in Proposition~\ref{prop:ThetaMap}.
We have for $g^a\mcl{N}^s_{\ul{\gl}}\in \mcl{H}^{s'}_{\ul{\gl}},g^b\mcl{N}^s_{\ul{\gl}}\in \mcl{K}^{s'}_{\ul{\gl}}$
\equ{\gb \ol{A}_2b_{s'}=\gth_0(Res([d]))(g^a\mcl{N}^s_{\ul{\gl}},g^b\mcl{N}^s_{\ul{\gl}}) \text{ from Equation~\ref{Eq:CohoClass}.}}
So, for  $g^a\mcl{N}^s_{\ul{\gl}}\in \mcl{H}^{s'}_{\ul{\gl}},g^b\mcl{N}^s_{\ul{\gl}}\in \frac{\mcl{N}^{s'}_{\ul{\gl}}}{\mcl{N}^s_{\ul{\gl}}}\subseteq \mcl{K}^{s'}_{\ul{\gl}}$,
\equa{\gb \ol{A}_2b_{s'}=\gth_0(Res([d]))(g^a\mcl{N}^s_{\ul{\gl}},g^b\mcl{N}^s_{\ul{\gl}})	
	&=d(g^a\mcl{N}^s_{\ul{\gl}},g^b\mcl{N}^s_{\ul{\gl}})-d(g^b\mcl{N}^s_{\ul{\gl}},g^a\mcl{N}^s_{\ul{\gl}})\\&=\gth([d])(g^a\mcl{N}^s_{\ul{\gl}},g^b\mcl{N}^s_{\ul{\gl}})=\gb_{t_0}\ol{A}_2b_{s'}.}
Hence $\gb=\gb_{t_0}$. Therefore for
$g^a\mcl{N}^s_{\ul{\gl}}\in \mcl{H}^{s'}_{\ul{\gl}},g^b\mcl{N}^s_{\ul{\gl}}\in \mcl{K}^{s'}_{\ul{\gl}}$
\equ{\gth_0(Res([d]))(g^a\mcl{N}^s_{\ul{\gl}},g^b\mcl{N}^s_{\ul{\gl}})=\gb_{t_0} \ol{A}_2b_{s'}.}
Now $Res([d])$ is invariant under the conjugation of the matrix $g^c=E_{(x',x'-1)}(p^{l'}c_r)$ where we have as a $3\times 3$ matrix:
\equ{g^c\mcl{N}^s_{\ul{\gl}}=\matthree{1\mod p^{l'+1}}{0 \mod p^{l'+1}}{0 \mod p^{l'+1}}{0\mod p^{l'+1}}{1\mod p^{l'+1}}{0\mod p^{l'+1}}{0\mod p^{l'+1}}{p^{l'}c_r\mod p^{l'+1}}{1\mod p^{l'+1}}.}
This conjugation detects the position $(x',y')$ and so we get for
$g^a\mcl{N}^s_{\ul{\gl}}\in \mcl{H}^{s'}_{\ul{\gl}},g^b\mcl{N}^s_{\ul{\gl}}\in \mcl{K}^{s'}_{\ul{\gl}}$
\equa{\gb_{t_0}\ol{A}_2b_{s'}&=\gth_0(Res([d]))(g^a\mcl{N}^s_{\ul{\gl}},g^b\mcl{N}^s_{\ul{\gl}})\\&=
\gth_0(Res([d]))(g^cg^a(g^c)^{-1}\mcl{N}^s_{\ul{\gl}},g^cg^b(g^c)^{-1}\mcl{N}^s_{\ul{\gl}})\\&=\gb_{t_0}\ol{A}_2(b_{s'}+c_r\ol{B}_1).}
Choosing $c_r\neq 0,\ol{B}_1\neq 0$ we get $\gb_{t_0}=0$. Hence $\gth([d])=0$ which implies $[d]\in \Ker(\gth)$. This proves the theorem in case $(4)$.

Now consider case $(5): Pos(s')=(x',y'),x'-y'=1$ and there exists $t_0\in T$ such that $Pos(t_0)=(y',x')$. Here the first coordinate of $t_0$ is $1$. In this scenario  for $g^a\in \mcl{P}_{\ul{\gl}},g^b\in \mcl{N}^{s'}_{\ul{\gl}}$, by invariance of $[d]$ under the action of $\mcl{D}_{\ul{\gl}}$ and $p\neq 3$, we conclude that
\equ{\gth([d])(g^a\mcl{N}^s_{\ul{\gl}},g^b\mcl{N}^s_{\ul{\gl}})=d(g^a\mcl{N}^s_{\ul{\gl}},g^b\mcl{N}^s_{\ul{\gl}})-d(g^b\mcl{N}^s_{\ul{\gl}},g^a\mcl{N}^s_{\ul{\gl}})=\gb_{t_0}a_{t_0}b_{s'}}
for some $\gb_{t_0}\in \Z/p\Z$. We show that $\gb_{t_0}=0$.

Here if the first coordinate of $s'$ is $l'$ then $s'\nin T, x'-y'=1 \Ra l'\geq 1$. So $x'+1\leq \gr=\gr_1+\cdots+\gr_k$. Let $A_1,pA_2$ be two typical elements which occur in positions $(x'+1,y')$ and $(y',x')$ in the matrix respectively. Note that $t_0=(1,1,1,m',m'+1)$ for some $m'$ where $\gr_{m'}=\gr_{m'+1}=1$ since $t_0\in T$ occurs in a first superdiagonal entry. Consider the abelian subgroup $\mcl{L}^{s'}_{\ul{\gl}}=\{ g^a+e_{(x'+1,y')}(p^{l'-1}A_1)+e_{(y',x')}(pA_2)
\mid g^a\in \mcl{N}^{s'}_{\ul{\gl}}\}$. We can immediately see that $\mcl{L}^{s'}_{\ul{\gl}}$ is a subgroup and $\mcl{N}^s_{\ul{\gl}}\trianglelefteq \mcl{L}^{s'}_{\ul{\gl}}$ is a normal subgroup. The group multiplication in $\frac{\mcl{L}^{s'}_{\ul{\gl}}}{\mcl{N}^s_{\ul{\gl}}}$ is given in terms of $3\times 3$ matrices as follows. Let $l''=\min(\gl_{m'+2},l'+1)\leq \gl_{m'+1}$. Now $l'+1\leq \gl_{m'+1} \Llra l'\leq \gl_{m'+2}$. So $l'\leq l'' \leq l'+1$.
For $g^a\mcl{N}^{s}_{\ul{\gl}},g^b\mcl{N}^{s}_{\ul{\gl}}\in \frac{\mcl{L}^{s'}_{\ul{\gl}}}{\mcl{N}^{s}_{\ul{\gl}}}$ the images of $g^a,g^b\in \mcl{L}^{s'}_{\ul{\gl}}$ the product is given as:
\fo{8}{8}{
	\equa{&\matthree{1\mod p^{l'+1}}{pA_2 \mod p^{l'+1}}{0 \mod p^{l'+1}}{p^{l'}a_{s'}\mod p^{l'+1}}{1\mod p^{l'+1}}{0\mod p^{l'+1}}{p^{l'-1}A_1\mod p^{l'}}{0\mod p^{l'}}{1\mod p^{l''}}\matthree{1\mod p^{l'+1}}{pB_2 \mod p^{l'+1}}{0 \mod p^{l'+1}}{p^{l'}b_{s'}\mod p^{l'+1}}{1\mod p^{l'+1}}{0\mod p^{l'+1}}{p^{l'-1}B_1\mod p^{l'}}{0\mod p^{l'}}{1\mod p^{l''}}\\
		&=\matthree{1\mod p^{l'+1}}{p(A_2+B_2) \mod p^{l'+1}}{0 \mod p^{l'+1}}{p^{l'}(a_{s'}+b_{s'})\mod p^{l'+1}}{1\mod p^{l'+1}}{0\mod p^{l'+1}}{p^{l'-1}(A_1+B_1)\mod p^{l'}}{0\mod p^{l'}}{1\mod p^{l''}}}
}
Consider the map $Res:H^2_{Trivial\ Action}(\frac{\mcl{P}_{\ul{\gl}}}{\mcl{N}^s_{\ul{\gl}}},\Z/p\Z) \lra H^2_{Trivial\ Action}(\frac{\mcl{L}^{s'}_{\ul{\gl}}}{\mcl{N}^s_{\ul{\gl}}},\Z/p\Z)$.
We have by using Corollary~\ref{cor:ProductCohomology} for $g^a\in \mcl{L}^{s'}_{\ul{\gl}},g^b\in \mcl{L}^{s'}_{\ul{\gl}}$
\equa{Res(d)&(g^a\mcl{N}^s_{\ul{\gl}},g^b\mcl{N}^s_{\ul{\gl}})\approx_{cohomologous}s_1(a_{s'},b_{s'})+s_2(p^{l'-1}A_1\mod p^{l'},p^{l'-1}B_1\mod p^{l'})\\&+s_3(pA_2\mod p^{l'+1},pB_2\mod p^{l'+1})+\ga_1\ol{A}_1\ol{B}_2+\gga_1\ol{A}_1b_{s'}+\gb_1\ol{A}_2b_{s'}}
for some $\ga_1,\gb_1,\gga_1\in \Z/p\Z$ where $s_i,i=1,2,3$ are the restrictions of the cocycle $d$ to the respective components, $\ol{A}_1,\ol{A}_2,\ol{B}_2$ are the residues of $A_1,A_2,B_2$ modulo $p$ respectively.
By invariance of $[d]$ under the action of the group $\mcl{D}_{\ul{\gl}}$ and using Theorem~\ref{theorem:BasicMultiple}, since $p$ is odd, we get that $s_i,i=1,2,3$ are cohomologous to zero and also $\ga_1=0=\gga_1$. 
So  \equ{Res(d)(g^a\mcl{N}^s_{\ul{\gl}},g^b\mcl{N}^s_{\ul{\gl}})\approx_{cohomologous}\gb_1\ol{A}_2b_{s'}.}
Let $\gth_1:H^2_{Trivial\ Action}(\frac{\mcl{L}^{s'}_{\ul{\gl}}}{\mcl{N}^s_{\ul{\gl}}},\Z/p\Z) \lra P(\frac{\mcl{L}^{s'}_{\ul{\gl}}}{\mcl{N}^s_{\ul{\gl}}},\frac{\mcl{L}^{s'}_{\ul{\gl}}}{\mcl{N}^s_{\ul{\gl}}},\Z/p\Z)$ where $\gth_1$ is given in Proposition~\ref{prop:ThetaMap}.
Now we have for $g^a\in \mcl{L}^{s'}_{\ul{\gl}},g^b\in \mcl{N}^{s'}_{\ul{\gl}}$ \equ{\gb_1\ol{A}_2b_{s'}=\gth_1(Res([d]))(g^a\mcl{N}^s_{\ul{\gl}},g^b\mcl{N}^s_{\ul{\gl}})=\gth([d])(g^a\mcl{N}^s_{\ul{\gl}},g^b\mcl{N}^s_{\ul{\gl}})=\gb_{t_0}\ol{A}_2b_{s'}\Ra \gb_1=\gb_{t_0}.}
Hence, for $g^a\in \mcl{L}^{s'}_{\ul{\gl}},g^b\in \mcl{L}^{s'}_{\ul{\gl}}$
\equ{Res(d)(g^a\mcl{N}^s_{\ul{\gl}},g^b\mcl{N}^s_{\ul{\gl}})\approx_{cohomologous}\gb_{t_0}\ol{A}_2b_{s'}.}

Now $Res([d])$ is invariant under conjugation by the following $3\times 3$ matrix which is the image of $E_{(x',x'+1)}(pC)\in \mcl{P}_{\ul{\gl}}$.
\equ{\ol{E}_{(x',x'+1)}(pC)\matthree{1\mod p^{l'+1}}{0 \mod p^{l'+1}}{0 \mod p^{l'+1}}{0\mod p^{l'+1}}{1\mod p^{l'+1}}{pC\mod p^{l'+1}}{0\mod p^{l'}}{0\mod p^{l'}}{1\mod p^{l''}}.}
We get for $g^a\in \mcl{L}^{s'}_{\ul{\gl}},g^b\in \mcl{L}^{s'}_{\ul{\gl}}$
\equa{&\gb_{t_0}\ol{A}_2b_{s'}=\gth_1(Res([d]))(g^a\mcl{N}^s_{\ul{\gl}},g^b\mcl{N}^s_{\ul{\gl}})=\\&\gth_1(Res([d]))(\ol{E}_{(x',x'+1)}(pC)g^a\ol{E}_{(x',x'+1)}(-pC)\mcl{N}^s_{\ul{\gl}},\ol{E}_{(x',x'+1)}(pC)g^b\ol{E}_{(x',x'+1)}(-pC)\mcl{N}^s_{\ul{\gl}})\\
&=\gb_{t_0}(\ol{A}_2+\ol{C}\ol{A}_1)b_{s'}\Ra \gb_{t_0}=0.}
Thus $\gth([d])=0\Ra [d]\in \Ker(\gth)$. This completes case $(5)$ of the theorem.

Hence the theorem follows.
\end{proof}
\begin{defn}[Standard $2$-Coboundaries on $\mcl{P}_{\ul{\gl}}$]
	\label{defn:StandardCoboundary}
Let $s\in S\bs\{\es\}$. We define the standard 1-cochain $w_{s}:\mcl{P}_{\ul{\gl}}\lra \Z/p\Z$ as:
\equ{\text{For }g^a\in \mcl{P}_{\ul{\gl}}, w_s(g^a)=a_s\mod p.}
Now we define the standard $2$-coboundary $v_s:\mcl{P}_{\ul{\gl}}\times \mcl{P}_{\ul{\gl}}\lra \Z/p\Z$ as:
\equ{\text{For }g^a,g^b, g^c=g^ag^b\in \mcl{P}_{\ul{\gl}},v_s(g^a,g^b)=w_s(g^c)-w_s(g^a)-w_s(g^b)=c_s-a_s-b_s \mod p.}
Hence by definition $v_s\in B^2_{Trivial\ Action}(\mcl{P}_{\ul{\gl}},\Z/p\Z)$.
\end{defn}
\begin{theorem}
	\label{theorem:KernelThetaTwo}
	Let $p\neq 3$ be an odd prime and
	let $s,s'\in S\bs T$ be such that $s'$ is the successor element of $s$ in $S\bs T$. 
	Let $Pos(s')=(x',x'+1)$.  Consider the  map $\gth:H^2_{Trivial\ Action}(\frac{\mcl{P}_{\ul{\gl}}}{\mcl{N}^s_{\ul{\gl}}},\Z/p\Z) \lra P(\frac{\mcl{P}_{\ul{\gl}}}{\mcl{N}^s_{\ul{\gl}}},\frac{\mcl{N}^{s'}_{\ul{\gl}}}{\mcl{N}^s_{\ul{\gl}}},\Z/p\Z)$ which is defined in Theorem~\ref{theorem:HSESequence}. Also consider the map $Inf:H^2_{Trivial\ Action}(\frac{\mcl{P}_{\ul{\gl}}}{\mcl{N}^s_{\ul{\gl}}},\Z/p\Z)\lra H^2_{Trivial\ Action}(\mcl{P}_{\ul{\gl}},\Z/p\Z)$.
	Then given any cohomology class $[x]\in H^2_{Trivial\ Action}(\frac{\mcl{P}_{\ul{\gl}}}{\mcl{N}^s_{\ul{\gl}}},$ $\Z/p\Z)^{\mcl{D}_{\ul{\gl}}}$ there exists a cohomology class $[v] \in H^2_{Trivial\ Action}(\frac{\mcl{P}_{\ul{\gl}}}{\mcl{N}^s_{\ul{\gl}}},\Z/p\Z)^{\mcl{D}_{\ul{\gl}}}$ such that 
	$Inf([v])=0$ and $[x]-[v]\in \Ker(\gth)$. 
\end{theorem}
\begin{proof}
	 Let $s'=(l',i',j',m',n')$ and let $r\in S$ be such that the first coordinate of $r$ is $l'$ and $Pos(r)=(x',x')$. Note that 
	 $r\leq_{MTO}s<_{MTO}s'$ and $l'>0$. So the first coordinate of $s$ is also $l'$. Define a cocycle $v\in Z^2_{Trivial\ Action}(\frac{\mcl{P}_{\ul{\gl}}}{\mcl{N}^s_{\ul{\gl}}},\Z/p\Z)^{\mcl{D}_{\ul{\gl}}}$ as follows.
	 For $g^a,g^b\in \mcl{P}_{\ul{\gl}}$ with $g^c=g^ag^b\in \mcl{P}_{\ul{\gl}}$, 
	 \equ{v(g^a\mcl{N}^s_{\ul{\gl}},g^b\mcl{N}^s_{\ul{\gl}})=c_r-a_r-b_r\mod p.}
	 We will show the following.
	 \begin{enumerate}
	 	\item $v$ is well defined.
	 	\item $v$ is a cocycle and $v\in Z^2_{Trivial\ Action}(\frac{\mcl{P}_{\ul{\gl}}}{\mcl{N}^s_{\ul{\gl}}},\Z/p\Z)^{\mcl{D}_{\ul{\gl}}}$.
	 	\item $Inf([v])=[v_r]=0$ where $v_r$ is the standard $2$-coboundary in  $B^2_{Trivial\ Action}($ $\mcl{P}_{\ul{\gl}},\Z/p\Z)$.
	 	\item There exists $\gb\in \Z/p\Z$ such that $[x]-\gb [v]\in \Ker(\gth)$.
 	 \end{enumerate}
We prove $(1)$. Since $s\in S\bs T$ we have $\mcl{N}^{s}_{\ul{\gl}}=\mcl{P}^{s}_{\ul{\gl}}\cap \mcl{P}'_{\ul{\gl}}$. This we have observed in the proof of Theorem~\ref{theorem:ModifiedChain}. Using Theorem~\ref{theorem:CosetRepresentatives}(3) and Theorem~\ref{theorem:Commutator}(2) we get that 
$g^a\mcl{N}^{s}_{\ul{\gl}}=g^b\mcl{N}^{s}_{\ul{\gl}}$ if and only if $a_t=b_t$ for all $s<_{TO}t,t\in S$ and also for all $t\in T$, that is, for all $t\in S$ such that $s<_{MTO}t$.

If $g^c=g^ag^b$ then we have 
\equan{GroupMult}{(g^c)_{(x',x')}=\us{z'<x'}{\sum}(g^a)_{(x',z')}(g^b)_{(z',x')}+(g^a)_{(x',x')}(g^b)_{(x',x')}+\us{z'>x'}{\sum}(g^a)_{(x',z')}(g^b)_{(z',x')}.}
Here we observe that if $z'<x'$ then $p\mid (g^b)_{(z',x')}$ and if $z'>x'$ then $p\mid (g^a)_{(x',z')}$.
So we observe that $c_r-a_r-b_r \mod p$ is depends only on  
\begin{itemize}
	\item $a_t,b_t$ for all $t=(l,i,j,m,n)$ where $l<l'$,
	\item those $b_t$ for all $t=(l,i,j,m,n)$ where $l=l'$ and which occur in positions $(z',x')$ with $z'<x'$,
	\item those $a_t$ for all $t=(l,i,j,m,n)$ where $l=l'$ and which occur in positions $(x',z')$ with $z'>x'$.	
\end{itemize}
Hence in all cases $s<_{TO} t$.
Therefore $v$ is well defined.

We prove $(2)$. It is clear that the following cocycle identity is satisfied.
\equ{v(g^a\mcl{N}^s_{\ul{\gl}},g^b\mcl{N}^s_{\ul{\gl}})+v(g^ag^b\mcl{N}^s_{\ul{\gl}},g^d\mcl{N}^s_{\ul{\gl}})=v(g^b\mcl{N}^s_{\ul{\gl}},g^d\mcl{N}^s_{\ul{\gl}})+v(g^a\mcl{N}^s_{\ul{\gl}},g^bg^d\mcl{N}^s_{\ul{\gl}}).}
This is because we have for $g^c=g^ag^b, g^e=g^bg^d, g^f=g^ag^bg^d$
\equ{(c_r-a_r-b_r)+(f_r-c_r-d_r)\equiv (e_r-b_r-d_r)+(f_r-a_r-e_r)\mod p.}
Moreover $r$ appears in a diagonal position.  Hence $v$ is invariant under the diagonal action. Therefore $v\in Z^2_{Trivial\ Action}(\frac{\mcl{P}_{\ul{\gl}}}{\mcl{N}^s_{\ul{\gl}}},\Z/p\Z)^{\mcl{D}_{\ul{\gl}}}$.

We prove $(3)$. It is clear that $Inf([v])=[v_r]$ where $v_r$ is the standard coboundary as defined in Definition~\ref{defn:StandardCoboundary}. Hence $Inf([v])=0$.

We prove $(4)$. Let $[x]\in H^2_{Trivial\ Action}(\frac{\mcl{P}_{\ul{\gl}}}{\mcl{N}^s_{\ul{\gl}}},\Z/p\Z)^{\mcl{D}_{\ul{\gl}}}$.
Let $t_0\in T$ be such that the first coordinate of $t_0$ is $0$ and $Pos(t_0)=(x'+1,x')$. Let $\ti{t}\in S$ be such that the first coordinate of $\ti{t}$ is $1$ and $Pos(\ti{t})=Pos(s')$. Then there are two cases.

\begin{enumerate}[label=(\roman*)]
	\item $\ti{t}\in T$. (Here in this case, since $s'\in S\bs T$, we have $s'\neq \ti{t}$ though $Pos(s')=Pos(\ti{t})$.)
	\item $\ti{t}\nin T$.
\end{enumerate}
Then for $g^a\in \mcl{P}_{\ul{\gl}},g^b\in \mcl{N}^{s'}_{\ul{\gl}}$, in both cases $(i),(ii)$, by invariance of $[x]$ under the action of $\mcl{D}_{\ul{\gl}}$, we conclude that
\equ{\gth([x])(g^a\mcl{N}^s_{\ul{\gl}},g^b\mcl{N}^s_{\ul{\gl}})=x(g^a\mcl{N}^s_{\ul{\gl}},g^b\mcl{N}^s_{\ul{\gl}})-x(g^b\mcl{N}^s_{\ul{\gl}},g^a\mcl{N}^s_{\ul{\gl}})=\gb_{t_0}a_{t_0}b_{s'}}
for some $\gb_{t_0}\in \Z/p\Z$. Note that in case $(i)$ we need $p\neq 3$ and in case $(ii)$ we just need $p$ to be an odd prime.

If $g^c=g^ag^b$ with $g^a\in \mcl{P}_{\ul{\gl}},g^b\in \mcl{N}^{s'}_{\ul{\gl}}$ then 
\equa{(g^b)_{(x',x'+1)}&=p^{l'}b_{s'}+p^{l'+1}(*),\\
	(g^b)_{(z',x')}&=p^{l'+1}(*) \text{ for }z'<x',\\ 
	(g^b)_{(z',x')}&=p^{l'}(*),p\mid (g^a)_{(x',z')} \text{ for }z'>x',\\
	(g^b)_{(x',x')}&=1+p^{l'}b_r+p^{l'+1}(*).}
Let $(g^a)_{(x',x')}=1+pa_{t_1}+p^2a_{t_2}+\cdots+p^{l'}a_{t_{l'}}+p^{l'+1}(*)$ where $t_i\in S$ with first coordinate of $t_i$ is $i$ and $Pos(t_i)=(x',x')$. Then we have $t_{l'}=r$ and we get from Equation~\ref{Eq:GroupMult} that, 
\equ{(g^c)_{(x',x')}\equiv (1+pa_{t_1}+p^2a_{t_2}+\cdots+p^{l'}a_{t_{l'}})(1+p^{l'}b_r)\mod p^{l'+1}.}
Therefore we have $c_r\equiv a_{t_{l'}}+b_r\equiv a_r+b_r\mod p\Ra v(g^a\mcl{N}^s_{\ul{\gl}},g^b\mcl{N}^s_{\ul{\gl}})=0$.
If $g^d=g^bg^a$ with $g^a\in \mcl{P}_{\ul{\gl}},g^b\in \mcl{N}^{s'}_{\ul{\gl}}$ then 
\equan{GroupMultTwo}{(g^d)_{(x',x')}&=\us{z'<x'}{\sum}(g^b)_{(x',z')}(g^a)_{(z',x')}+(g^b)_{(x',x')}(g^a)_{(x',x')}+(g^b)_{(x',x'+1)}(g^a)_{(x'+1,x')}\\&+\us{z'>x'+1}{\sum}(g^b)_{(x',z')}(g^a)_{(z',x')}.}
Here we have 
\equa{(g^b)_{(x',z')}&=p^{l'}(*),p\mid (g^a)_{(z',x')} \text{ for } z'<x',\\
(g^b)_{(x',z')}&=p^{l'+1}(*) \text{ for } z'>x'+1.}
So from Equation~\ref{Eq:GroupMultTwo} we get that,
\equ{(g^d)_{(x',x')}\equiv (1+p^{l'}b_r)(1+pa_{t_1}+p^2a_{t_2}+\cdots+p^{l'}a_{t_{l'}})+p^{l'}b_{s'}(a_{t_0}+p(*))\mod p^{l'+1}}
So $d_r\equiv b_r+a_{t_{l'}}+b_{s'}a_{t_0}\equiv b_r+a_r+b_{s'}a_{t_0} \mod p \Ra v(g^b\mcl{N}^s_{\ul{\gl}},g^a\mcl{N}^s_{\ul{\gl}})=a_{t_0}b_{s'}$.
So
\equ{\gth([v])(g^a\mcl{N}^s_{\ul{\gl}},g^b\mcl{N}^s_{\ul{\gl}})=v(g^a\mcl{N}^s_{\ul{\gl}},g^b\mcl{N}^s_{\ul{\gl}})-v(g^b\mcl{N}^s_{\ul{\gl}},g^a\mcl{N}^s_{\ul{\gl}})=-a_{t_0}b_{s'}.}
Hence we obtain $\gth([x]+\gb_{t_0}[v])=0\Ra [x]+\gb_{t_0}[v]\in \Ker(\gth)$.

This completes the proof of the theorem.
\end{proof}
\begin{theorem}
\label{theorem:Averaging}
Let $p$ be a prime. Let $s,s'\in S\bs T$ be such that $s'$ is the successor element of $s$ in $S\bs T$. Consider the inflation map 
$H^2_{Trivial\ Action}(\frac{\mcl{P}_{\ul{\gl}}}{\mcl{N}^{s'}_{\ul{\gl}}},\Z/p\Z) \os{Inf}{\lra} H^2_{Trivial\ Action}(\frac{\mcl{P}_{\ul{\gl}}}{\mcl{N}^s_{\ul{\gl}}},$ $\Z/p\Z)$. If $[c]\in H^2_{Trivial\ Action}(\frac{\mcl{P}_{\ul{\gl}}}{\mcl{N}^s_{\ul{\gl}}},\Z/p\Z)^{\mcl{D}_{\ul{\gl}}}$ and $[c]\in \text{Im}(Inf)$ then there exists  $[d]\in H^2_{Trivial\ Action}(\frac{\mcl{P}_{\ul{\gl}}}{\mcl{N}^{s'}_{\ul{\gl}}},\Z/p\Z)^{\mcl{D}_{\ul{\gl}}}$ such that $Inf([d])=[c]$.
\end{theorem}
\begin{proof}
Let $Inf([x])=[c]$. Then define \equ{[d]=\frac{1}{\mid \mcl{D}_{\ul{\gl}}\mid} \us{t\in \mcl{D}_{\ul{\gl}}}{\sum} t\bullet [x]. }
We observe that $Inf([d])=[c]$ and $[d]\in H^2_{Trivial\ Action}(\frac{\mcl{P}_{\ul{\gl}}}{\mcl{N}^{s'}_{\ul{\gl}}},\Z/p\Z)^{\mcl{D}_{\ul{\gl}}}$. Note here $p$ does not divide the cardinality of $\mcl{D}_{\ul{\gl}}$.
\end{proof}
\begin{remark}
	In the proof of Theorem~\ref{theorem:Averaging}, $p$ can be any prime. 
\end{remark}
\begin{theorem}
\label{theorem:InflationFromAbelinization}
Let $s$ be the maximal element of $S\bs T$. Consider the inflation map $H^2_{Trivial\ Action}(\frac{\mcl{P}_{\ul{\gl}}}{\mcl{N}^s_{\ul{\gl}}}=\frac{\mcl{P}_{\ul{\gl}}}{[\mcl{P}_{\ul{\gl}},\mcl{P}_{\ul{\gl}}]},\Z/p\Z) \os{Inf}{\lra} H^2_{Trivial\ Action}(\mcl{P}_{\ul{\gl}},\Z/p\Z)$. Then we have \equ{Inf(H^2_{Trivial\ Action}(\frac{\mcl{P}_{\ul{\gl}}}{\mcl{N}^s_{\ul{\gl}}}=\frac{\mcl{P}_{\ul{\gl}}}{[\mcl{P}_{\ul{\gl}},\mcl{P}_{\ul{\gl}}]},\Z/p\Z)^{\mcl{D}_{\ul{\gl}}})=0.}
\end{theorem}
\begin{proof}
From Theorem~\ref{theorem:Commutator}(3) we have $\frac{\mcl{P}_{\ul{\gl}}}{[\mcl{P}_{\ul{\gl}},\mcl{P}_{\ul{\gl}}]} \cong (\Z/p\Z)^{\mid T \mid}$. Consider the map $\gth:H^2_{Trivial\ Action}(\frac{\mcl{P}_{\ul{\gl}}}{[\mcl{P}_{\ul{\gl}},\mcl{P}_{\ul{\gl}}]},\Z/p\Z) \lra P(\frac{\mcl{P}_{\ul{\gl}}}{[\mcl{P}_{\ul{\gl}},\mcl{P}_{\ul{\gl}}]}, \frac{\mcl{P}_{\ul{\gl}}}{[\mcl{P}_{\ul{\gl}},\mcl{P}_{\ul{\gl}}]},\Z/p\Z)$.
If $[x]\in H^2_{Trivial\ Action}$ $(\frac{\mcl{P}_{\ul{\gl}}}{[\mcl{P}_{\ul{\gl}},\mcl{P}_{\ul{\gl}}]},\Z/p\Z)^{\mcl{D}_{\ul{\gl}}}$ then we have, for $g^a,g^b\in \mcl{P}_{\ul{\gl}}$,
\equa{x(g^a\mcl{N}^s_{\ul{\gl}},g^b\mcl{N}^s_{\ul{\gl}})&\approx_{cohomologous}\us{t\in T_1}{\sum}\gb_ta_{\gf(t)}b_t\text{ and }\\
	\gth([x])(g^a\mcl{N}^s_{\ul{\gl}},&g^b\mcl{N}^s_{\ul{\gl}})=\us{t\in T_1}{\sum}\gb_ta_{\gf(t)}b_t}
for some $\gb_t\in \Z/p\Z$
where $T_1=\{t\in T\mid \text{ there exists }1\leq m<k, t=(1,1,1,m,m+1),\gr_m=1=\gr_{m+1}\}$ and $\gf:T_1\lra T$ such that 
$\gf(t)=(0,1,1,m+1,m)$ the entry in $T$ which is symmetric w.r.t $t$ about the diagonal.

Now for $t\in T_1$ define a $2$-cocycle $u_t:\mcl{P}_{\ul{\gl}} \times \mcl{P}_{\ul{\gl}}\lra \Z/p\Z$ given by 
$u_t(g^a,g^b)=a_{\gf(t)}b_t$. Indeed $u_t\in Z^2_{Trivial\ Action}(\mcl{P}_{\ul{\gl}},\Z/p\Z)^{\mcl{D}_{\ul{\gl}}}$ because the cocycle identity is satisfied, that is, for $g^a,g^b,g^c\in \mcl{P}_{\ul{\gl}}$
\equ{a_{\gf(t)}b_t+(a_{\gf(t)}+b_{\gf(t)})c_t=b_{\gf(t)}c_t+a_{\gf(t)}(b_t+c_t)}
and $u_t$ is invariant under the action of the group $\mcl{D}_{\ul{\gl}}$.  
Now we know that 
\equ{Inf([x])(g^a,g^b)\approx_{cohomologous}\us{t\in T_1}{\sum}\gb_ta_{\gf(t)}b_t.}
So to prove the theorem it is enough to show that $u_{t_0}$ is a $2$-coboundary for any $t_0\in T_1$, that is, $u_{t_0}\in B^2_{Trivial\ Action}(\mcl{P}_{\ul{\gl}},\Z/p\Z)^{\mcl{D}_{\ul{\gl}}}$ for $t_0\in T_1$. Let $Pos(t_0)=(x,x+1),Pos(\gf(t_0))=(x+1,x)$ for $1\leq x<\gr=\gr_1+\cdots+\gr_k$.

Let \equa{S^x_0&=\{s=(l,i,j,m,n)\in S\mid l=0, \gch(Pos(s))>0,\text{ if }Pos(s)=(x',y')\text{ then }x'\leq x\},\\ \ti{S}^x_1&=\{s=(l,i,j,m,n)\in S\mid l=1,\gch(Pos(s))<0,\text{ if }Pos(s)=(x',y')\text{ then }y'\leq x\},\\ 
S^x_2&=\{s=(l,i,j,m,n)\in S\mid l=1, \gch(Pos(s))=0,\text{ if }Pos(s)=(x',x')\text{ then }x'\leq x\}.}
Let $\gf:\ti{S}_1^x\lra S_0^x$ be defined as $\gf(1,i,j,m,n)=(0,j,i,n,m)$. Note that the map $\gf$ defined on $T_1$ agrees with the map $\gf$ defined on $\ti{S}^x_1$ on the set $\ti{S}_1^x\cap T_1$. Now it may happen that for $(0,j,i,n,m)\in S^x_0$, the element $(1,i,j,m,n)$ need not be in $\ti{S}^x_1$. So we enlarge the set $\ti{S}^x_1$ to $S^x_1$ so that we include that all such elements. Now the extended map $\gf:S^x_1\lra S^x_0, \gf(1,i,j,m,n)=(0,j,i,n,m)$ is a bijection. Define $a_t=0=b_t$ for all $t\in S^x_1\bs \ti{S}^x_1$. 

We observe the following aspect of matrix multiplication.
\equan{InflationCob}{\us{s\in S^x_2}{\sum}v_s(g^a,g^b)&\equiv \us{s\in \ti{S}_1^x}{\sum}\big(a_sb_{\gf(s)}+a_{\gf(s)}b_s\big)+a_{t_0}b_{\gf(t_0)}\\ &\equiv \us{s\in S_1^x}{\sum}\big(a_sb_{\gf(s)}+a_{\gf(s)}b_s\big)+a_{t_0}b_{\gf(t_0)}\mod p}
where $v_s$ is the standard coboundary on $\mcl{P}_{\ul{\gl}}$ as in Definition~\ref{defn:StandardCoboundary}.
We also observe that for $g^c=g^ag^b,s\in S^x_1,Pos(s)=(x',y')$
\equ{c_s\equiv a_s+b_s+\us{\os{t_1 \in S^x_1,t_2\in S^x_0}{Pos(t_1)=(z',y'),Pos(t_2)=(x',z')}}{\sum}b_{t_1}a_{t_2}+\us{\os{t_1\in S^x_1,t_2 \in S^x_0,}{Pos(t_1)=(x',z'),Pos(t_2)=(z',y')}}{\sum}a_{t_1}b_{t_2}\mod p.}
Also for $g^c=g^ag^b,s\in S^x_0,Pos(s)=(x',y')$
\equ{c_s\equiv a_s+b_s+\us{\os{t_1,t_2\in S^x_0,}{Pos(t_1)=(x',z'),Pos(t_2)=(z',y')}}{\sum}a_{t_1}b_{t_2} \mod p.}
Note that if $s\in S^x_1\bs \ti{S}^x_1$ then $c_s=0$ since $g^c=g^ag^b\in \mcl{P}_{\ul{\gl}}$.
The entries $a_s,b_s,c_s$ with $s\in S^x_1\cup S^x_0$ gives an instance for the application of Theorem~\ref{theorem:Identity}.
Hence we conclude that 
\equ{\us{s\in S_1^x}{\sum}\big(a_sb_{\gf(s)}+a_{\gf(s)}b_s\big)} is a coboundary since each of the following summands in the identity given in Theorem~\ref{theorem:Identity}
\equ{v_{t_1,t_2,\cdots,t_i}(g^a,g^b)=c_{t_1}c_{t_2}\cdots c_{t_i}-a_{t_1}a_{t_2}\cdots a_{t_i}-b_{t_1}b_{t_2}\cdots b_{t_i}} is a $2$-coboundary on $\mcl{P}_{\ul{\gl}}$ where the positions of $t_j,1\leq j\leq i$ are either strictly upper triangular or strictly lower triangular.
Moreover we have $Pos(t_1)=(x_1',x_2')$, $Pos(t_2)=(x_2',x_3'),\cdots,Pos(t_i)=(x_i',x_1')$ for some $1\leq x_1',x_2',\cdots,x_i'\leq \gr$. Hence the above coboundary is a $\mcl{D}_{\ul{\gl}}$ invariant coboundary. 

Using Equation~\ref{Eq:InflationCob} we get that $a_{t_0}b_{\gf(t_0)}$ is a coboundary. But now we have
\equ{a_{t_0}b_{\gf(t_0)}+a_{\gf(t_0)}b_{t_0}=(a_{t_0}+b_{t_0})(a_{\gf(t_0)}+b_{\gf(t_0)})-a_{t_0}a_{\gf(t_0)}-b_{t_0}b_{\gf(t_0)}.}

So $u_{t_0}(g^a,g^b)=a_{\gf(t_0)}b_{t_0}$ is a $2$-coboundary which is what we wanted to prove. This proves the theorem.
\end{proof}
\begin{remark}
	In the proof of Theorem~\ref{theorem:InflationFromAbelinization}, $p$ can be any prime. 
\end{remark}
\subsubsection{\bf{The Cohomology Vanishing Theorem}}
\begin{theorem}
	\label{theorem:VanishingH2}
	Let $\ugl$ be a partition such that $\gl_i=k-i+1$. Let $\grpp$ be the finite abelian $p\operatorname{-}$group associated to $\ul{\gl}$, where $p\neq 3$ is an odd prime.
Let $\autgp=\Aut(\grpp)$ be its automorphism group. Then $H^2_{Trivial\ Action}(\autgp,\Z/p\Z)=0$. As a consequence $H^2_{Trivial\ Action}(\autgp,\grpp)=0$.
\end{theorem}
\begin{proof}
To prove $H^2_{Trivial\ Action}(\autgp,\Z/p\Z)=0$ it is enough to prove that $H^2_{Trivial\ Action}$ $(\mcl{P}_{\ul{\gl}}$, $\Z/p\Z)^{\mcl{D}_{\ul{\gl}}}=0$. For this we use the normal series $\mcl{N}^s_{\ul{\gl}},s\in S$ and use repeatedly the extended Hochschild-Serre exact sequence for the central extensions given in exact sequence~\ref{Eq:ModifiedCentralExtension} for $s,s'\in S\bs T$ with $s'$ the successor element of $s$ in $S\bs T$. 

If $[x]\in H^2_{Trivial\ Action}(\frac{\mcl{P}_{\ul{\gl}}}{\mcl{N}^s_{\ul{\gl}}},\Z/p\Z)^{\mcl{D}_{\ul{\gl}}}$ then $[x]\in \Ker(Res:H^2_{Trivial\ Action}(\frac{\mcl{P}_{\ul{\gl}}}{\mcl{N}^s_{\ul{\gl}}},\Z/p\Z)$ $\lra H^2_{Trivial\ Action}(\frac{\mcl{N}^{s'}_{\ul{\gl}}}{\mcl{N}^s_{\ul{\gl}}},\Z/p\Z))$ and one of the following $(1),(2)$ occurs, using Theorems~\ref{theorem:KernelRes},~\ref{theorem:KernelThetaOne},~\ref{theorem:KernelThetaTwo}.
\begin{enumerate}
\item $[x] \in \Ker(\gth:H^2_{Trivial\ Action}(\frac{\mcl{P}_{\ul{\gl}}}{\mcl{N}^s_{\ul{\gl}}},\Z/p\Z)\lra P(\frac{\mcl{P}_{\ul{\gl}}}{\mcl{N}^s_{\ul{\gl}}},\frac{\mcl{N}^{s'}_{\ul{\gl}}}{\mcl{N}^s_{\ul{\gl}}},\Z/p\Z))$.
\item There exists a $\mcl{D}_{\ul{\gl}}$ invariant cohomology class $[v]\in H^2_{Trivial\ Action}(\frac{\mcl{P}_{\ul{\gl}}}{\mcl{N}^s_{\ul{\gl}}},\Z/p\Z)^{\mcl{D}_{\ul{\gl}}}$ such that $[x]-[v]\in \Ker(\gth:H^2_{Trivial\ Action}(\frac{\mcl{P}_{\ul{\gl}}}{\mcl{N}^s_{\ul{\gl}}},\Z/p\Z)\lra P(\frac{\mcl{P}_{\ul{\gl}}}{\mcl{N}^s_{\ul{\gl}}},\frac{\mcl{N}^{s'}_{\ul{\gl}}}{\mcl{N}^s_{\ul{\gl}}},\Z/p\Z))$ and $Inf([v])=0$ for the map $Inf:H^2_{Trivial\ Action}(\frac{\mcl{P}_{\ul{\gl}}}{\mcl{N}^s_{\ul{\gl}}},\Z/p\Z) \lra H^2_{Trivial\ Action}$ $(\mcl{P}_{\ul{\gl}},\Z/p\Z)$.
\end{enumerate} 
By exactness of the extended Hochschild-Serre exact sequence and using Theorem~\ref{theorem:Averaging} we can get a $\mcl{D}_{\ul{\gl}}$ invariant cohomology class $[y]\in H^2_{Trivial\ Action}(\frac{\mcl{P}_{\ul{\gl}}}{\mcl{N}^{s'}_{\ul{\gl}}},\Z/p\Z)^{\mcl{D}_{\ul{\gl}}}$ such that either 
$Inf([y])=[x]$ in case $(1)$ or $Inf([y])=[x]-[v]$ in case $(2)$ where $Inf:H^2_{Trivial\ Action}(\frac{\mcl{P}_{\ul{\gl}}}{\mcl{N}^{s'}_{\ul{\gl}}},\Z/p\Z)\lra H^2_{Trivial\ Action}(\frac{\mcl{P}_{\ul{\gl}}}{\mcl{N}^{s}_{\ul{\gl}}},\Z/p\Z)$.

We start with $[x]\in H^2_{Trivial\ Action}(\mcl{P}_{\ul{\gl}},\Z/p\Z)^{\mcl{D}_{\ul{\gl}}}$ and by continuing this process repeatedly we get that there exists $[z]\in H^2_{Trivial\ Action}(\frac{\mcl{P}_{\ul{\gl}}}{[\mcl{P}_{\ul{\gl}},\mcl{P}_{\ul{\gl}}]},\Z/p\Z)^{\mcl{D}_{\ul{\gl}}}$ such that we have 
$Inf([z])=[x]$ where $Inf:H^2_{Trivial\ Action}(\frac{\mcl{P}_{\ul{\gl}}}{[\mcl{P}_{\ul{\gl}},\mcl{P}_{\ul{\gl}}]},\Z/p\Z) \lra H^2_{Trivial\ Action}(\mcl{P}_{\ul{\gl}},\Z/p\Z)$. Now using Theorem~\ref{theorem:InflationFromAbelinization} we conclude that $[x]=Inf([z])=0$. 

Hence $H^2_{Trivial\ Action}(\mcl{P}_{\ul{\gl}},\Z/p\Z)^{\mcl{D}_{\ul{\gl}}}=0\Ra H^2_{Trivial\ Action}(\autgp,\Z/p\Z)=0$.
Now using Theorem~\ref{theorem:SecondCohoTriv} we conclude that $H^2_{Trivial\ Action}(\autgp,\grpp)=0$. This proves the theorem.
\end{proof}
\subsubsection{\bf{The Vanishing/Nonvanishing Criterion for $H^2_{Trivial\ Action}(\autgp,\grpp)$ for Odd Primes $p\neq 3$}}
As a consequence of Theorems~\ref{theorem:NonvanishingH2},~\ref{theorem:VanishingH2} we have proved the following theorem for odd primes $p\neq 3$.
\begin{theorem}
	\label{theorem:CriterionH2TrivialAction}
	Let $\ugl$ be a partition. Let $\grpp$ be the finite abelian $p$-group associated to $\ul{\gl}$, where $p\neq 3$ is an odd prime. Let $\autgp=\Aut(\grpp)$ be its automorphism group. Then $H^2_{Trivial\ Action}(\autgp,\grpp)=0$ if and only if the difference between two successive parts of $\ul{\gl}$ is at most $1$.
\end{theorem}

\end{document}